\PassOptionsToPackage{dvips}{graphicx}
\PassOptionsToPackage{dvips}{xcolor}
\documentclass[12pt]{amsart}

\usepackage{graphicx}
\usepackage{amsmath,amsfonts,amsthm,amscd,amssymb,mathrsfs,bm,pb-diagram,here,extsizes,midpage,enumitem,mathtools}
\usepackage{setspace}
\usepackage[all,cmtip]{xy}

\usepackage{tikz}
\usepackage{tikz-cd}
\usepackage{blkarray}
\usepackage{multirow}

\usepackage{mathrsfs, arydshln}
\usepackage{graphicx}
\usepackage[all]{xy}
\newtheorem{theorem}{Theorem}

\newtheorem{proposition}[theorem]{Proposition}
\newtheorem{lemma}[theorem]{Lemma}
\newtheorem{definition}[theorem]{Definition}

\newtheorem{corollary}[theorem]{Corollary}

\newtheorem{remark}[theorem]{Remark}
\newcommand{\aaa}{\alpha}

\newcommand{\ccc}{\gamma}
\newcommand{\CCC}{\Gamma}
\newcommand{\ddd}{\delta}
\newcommand{\DDD}{\Delta}

\newcommand{\id}{{\rm{id}}}
\newcommand{\lmd}{\lambda}
\newcommand{\Lmd}{\Lambda}

\newcommand{\PP}{\mathbb{P}}
\newcommand{\CC}{\mathbb{C}}

\newcommand{\RR}{\mathbb{R}}
\newcommand{\ZZ}{\mathbb{Z}}

\newcommand{\FF}{\mathbb{F}}
\newcommand{\RP}{\mathbb{RP}}

\newcommand{\U}{{\rm{U}}}
\newcommand{\SU}{{\rm{SU}}}

\newcommand{\qdr}{\PP^1\times\PP^1}

\newcommand{\upone}{^{(1)}}
\newcommand{\uptwo}{^{(2)}}

\newcommand{\upm}{^{(m)}}
\newcommand{\upnn}{^{(n)}}

\renewcommand{\hat}{\widehat}
\renewcommand{\tilde}{\widetilde}

\newcommand{\ul}{\underline}

\newcommand{\mf}{\mathfrak}

\newcommand{\ol}{\overline}

\newcommand{\lras}{\,\longrightarrow\,}

\newcommand{\set}{\,|\,}

\newcommand{\proofend}{\hfill$\square$}

\newcommand{\inv}{^{-1}}

\newcommand{\ms}{\mathscr}
\newcommand{\minus}{\backslash}
\newcommand{\ptl}{\partial}

\newcommand{\qandq}{\quad{\text{and}}\quad}

\newcommand{\pr}{{\rm{pr}}}

\newcommand{\us}{^{\sigma}}
\newcommand{\uc}{^{\circ}}

\DeclareFontFamily{U}{mathx}{}
\DeclareFontShape{U}{mathx}{m}{n}{<-> mathx10}{}
\DeclareSymbolFont{mathx}{U}{mathx}{m}{n}
\DeclareMathAccent{\widehat}{0}{mathx}{"70}
\DeclareMathAccent{\widecheck}{0}{mathx}{"71}

\numberwithin{equation}{section}
\numberwithin{theorem}{section}

\begin{document}
\bibliographystyle{alpha} 
\title[]
{On the twistor spaces of ALE gravitational instantons
of type $A_{\rm odd}$}
\author{Nobuhiro Honda}
\address{Department of Mathematics, Institute
of Science Tokyo, 2-12-1, O-okayama, Meguro, 152-8551, JAPAN}
\email{honda@math.titech.ac.jp}

\thanks{The author was partially supported by JSPS KAKENHI Grant 22K03308.
}
\begin{abstract}
We study the twistor spaces of toric ALE gravitational instantons of type $A_{2n-1}$ and the associated non-standard minitwistor spaces introduced by Hitchin.
By analyzing the base locus of the linear system that induces the quotient meromorphic map from the compactified twistor space, we explicitly determine the images of certain distinguished twistor lines as hyperplane sections of the minitwistor space.
Using this family of special minitwistor lines as boundary data, we describe the $3$-dimensional family of real minitwistor lines arising from the instanton.
The central sphere in the gravitational instanton appears naturally throughout the analysis.\end{abstract}

\maketitle

\section{Introduction}
Let $\Gamma \subset \mathrm{SU}(2)$ be a finite cyclic subgroup, and let
$M$ be the minimal resolution of the quotient singularity
$\mathbb{C}^2/\Gamma$ at the origin.
Eguchi--Hanson \cite{EH78} and Gibbons--Hawking \cite{GH78} constructed complete
hyperk\"ahler metrics on $M$ that are asymptotic to the flat metric with Euclidean
volume growth, thereby providing the first examples of ALE gravitational instantons.
For instantons of type $A_k$, Hitchin \cite{Hi79} gave an explicit algebraic
construction of the associated twistor space.
This point of view was further developed by Kronheimer \cite{Kr89_1,Kr89_2},
who constructed and classified ALE hyperk\"ahler metrics on minimal resolutions of
$\mathbb{C}^2/\Gamma$ for arbitrary finite subgroups $\Gamma \subset \mathrm{SU}(2)$.

These gravitational instantons admit an $S^1$-action induced by scalar multiplication
on $\mathbb{C}^2$.
Throughout this paper, we refer to it as the {\em scalar $S^1$-action}.
This action preserves each component of the exceptional divisor of the resolution.
Except for the $A_{\rm even}$-type, there exists a unique component that is fixed pointwise by this action.
In \cite{Hi21}, Hitchin called this component the {\em central sphere} and
determined explicitly the restriction of the metric to it by using the simultaneous
resolution of the algebraic model of the twistor space.
Moreover, in \cite{Hi25}, he introduced a compactification of the twistor space and
obtained a compact complex surface that may be regarded as the quotient of the
compactified twistor space by a $\mathbb{C}^*$-action, namely the complexification of
the scalar $S^1$-action.
Hitchin further identified the linear system on this quotient surface to which the
image of a generic twistor line belongs, and showed that this image has nodes whose
number is determined uniquely by $\Gamma$.
In particular, the quotient surface is a minitwistor space in the sense of \cite{HN11}.

As noted above, the $A_k$ gravitational instanton admits a central sphere only when
$k$ is odd.
In this case, the element $-1 \in S^1$ acts trivially on the instanton, and we
therefore consider the induced effective action of $S^1/\{\pm1\} \simeq S^1$.
Besides the scalar action, the instanton admits another $S^1$-action preserving the
complex structures $I$, $J$, and $K$ associated with the hyperk\"ahler structure.
Consequently, the instanton is toric; we call this second action the
{\em tri-holomorphic $S^1$-action}.
The main objects of this paper are the twistor spaces and the associated minitwistor spaces
arising from toric ALE gravitational instantons of type $A_{2n-1}$ together with the scalar
$S^1$-action.

Throughout the paper, we denote by $Z$ the twistor space of a toric $A_{2n-1}$
ALE gravitational instanton.
In this case, the compactification of the twistor space is obtained from $Z$ by adding
three divisors (see Section~\ref{s:2} for details), and we denote it by $\tilde Z$.
The holomorphic map $Z \lras \PP^1$ associated with the hyperk\"ahler structure extends to
$\tilde Z$.
The twistor space $Z$ admits two $\CC^*$-actions, obtained by complexifying the scalar
$S^1$-action and the tri-holomorphic $S^1$-action, and these actions extend to $\tilde Z$.
We denote by $\CC^*_s$ (resp.\,$\CC^*_t$) the $\CC^*$-group obtained as the complexification of
the scalar (resp.\ tri-holomorphic) $S^1$-action.
The $\CC^*_s$-action on $Z$ covers the standard $\CC^*$-action on $\PP^1$.

Let $0$ and $\infty$ be the two fixed points of $\PP^1$.
The fibers over these points are biholomorphic to the minimal resolution of $\PP^2/\Gamma$
and to its complex conjugate, respectively, while all other fibers are mutually biholomorphic
via the $\CC^*_s$-action.
Each of these fibers is a smooth rational surface, namely the minitwistor space obtained in
\cite{Hi25}, and we denote it by $\tilde{\ms T}$.
The surface $\tilde{\ms T}$ contains two mutually disjoint $(-n)$-curves arising from the
compactification, and contracting them yields a surface $\ms T$ with two $A_{n,1}$ singularities.
By \cite[Proposition 6.1]{H25}, this surface is biholomorphic to the minitwistor space arising
from a hyperelliptic curve of genus $n-1$, including the real structure.
The surface $\ms T$ admits a natural embedding into $\PP^{n+2}$ induced by the complete linear
system generated by minitwistor lines, and in this paper we regard $\ms T$ as the minitwistor
space rather than $\tilde{\ms T}$.
Both $\tilde{\ms T}$ and $\ms T$ carry a residual $\CC^*$-action induced by the $\CC^*_t$-action
on $\tilde Z$.

We now briefly explain the main steps and results.
Using the divisors added in the compactification, we first define a linear system on $\tilde Z$
whose members are all $\CC^*_s$-invariant.
We denote by $\tilde\Psi$ the meromorphic map associated with this linear system.
By explicitly giving generators of the linear system (Proposition~\ref{p:gen}), we show that it
has dimension $(n+2)$ and that the image of $\tilde\Psi:\tilde Z \lras\PP^{n+2}$ is
precisely the minitwistor space $\ms T$ (Proposition~\ref{p:ZT}).
Since the generic fiber is irreducible, $\tilde\Psi$ may be regarded as the quotient map by the
$\CC^*_s$-action on $\tilde Z$; we call $\tilde\Psi$ the {\em meromorphic quotient map}.

In what follows, we identify the minimal resolution of $\CC^2/\Gamma$ with the fiber of
$Z \lras \PP^1$ over the point $u=0 \in \PP^1$.
In particular, the exceptional curves of the resolution, including the central sphere, lie on
this fiber.
In Section~\ref{ss:itl}, we determine explicit equations for the twistor lines that meet the
exceptional curve of the resolution (Proposition~\ref{p:tl}).
We also determine their images in the standard minitwistor space $\ms O(2)$ associated with $Z$.
These results will be used repeatedly in the subsequent analysis.

A key geometric input is that a generic point of the twistor space $Z$ converges to a point on
the central sphere in the limit of the $\CC^*_s$-action \cite{Hi25}.
Accordingly, the central sphere appears as a base curve of the above linear system on $\tilde Z$,
and it should be regarded as the principal component of the base locus.
In fact, the base locus consists of the entire exceptional divisor of the minimal resolution
(Proposition~\ref{p:bs}).
We begin by blowing up $\tilde Z$ along a chain of smooth rational curves (and its conjugate),
which is a slight extension of the chain of exceptional curves; however, new base curves then
appear on the exceptional divisors of this blowup (Proposition~\ref{p:bs1}).
Although further blowups are required, we show that the base locus always consists of chains of
smooth rational curves arising inductively, and that the complete elimination of the base locus
can be obtained by a successive blowup procedure along such chains.
As a consequence, many exceptional divisors are inserted into the fibers over
$u=0$ and $\infty$, while no modification occurs over $\CC^*=\PP^1\setminus\{0,\infty\}$.
Among the resulting exceptional components, the one lying over the central sphere is irreducible
and biholomorphic to $\tilde{\ms T}$ (Proposition~\ref{p:El}), and this feature will play a
decisive role.

We denote by $\tilde Z\upnn$ the space obtained after the complete elimination of the base locus.
Since the blowups are performed along chains of rational curves, the space $\tilde Z\upnn$
acquires several nodes.
Moreover, the fibers over $u=0,\infty\in\PP^1$ become highly reducible, making $\tilde Z\upnn$
seem difficult to handle at first sight.
Nevertheless, by applying suitable bimeromorphic transformations to $\tilde Z\upnn$ using these
nodes, the geometry simplifies considerably: all exceptional divisors, except those lying over the
central sphere, can be successively contracted to curves.
The resulting space still contains the proper transforms of the original fibers of
$\tilde Z \lras \PP^1$ over $u=0$ and $\infty$, and these divisors can also be contracted to curves,
while the component arising from the central sphere remains.
Consequently, the final space has irreducible fibers over $u=0$ and $\infty$, which are precisely the images of the exceptional divisor of the central sphere and its conjugate.
Furthermore, these two fibers remain isomorphic to $\tilde{\ms T}$, so that all fibers are
isomorphic to $\tilde{\ms T}$.
The $\CC^*_s$-action (as well as the real structure) survives throughout these bimeromorphic
transformations, and it follows that the resulting space is the product manifold
$\tilde{\ms T}\times\PP^1$ (Proposition~\ref{p:prod}).
Thus, {\em the compactified toric $A_{2n-1}$ twistor space becomes globally trivial over $\PP^1$
after suitable bimeromorphic modifications.}

In Section~\ref{s:mtl}, using these modifications, we investigate the image $\tilde\Psi(L)\subset\ms T$,
where $L$ is any twistor line in $Z$ that meets the (extension of the) chain of exceptional curves.
If $L$ intersects the central sphere, then $L$ is $\CC^*_s$-invariant, and hence its naive image under
$\tilde\Psi$ would collapse to a point.
To extract the correct geometric information, we introduce the {\em meromorphic image} as follows:
we first take the inverse image of $L$ under the blowups appearing in the elimination of the base
locus, and then take its image in the usual sense under the subsequent blowdowns and the projection
to $\ms T$.
This yields a curve in $\ms T$, which should be regarded as the true image of $L$ under $\tilde\Psi$.
We carry out this procedure not only for twistor lines meeting the central sphere but also for those
meeting other components of the chain.
Consequently, we explicitly determine the meromorphic image $\tilde\Psi(L)$ for any twistor line $L$
meeting the (extension of the) chain; see Propositions \ref{p:a}, \ref{p:i1} and \ref{p:l}.
Most of these images are reducible curves and contain a single non-reduced component.
These curves are special minitwistor lines in $\ms T$, and in the terminology of \cite{H25}, these might be called irregular minitwistor lines.
We identify them as explicit hyperplane sections of $\ms T$ and show that they constitute a
continuous family.

The toric $A_{2n-1}$ gravitational instanton is determined by $2n$ points lying on a line in Euclidean
space $\RR^3$, which serve as multi-monopole centers.
Viewing these points as points on the real circle (the equator) in $\PP^1$, we obtain a hyperelliptic
curve $\Sigma$ branched at these $2n$ points.
Let $\Lambda$ denote this copy of $\PP^1$, and embed it into $\PP^n$ as a rational normal curve of degree
$n$.
Then $\Sigma$ can be naturally realized in the projective cone ${\rm C}(\Lambda)\subset\PP^{n+1}$ over
$\Lambda$.
The minitwistor space $\ms T$ is a double cover of ${\rm C}(\Lambda)$ branched along $\Sigma$.
Since $\Sigma$ does not meet the vertex of ${\rm C}(\Lambda)$, the surface $\ms T$ has two $A_{n,1}$
singularities, and the minimal resolution of $\ms T$, equipped with the induced real structure, is
exactly Hitchin's minitwistor space $\tilde{\ms T}$.

In Section~\ref{ss:he}, we choose a fundamental domain $\Sigma''\subset\Sigma$ for the group action
generated by the hyperelliptic involution and the real structure.
This is a quarter of $\Sigma$, viewed as a manifold with corners with $2n$ edges.
In Section~\ref{ss:2mtl}, we obtain a 2-dimensional family of minitwistor lines parameterized by
$\Sigma''$.
More precisely, we study hyperplane sections of ${\rm C}(\Lambda)$ that are tangent to the branch curve
$\Sigma$ at $(n-1)$ points; their inverse images in $\ms T$ give members of the family, and the nodes of the resulting minitwistor lines lie over the tangency points.
The boundary of $\Sigma''$ then naturally becomes a parameter space for the above special minitwistor
lines.
As in the Lorentzian EW case discussed in \cite{H25}, we use the Abel--Jacobi map of $\Sigma$ together
with the doubling map on its Jacobian to extend the family from the boundary to the interior of
$\Sigma''$ (Proposition~\ref{p:ext}).

By the general result of Jones--Tod \cite{JT85} on the reduction of anti-self-dual structures to
Einstein--Weyl (EW) structures, the quotient of the present $A_{2n-1}$ gravitational instanton by the
scalar $S^1$-action carries an EW structure.
The projective surface $\ms T$ can be regarded as the minitwistor space of this EW space.
The EW space also admits an $S^1$-action induced by the tri-holomorphic $S^1$-action; we call it the
{\em residual $S^1$-action}.
Geometrically, this action may be viewed as a rotation around an axis in the EW space.
The axis is formed by the images of the exceptional divisors of the minimal resolution of
$\CC^2/\Gamma$ together with the two components arising from the coordinate axes in $\CC^2$, while the
central sphere has to be removed to obtain the axis since it is mapped to the conformal infinity of the EW space.

The EW space has orbifold singularities along most of the axis of rotation, since the scalar $S^1$-action
on the original components has a non-trivial finite stabilizer subgroup, except for the two components
adjacent to the central one.
We will confirm that the order of the stabilizer subgroup coincides with the multiplicity of the
non-reduced component of the meromorphic image $\tilde\Psi(L)$ obtained in Section~\ref{s:mtl}, where $L$
is a twistor line meeting the component in question.
We also note that the images of the two components from the coordinate axes in $\CC^2$ meet at the point
at infinity; the union of their images forms a {\em single} segment of the axis of rotation, located in
the middle of the whole axis, and the image of the point at infinity lies in the interior of this middle
segment.
The orbifold order is highest along this segment.
These points are discussed at the beginning of Section~\ref{ss:cplt}.

The minitwistor lines parameterized by the quarter $\Sigma''$ will constitute a {\em slice} in the EW space
with respect to the residual $S^1$-action.
In Section~\ref{ss:cplt}, we move the quarter $\Sigma''$ by the residual $S^1$-action to obtain a 3-dimensional
family of minitwistor lines in $\ms T$, and prove that this family precisely parameterizes the images
of twistor lines in $Z$ (Theorem~\ref{t:main}).

Finally, the connected component of the space of real minitwistor lines in the present case is different from the components in the Lorentzian case considered in \cite{H25}.
This difference arises because the present minitwistor lines have no real circle, due to the definiteness of the conformal structure, whereas those in \cite{H25} possess a real circle coming from the indefiniteness.
In both cases, the set of nodes of a (real) minitwistor line is real as a whole.
However, whether each individual node is real is a nontrivial issue.
In the Lorentzian case treated in \cite{H25}, all nodes of real minitwistor lines are real.
In Section \ref{ss:sing}, we show that, in contrast, in the present (Riemannian) case, some minitwistor lines have only real nodes, whereas others have non-real nodes.
Thus, a transition occurs within the same family of minitwistor lines. Such a transition can occur only through collisions of nodes.
This situation differs from the Lorentzian case, where collision can never occur and all nodes of minitwistor lines are uniformly real.
We can understand this transition clearly in the case of a small genus as shown at the end of Section \ref{s:EW}, whereas the situation becomes considerably more complicated in a higher genus.

\section{The twistor spaces, their compactifications and  minitwistors}\label{s:2}	
First, we recall the construction of the twistor space associated to the ALE gravitational instanton of type $A_k$ for arbitrary $k$, which was given by Hitchin \cite{Hi79}.

Let $k>0$ be an integer and $\CCC\subset\SU(2)$ a cyclic subgroup of order $(k+1)$. A gravitational instanton of type $A_k$ discussed in this paper is the minimal resolution of $\CC^2/\CCC$ equipped with a hyperk\"ahler metric with Euclidean volume growth, where the complex structure of the resolution is one of the compatible complex structures in the hyperk\"ahler family.
The twistor space of this gravitational instanton is constructed as follows.
Consider a hypersurface in the total space of the vector bundle $\ms O(k+1)\oplus\ms O(k+1)\oplus\ms O(2)$ over $\PP^1$ defined by the equation
\begin{align}\label{tw1}
xy = (z-a_1u)(z-a_2u)\dots(z-a_{k+1}u),
\end{align}
where $u$ is an affine coordinate on $\PP^1$, $x,y$ are fiber coordinates on $\ms O(k+1)$, $z$ is a fiber coordinate on $\ms O(2)$, 
 and $a_1,\dots, a_{k+1}$ are distinct real numbers determined from the hyperk\"ahler metric.
We may assume $a_1<a_2<\dots<a_{k+1}$.
The hypersurface \eqref{tw1} has compound $A_k$-singularities over $u=0,\infty\in\PP^1$ and is smooth away from these singularities.
The twistor space $Z$ of the ALE gravitational instanton of type $A_k$ is obtained from this hypersurface by taking 
suitable simultaneous (small) resolutions of these singularities.
We call the hypersurface \eqref{tw1} the {\em projective model} of the twistor space.
The holomorphic map $p:Z\lras\PP^1$ associated to the hyperk\"ahler structure is given by the projection $(x,y,z,u)\longmapsto u$ on the projective model.
% in the above coordinates.
The real structure $\sigma$ on the twistor space $Z$ is given, on the projective model, as
\begin{align}\label{rsZ}
(x,y,z,u)\stackrel{\sigma}\longmapsto \Big( (-1)^{k+1}\frac{\ol y}{\ol u^{k+1}},
\frac{\ol x}{\ol u^{k+1}}, -\frac{\ol z}{\ol u^{2}}, 
-\frac1{\ol u}\Big).
\end{align}
%The lift of this to the resolution is the real structure of the twistor space $Z$.

The $S^1$-action on $\CC^2$ which preserves the standard hyperk\"ahler structure induces a holomorphic $S^1$-action on $Z$ preserving each fiber of $p$,
which is given, on the projective model, by 
\begin{align}\label{act-t}
(x,y,z,u)\stackrel{t}\longmapsto
\big(tx, t\inv y, z, u\big),\quad t\in \U(1)=S^1.
\end{align}
This is indeed a $\CC^*$-action on the twistor space by just allowing $t\in\CC^*$.
%We denote $\CC^*_t$ for $\CC^*$ that acts on $Z$ by \eqref{act-t}.
Similarly, the scalar multiplication on $\CC^2$ induces an $S^1$ or $\CC^*$-action on $Z$, which is given by 
\begin{align}\label{act-s1}
(x,y,z,u)\stackrel{s}\longmapsto
\big(s^{k+1}x, s^{k+1}y, s^2 z, s^2u\big),\quad s\in
S^1  \subset
\CC^*
\end{align}
on the projective model.
%This commutes with $\sigma$ iff $|t|=1$. 
From \eqref{act-t} and \eqref{act-s1}, the twistor space $Z$ admits a $T^2$-action and a $T^2_{\CC}=(\CC^*\times\CC^*)$-action as its complexification.
%, where $T^2=\{(s,t)\in\CC^*_s\times\CC^*_t\set |s|=|t|=1\}$.

From the ALE property of the metric, the gravitational instanton can be compactified as an anti-self-dual orbifold by adding a point at infinity.
Correspondingly, the twistor space $Z$ can be compactified by adding the twistor line over the point at infinity.
%Let $Z'$ be the compactification.
The projection $p:Z\lras\PP^1$ lifts as a meromorphic map from the compactification, which has the added twistor line as the indeterminacy locus \cite{Kr89_2}.
This indeterminacy is eliminated by blowing up along the line,
obtaining a holomorphic map from the blowup to $\PP^1$.
But the blowup still has singularities along two 
$\PP^1$ which are two $T^2$-invariant fibers of the projection from the exceptional divisor to the twistor line at infinity.
These singularities can be resolved by blowing up these fibers.
Let $\tilde Z$ be the smooth compact complex threefold obtained this way, and $\tilde p:\tilde Z\lras \PP^1$ be the holomorphic map naturally induced from $p:Z\lras\PP^1$.
Each fiber of $\tilde p$ is identified with the compactification of a fiber of $p$ by the twistor line at infinity and two $\PP^1$ from the second blowup \cite[Figure 1]{Hi25}.
The real structure and the two $\CC^*$-actions \eqref{act-t} and \eqref{act-s1} extend to $\tilde Z$.
By the $\CC^*$-action \eqref{act-s1}, all fibers of $\tilde p$ are mutually biholomorphic except for the fibers over the two points $u=0,\infty\in\PP^1$.

The projection from the projective model of $Z$ to the $\ms O(2)$-factor
%, which is $(x,y,z,u)\longmapsto (z,u)$, 
induces a holomorphic map $\Phi:Z\lras\ms O(2)$.
This is the dimensional reduction \cite{JT85} of the twistor space to the minitwistor space $\ms O(2)$ of 
the Euclidean space $\RR^3$ under the $\CC^*$-action \eqref{act-t}.
The map $\Phi$ extends to a holomorphic map $\tilde\Phi:\tilde Z\lras\FF_2$, where $\FF_2$ is the Hirzebruch surface of degree two, which is a compactification of $\ms O(2)$ by a $(-2)$-curve attached at infinity.
This can still be regarded as a quotient map of $\tilde Z$ by the $\CC^*$-action \eqref{act-t}.
Throughout this paper, we denote $\bm E$ for the divisor in $\tilde Z$ which is the strict transform of the exceptional divisor of the (first) blowup along the twistor line at infinity. This is real (i.e.,\,$\sigma$-invariant) and biholomorphic to $\PP^1\times\PP^1$.
% and one of the projections from $\bm E$ to $\PP^1$ is identified with the restriction $p|_{\bm E}:\bm E\lras\PP^1$.
In terms of the above projection $\tilde\Phi$, if $\Xi_{\infty}$ denotes the $(-2)$-curve added at infinity to $\ms O(2)$, then 
\begin{align}\label{E}
\bm E = \tilde\Phi\inv(\Xi_{\infty}).
\end{align}
Also, we denote $\bm D$ and $\ol{\bm D}$ for the pair of the exceptional divisors of the second blowup in obtaining $\tilde Z$. 
(See Figure \ref{f:twistor} for the divisors $\bm E, \bm D$ and $\ol{\bm D}$.)
These are mutually disjoint sections of the above projection $\tilde\Phi:\tilde Z\lras\FF_2$, so they are biholomorphic to $\FF_2$.
We have $\tilde Z = Z\sqcup (\bm E\cup \bm D\cup \ol{\bm D})$.
This is the compactification given in \cite[Section 3]{Hi25}.
For a point $u\in\PP^1$, we denote by
$$\tilde Z_u:=\tilde p\inv(u)$$ for the fiber over $u$.

\begin{figure}
\includegraphics[height=3in]{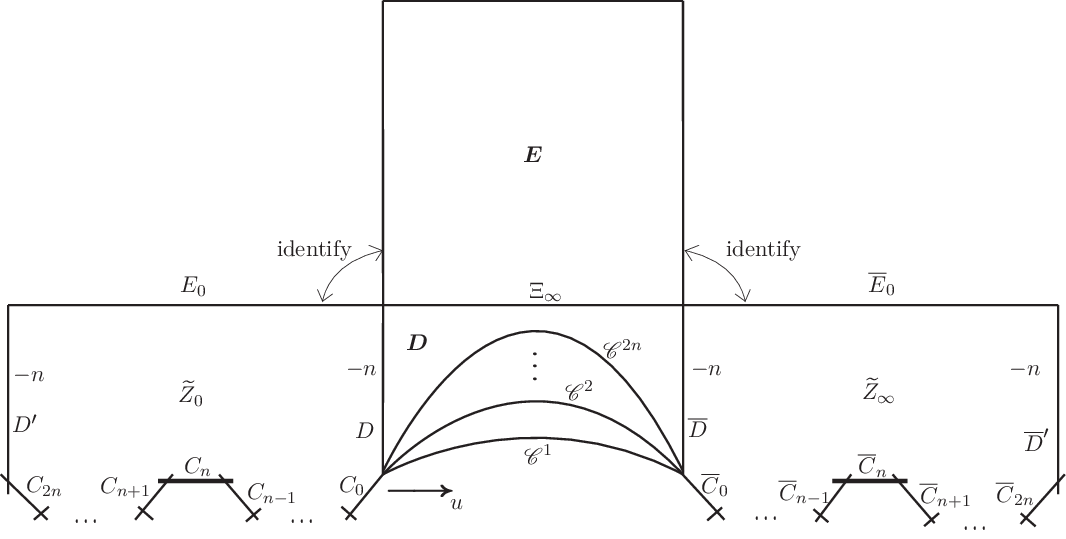}
\caption{The compactified twistor space $\tilde Z$}
\label{f:twistor}
\end{figure}

The two singularities of the fibers over $u=0,\infty$ in the projective model are the $A_{k}$-singularity, so the exceptional curve of the simultaneous resolution on $Z$ is a chain of $k$ smooth rational curves.
We denote 
    $C_1 + \dots  + C_{k}%\quad (l=g+1)
    $
for the chain over $u=0$ with the components being arranged in this order. Writing $\ol C_i = \sigma(C_i)$, 
the chain over $u=\infty$ may be written $\ol C_1 + \dots  + \ol C_{k}$.
This is the exceptional curve of the simultaneous resolution of the singularity over $u=\infty$.
By \cite[p.\,262]{Hi21}, under the correct choice of the small resolution, homogeneous coordinates on the component $C_i$ ($1\le i\le k$) are given by 
\begin{align}\label{nui}
\Big(x:\prod_{j=1}^{i} (z-a_ju)\Big).
\end{align}
%(See ** for detail.)
It follows that in a non-homogeneous coordinate, the $\CC^*$-action \eqref{act-s1} on the component $C_i$ is given by multiplication by $s^{k+1-2i}$.
Therefore, the $\CC^*$-action \eqref{act-s1} has a component which is pointwise fixed if and only if $k$ is odd. 
If $k$ is odd, putting 
\begin{align}\label{n}
k=2n-1,
\end{align}
the fixed component
is the middle one, which is $C_n$. This is called the {\em central sphere}
\cite{Hi21, Hi25}.
The existence of this component is crucial in \cite{Hi21,Hi25}
and also in this paper.
In the following, we always assume that $k$ is odd and use the letter $n$ to mean \eqref{n}.
%So if we define a $\CC^*$-subgroup $G_i$ by 
%\begin{align}\label{}
%G_i = \big\{ (s,t)\in \CC^*_s\times\CC^*_t\set s^{l-2i}t=1\big\},
%\end{align}
%then this is the stabilizer subgroup of $C_i$.
The $\CC^*$-action \eqref{act-s1} is non-effective if $k$ is odd, and re-defining $s$ as $s^2$, we obtain an effective $\CC^*$-action
\begin{align}\label{act-s}
(x,y,z,u)\stackrel{s}\longmapsto
\big(s^n x, s^ny,s z, su\big),\quad s\in\CC^*.
\end{align}
In the sequel, we denote $\CC^*_s$ (resp.\,$\CC^*_t$) to mean the group $\CC^*$ in \eqref{act-s} 
(resp.\,\eqref{act-t}) and 
$T^2_{\CC}$ to mean $\CC^*_s\times\CC^*_t$.
$Z$ has a $T^2_{\CC}$-action and it extends to the compactification $\tilde Z$.

In addition to the exceptional curves $C_1,\dots,C_{2n-1}$, we define $C_0$ and $C_{2n}$ to be smooth rational curves in $\tilde Z$ which are obtained from 
the curves $\{y=z=u=0\}$ and $\{x=z=u=0\}$ in the projective model \eqref{tw1} of $Z$ through the compactification and the resolutions.
We call each of them a {\em coordinate axis} %and {\em $y$-axis} respectively, 
because they correspond to two coordinate axes on $\CC^2$ by the quotient map to $\CC^2/\Gamma$.
%correspond to $x$-axis and $y$-axis in $\CC^3$ in which the $A_{2l-1}$-singularity is realized.
%We write $C_0$ and $C_{2l}$ for the $x$-axis and $y$-axis (in $\tilde Z$) respectively.
The sum 
$$
C_0 + C_1 + \dots + C_{2n-1} + C_{2n}
$$ is also a chain of rational curves.
All these components are $T^2_{\CC}$-invariant, although one of the two fixed points on the coordinate axis $C_0$ does not belong to the twistor space $Z$ and the same for another coodinate axis $C_{2n}$.

From \eqref{act-s}, a generic fiber of $\tilde p:\tilde Z\lras \PP^1$ can be thought of as an orbit space of the $\CC^*_s$-action, and Hitchin \cite{Hi25} obtained a compact minitwistor space from $\tilde Z$ as a fiber $\tilde Z_1=\tilde p\inv(1)$. 
More concretely, 
consider the fiber of the original projection $p:Z\lras\PP^1$ over the point $u=1\in\PP^1$, which is an affine surface
\begin{align}\label{Hi1}
xy = (z-a_1)(z-a_2)\dots(z-a_{2n})
\end{align}
in $\CC^3$.
This surface can be naturally compactified by thinking $z$ as an affine coordinate on $\PP^1$ and $(x,y)$ as affine fiber coordinates on the $\PP^2$-bundle $\mathbb P(\ms O(n)\oplus\ms O(n)\oplus\ms O)\lras\PP^1$ this time.
Let $S$ be the compact projective surface obtained this way.
This is non-singular, identified with the fiber $\tilde Z_1$, and has a conic bundle structure over $\PP^1$ whose coordinate is $z$.
Introducing a fiber coordinate $w$ on the $\ms O$-factor, the divisor added in the compactification consists of the fiber conic over the point $z=\infty$, which is the intersection with the divisor $\bm E$, and
two sections $D_1:=\{y=w=0\}$ and $D'_1:=\{x=w=0\}$.
These two sections are $(-n)$-curves on $S$ (see Figure \ref{f:mt}).
In \cite{Hi25} these are written $D_1$ and $D_2$.
As a conic bundle $S\lras\PP^1$, reducible fibers are exactly over the $2n$ points $z=a_1,\dots, a_{2n}$.
The twistor line at infinity appears as a regular fiber conic
over $z=\infty$.
We make distinction between the divisors $\bm D$ and $\ol{\bm D}$ 
by the properties $D_1=\bm D\cap S$ and $D'_1=\ol{\bm D}\cap S$.

\begin{figure}
\includegraphics{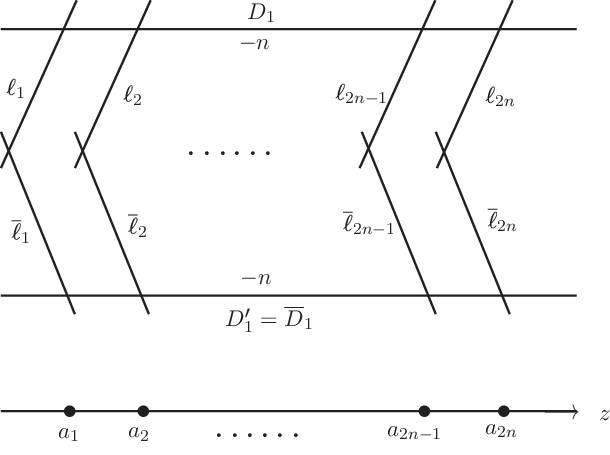}
\caption{The resolved minitwistor space $S=\tilde{\ms T}$}
\label{f:mt}
\end{figure}

While the affine surface \eqref{Hi1} or its compactification $S$ is non-real, they admit a natural real structure using the $\CC^*_s$-action.
Concretely, it is given by the composition of $\sigma$, which maps the fiber over $u=1$ to the fiber over $u=-1$, and the isomorphism between these fibers obtained by letting $s=-1$ in \eqref{act-s}.
In \cite{Hi25} this is written $\tau$.
In this paper, since we use the letter $\sigma$ to mean the real structure on the twistor space $Z$ or its compactification $\tilde Z$, we denote by $\sigma$ for this real structure on $S$.
This is explicitly given by 
\begin{align}\label{rs1}
(x,y,z)\stackrel{\sigma}\longmapsto \big((-1)^n\ol y,(-1)^n\ol x,\ol z\big).
\end{align}
Since the $(-n)$-curve $D'_1$ is equal to $\sigma(D_1)$, in the following, we mainly write $\ol D_1$ instead of $D'_1$.

For a later use, we provide another description of the minitwistor space $S$. 
Putting $f(z):=\prod_{i=1}^{2n}(z-a_i)$, by the variable change $x=v+iw, y=v-iw$, \eqref{Hi1} is transformed to $v^2+w^2=f(z)$, and in these new coordinates, the real structure \eqref{rs1} is given by $(v,w,z)\longmapsto ((-1)^n \ol v, (-1)^n \ol w, \ol z)$.
Rewriting the equation \eqref{Hi1} of $S$ to $w^2 = f(z)-v^2$,
define 
\begin{align}\label{Sgm}
\Sigma:=\big\{(v,z)\in\ms O_{\PP^1}(n)\set v^2 = f(z)\big\}.
\end{align}
This is a hyperelliptic curve branched at $a_1,\dots,a_{2n}$ and its genus $g$ is $(n-1)$.
Identifying $\ms O(n)$ with the cone ${\rm C}(\Lmd)$ over a rational normal curve $\Lmd\subset\PP^n$ minus the vertex, $\Sigma$ can be regarded as embedded in ${\rm C}(\Lmd)$.
From the equation $w^2 = f(z)-v^2$, if we compactify the affine surface \eqref{Hi1} in $\PP^{n+2}$ (instead of compactifying in the above $\PP^2$-bundle over $\PP^1$), we obtain a double covering of the cone ${\rm C}(\Lmd)$ branched at $\Sigma$.
We denote $\ms T$ for this double covering.
Since $\Sigma$ does not pass through the vertex of the cone, $\ms T$ has two cone singularities ($A_{n,1}$-singularities) over the vertex. 
The minimal resolution of $\ms T$ is biholomorphic to $S$, including the real structure \cite[Proposition 6.1]{H25}.
The exceptional curves of the resolution are the two $(-n)$-curves $D_1$ and $D'_1=\ol D_1$ on $S$. 

In summary, we have the following commutative diagrams of rational maps: 
\begin{equation}\label{cdm0}
\begin{tikzcd}
\PP^{n+2} \arrow[r,"\Pi"] \arrow[d] &
\PP^{n+1} \arrow[dl,"\pi"]\\
\PP^{n} &
\end{tikzcd}
\hspace{3mm}
\scalebox{1.5}{$\supset$}
\hspace{4mm}
\begin{tikzcd}
\ms T \arrow[r, "\Pi", "2:1"']
 \arrow[d] &
{\rm C}(\Lmd)\supset\Sigma \arrow[dl,"\pi"]\\
\Lmd &
\end{tikzcd}
\end{equation}
Each of $\Pi$ and $\pi$ in the left diagram is a projection from a point and the same notations in the right diagram are their restrictions to $\ms T$ and ${\rm C}(\Lmd)$ respectively.
The vertical map in the right diagram exhibits $\ms T$ as a rational conic bundle over $\Lmd\simeq\PP^1$.
In the following, we use the notation $\tilde{\ms T}$ for $S$ to indicate that it is the (minimal) resolution of $\ms T$.
The composition $\tilde{\ms T}\lras\ms T\lras\Lmd$ is holomorphic and may be identified with the above projection $S=\tilde{\ms T}\lras\PP^1$ which takes the $z$-coordinate.

We will also make use of the following realization of $\tilde{\ms T}$ as a birational change of $\qdr$. 
The reducible fibers of the last projection $\ms{\tilde T}\lras\Lmd$ are over the $2n$ points $a_1,\dots,a_{2n}$, and 
all of them consist of two $(-1)$-curves in $\tilde{\ms T}$.
For $1\le i\le 2n$, we denote $\ell_i$ and $\ol{\ell}_i$ for these $(-1)$-curves over $a_i$.
We make a distinction for these two components by the property that $\ell_i$ intersects the section $D_1$ (see Figure \ref{f:mt}.)
Explicitly, under the above definition of the curves $D_1$ and $D'_1$,
$\ell_i=\{y=z-a_i=0\}$ and $\ol{\ell}_i = \{x=z-a_i=0\}$ in the coordinate used in \eqref{Hi1}.
We use the same notation $\ell_i$ and $\ol\ell_i$ for the images of these $(-1)$-curves into $\ms T$ respectively.
These images are lines in $\ms T\subset\PP^{n+2}$ and they are all lines on $\ms T\subset\PP^{n+2}$.

\section{A linear system on the compactification}
In this section, we investigate a certain linear system on the compactified twistor space $\tilde Z$ which is useful to investigate properties of $Z$ and $\tilde Z$.
First, using the projection $\tilde p:\tilde Z\lras\PP^1$ and the divisor $\bm E$, we define a line bundle $F$ over $\tilde Z$ by
$$
F:=\tilde p^*\ms O_{\PP^1}(2) \otimes [\bm E].
$$
Using that $\bm E = \tilde\Phi\inv(\Xi_{\infty})$ as in \eqref{E},
if $\Xi_0$ denotes a $(+2)$-section of the projection $\FF_2\lras\PP^1$,
$F$ can be simply written as $\tilde\Phi^*\ms O_{\FF_2}(\Xi_0)$.
Note that $F\cdot L = 2$ for the intersection number with a twistor line $L\subset Z$ because $L$ does not intersect $\bm E$ and $L$ is a section of the projection $p: Z\lras\PP^1$.

\begin{proposition}\label{p:cb} For the anti-canonical bundle of $\tilde Z$, we have:
$$-K_{\tilde Z}\simeq 2F + \bm D + \ol{\bm D}.$$
\end{proposition}

\proof 
By \cite{HKLR}, the anti-canonical bundle $-K_Z$ of the twistor space $Z$ itself is given by $p^*\ms O_{\PP^1}(4)$.
So for the compactification $\tilde Z$, using the reality, we can write
$-K_{\tilde Z}\simeq \tilde p^*\ms O_{\PP^1}(4)+ l\bm E + m(\bm D + \ol{\bm D})$ for some $l,m\in\ZZ$.
On the other hand, by adjunction, 
$-K_{\tilde Z}|_{\tilde Z_1} \simeq -K_{\tilde Z_1}$
for the fiber $\tilde Z_1=\tilde p\inv(1)$.
From the structure of $\tilde Z_1$ that can be seen from the equation \eqref{Hi1}, 
it is easy to see that $-K_{\tilde Z_1} \simeq 
(2\bm E + \bm D + \ol{\bm D})|_{\tilde Z_1}$.
As $\tilde p^*\ms O_{\PP^1}(1)$ is trivial over the fiber $\tilde Z_1$, we obtain $l=2$ and $m=1$.
From the definition of the line bundle $F$, this gives the desired isomorphism.
\proofend

\medskip
In the rest of this section, we give the linear system on $\tilde Z$ that induces the quotient map onto the minitwistor space $\ms T$.
We recall that the boundary divisor that compactifies $Z$ consists of three components $\bm E,\bm D$ and $\ol{\bm D}$.
Recall also that we are considering the gravitational instanton of type $A_{2n-1}$.
Using the number $n$, we define a line bundle $\bm L$ on $\tilde Z$ by 
\begin{align*}
    \bm L&:= nF\otimes [\bm D + \ol {\bm D}]\\
    &\simeq  \tilde p^*\ms O_{\PP^1}(2n) \otimes [n\bm E + \bm D + \ol {\bm D}].
\end{align*}
If $L\subset Z$ is a twistor line, then 
since $F\cdot L = 2$ and $\bm E\cdot L = \bm D\cdot L = \ol{\bm D}\cdot L = 0$,
\begin{align}\label{LL}
\bm L\cdot L = 2n.
\end{align}
Since all line bundles and divisors used here are $T^2_{\CC}$-invariant,
the $T^2_{\CC}$-action on $\tilde Z$ lifts to the line bundle $\bm L$. 
Regarding $\CC^*_s$ as a subgroup of $T^2_{\CC}$, we denote by $|\bm L|^{\CC^*_s}$ the linear system associated to the subspace $H^0(\tilde Z,\bm L)^{\CC^*_s}$ of $H^0(\tilde Z,\bm L)$ that consists of $\CC^*_s$-invariant sections of $\bm L$, and similarly for $|\bm L|^{T^2_{\CC}}$.
In the rest of this section, we prove:

\begin{proposition}\label{p:ZT}
We have $\dim |\bm L|^{\CC^*_s} = n+2$ and the image of the associated meromorphic mapping $$\tilde{\Psi}:\tilde Z\lras \PP^{n+2}$$ is the mintwistor space $\ms T$,
and the restriction $\tilde{\Psi}|_{\tilde {\ms T}}$ of $\tilde{\Psi}$ to the fiber $\tilde {\ms T}=\tilde p\inv(1)$ is identified with the minimal resolution of $\ms T$.
\end{proposition}

To show this proposition,
we first define a divisor $G$ and a line bundle $\bm M$ on $\tilde Z$ by
\begin{align*}
G&:=\tilde Z_0+ \tilde Z_\infty+ \bm E+ \bm D+\ol {\bm D},\\
\bm M&:=\bm L\otimes [-G].
\end{align*}
The divisor $G$ is smooth normal crossing
and $T^2_{\CC}$-invariant (see Figure \ref{f:twistor}.)
In particular, the restriction exact sequence
\begin{align}\label{restG}
    0 \lras \ms O_{\tilde Z}(\bm M)\lras \ms O_{\tilde Z}(\bm L)\lras \ms O_{G}(\bm L)\lras 0
\end{align}
is $T^2_{\CC}$-equivariant.
Let $\Xi_0\subset\FF_2$ be a $(+2)$-section of the ruling $\FF_2\lras\PP^1$ which is $\CC^*_s$-invariant and $f$ a fiber class.
In terms of $\Xi_0$ and $g:=n-1$, since $\Xi_0\sim\Xi_{\infty}+2f$, we have linear equivalences
\begin{align}
    \bm M &\simeq g\bm E + \tilde p^*\ms O_{\PP^1}(2g)\notag\\
    &\simeq \tilde \Phi^*\ms O_{\FF_2}(g\Xi_{\infty}+ 2gf)\notag\\
    &\simeq \tilde \Phi^*\ms O_{\FF_2}(g\Xi_{0}).\label{M1}
\end{align}
Using these, we next show:

\begin{lemma} For any $q>0$, 
$
H^q(\tilde Z, \ms O_{\tilde Z}(\bm M)) =0
$
and there is a short exact sequence
\begin{align}\label{restG1}
0 \lras H^0\big(\ms O_{\tilde Z}(\bm M)\big)\lras 
H^0\big(\ms O_{\tilde Z}(\bm L)\big)\lras 
H^0\big(\ms O_G(\bm L)\big)
\lras 0,
\end{align}
which is $T^2_{\CC}$- (and hence $\CC^*_s$-) equivariant.
\end{lemma}

\proof
The exact sequence \eqref{restG1} follows from the cohomology exact sequence of \eqref{restG} and the former vanishing in the case $q=1$.
This sequence is $T^2_{\CC}$-equivariant because the sequence \eqref{restG} is $T^2_{\CC}$-equivariant.

For the former vanishing, 
since all fibers of $\tilde\Phi$ are either a smooth rational curve or a chain of smooth rational curves,
$H^q(\tilde \Phi\inv(y),\ms O_{\tilde \Phi\inv(y)})=0$ for all $q>0$ and all $y\in\FF_2$.
From \eqref{M1}, the projection formula and 
Grauert's theorem about direct image sheaves, this means 
\begin{align*}
R^q \tilde\Phi_*\ms O_{\tilde Z}(\bm M)
&\simeq 
R^q \tilde\Phi_*\big(\tilde\Phi^*\ms O_{\FF_2}(g\Xi_{0})\big)\\
&\simeq 
\begin{cases}
\ms O_{\FF_2}(g\Xi_{0}) & q=0,\\
0 & q>0.
\end{cases}
\end{align*}
From Leray spectral sequence, this implies isomorphisms
\begin{align}\label{M2}
    H^q\big(\tilde Z, \ms O_{\tilde Z}(\bm M)\big) \simeq 
    H^q\big(\FF_2,\ms O_{\FF_2}(g\Xi_{0})\big),
    \quad \forall q\ge 0.
\end{align}
As $-K_{\FF_2}\simeq 2\Xi_0$,
$g\Xi_0 - K_{\FF_2} \simeq (g+2)\Xi_0$ and this is nef and big. Hence, the RHS of \eqref{M2} vanishes when $q>0$
by Kawamata-Viehweg.
\proofend

\begin{lemma} \label{l:Md} 
$
\dim H^0\big(\tilde Z, \ms O_{\tilde Z}(\bm M)\big)
^{\CC^*_s} = g+1\,\, (=n).
$
\end{lemma}

\proof

\medskip
We use the isomorphism \eqref{M2} for $q=0$, which is $\CC^*_s$-equivariant
since the projection $\tilde\Phi:\tilde Z\lras\FF_2$ is $\CC^*_s$-equivariant.
Still letting $\Xi_0$ be the $\CC^*_s$-invariant section, for any $i\ge 0$,  we have the restriction sequence
\begin{align*}
0 \lras \ms O_{\FF_2}((i-1)\Xi_0)
\lras \ms O_{\FF_2}(i\Xi_0)
\lras \ms O_{\Xi_0}(i\Xi_0)
\lras 0,
\end{align*}
which is exact and $\CC^*_s$-equivariant.
Using Kawamata-Vieweg as above, we readily have $H^1(\ms O_{\FF_2}((i-1)\Xi_0))=0$ for any $i\ge0$.
As $\Xi_0^2 = 2$, $\ms O_{\Xi_0}(i\Xi_0)\simeq\ms O_{\PP^1}(2i)$ and since $\CC^*_s$ acts on $\Xi_0\simeq\PP^1$ in the standard way, among elements of $H^0(\ms O_{\PP^1}(2i))\simeq\CC^{2i+1}$, only constants are $\CC^*_s$-invariant.
From this, using the cohomology exact sequence of the above sequence
for all $0\le i\le g$, we obtain that
$%\begin{align*}
\dim H^0(\FF_2,\ms O_{\FF_2}(g\Xi_{0}))
^{\CC^*_s} = g+1.$
%\end{align*}
From the isomorphism \eqref{M2} for $q=0$, this means the desired conclusion.
\proofend

\medskip
From the lemma and the exact sequence \eqref{restG1} which is 
$\CC^*_s$-equivariant, to prove $\dim |\bm L|^{\CC^*_s} = n+2$ as in
Proposition \ref{p:ZT}, it is enough to show: 

\begin{lemma}\label{l:LG}
$H^0(\ms O_G(\bm L))^{\CC^*_s}\simeq\CC^3$.
\end{lemma}

\proof
First, we determine the restriction of $\bm L$ to 
the components of the divisor $G$.
Using that $\bm D$ intersects the fiber $\tilde Z_0=\tilde p\inv(0)$ transversely along a $(-n)$-curve on $\tilde Z_0$
and also that $\bm D$ intersects $\bm E$ transversely along a $(0,1)$-curve on $\bm E\simeq\Xi_{\infty}\times\PP^1$, we obtain an isomorphism
\begin{align}\label{DD}
\bm D|_{\bm D}\simeq\ms O_{\FF_2}(-n\Xi_0).
\end{align}
Further, we have $F\simeq\tilde\Phi^*\ms O_{\FF_2}(\Xi_0)$.
From these, using reality also, we readily obtain 
\begin{align}\label{LD}
\bm L|_{\bm D} &\simeq \bm L|_{\ol {\bm D}} \simeq \ms O_{\FF_2}.
\end{align}
Next, under the isomorphism $\bm E\simeq\Xi_{\infty}\times\PP^1$, the intersections $\bm E\cap \bm D$ and
$\bm E\cap \ol{\bm D}$ are $(0,1)$-curves and each of these intersections is transversal.
From this, again using $F\simeq\tilde\Phi^*\ms O_{\FF_2}(\Xi_0)$ and that $\Xi_0\cap \Xi_{\infty}=\emptyset$,  we obtain 
\begin{align}\label{ED}
\bm L|_{\bm E}\simeq \ms O_{\bm E}(0,2).
\end{align}
Next, if $E_0$ denotes the restriction $\bm E|_{\tilde Z_0}$ which is a regular fiber of $\tilde\Phi|_{\tilde Z_0}:\tilde Z_0\lras\PP^1$ and putting $D=\bm D|_{\tilde Z_0}$ and $D'=\ol{\bm D}|_{\tilde Z_0}$ for simplicity (see Figure \ref{f:twistor}), then we obtain 
\begin{align}\label{LZ}
\bm L|_{\tilde Z_0}  \simeq n E_0+ D + D'.
\end{align}

From \eqref{LD} and \eqref{ED}, we have
$$
H^0\big(\bm L|_{\bm D\cup\bm E\cup \ol{\bm D}}\big)\simeq
H^0\big(\bm L|_{\bm E}\big)
\simeq\CC^3.
$$
Further, since $\CC^*_s$ acts on $\bm E\simeq\Xi_{\infty}\times\PP^1$
as the product of the standard action on the $\Xi_{\infty}$-factor and the trivial action on $\PP^1$-factor,
from \eqref{ED}, the $\CC^*_s$-action on $H^0(\bm L|_{\bm E})$ is trivial. Therefore, so is the one on $H^0(\bm L|_{\bm D\cup\bm E\cup \ol{\bm D}})\simeq\CC^3$.

We claim that any section of the line bundle $\bm L$
defined over the union $\bm D\cup\bm E\cup \ol{\bm D}$, which is always  
$\CC^*_s$-invariant as above, extends
to a $\CC^*_s$-invariant section defined over the two divisors $\tilde Z_0$ and $\tilde Z_{\infty}$ in a unique way.
This implies that $H^0(\bm L|_G)^{\CC^*_s}\simeq H^0(\bm L|_{\bm D\cup\bm E\cup \ol{\bm D}})$, which means $H^0(\bm L|_G)^{\CC^*_s}\simeq\CC^3$ as required.
To prove the claim, it suffices to show that any section of 
$\bm L|_{E_0}\simeq\ms O_{E_0}(2)$, where $E_0=\bm E\cap \tilde Z_0$ as above,
 extends uniquely to a $\CC^*_s$-invariant section of $\bm L|_{\tilde Z_0}$.
For this, we use the fact that the closure of generic $\CC^*_s$-orbit lying on $\tilde Z_0$ transversally intersects the central sphere $C_n$ and also another fixed curve $E_0$.
Using that $\sum_{i=1}^{2n-1}C_i$ is the exceptional curve of $A_{2n-1}$-singularity, we readily see that 
the closure of generic $\CC^*_s$-orbit closure lying on $\tilde Z_0$ is linearly equivalent to the curves
\begin{align}\label{DD'}
D+ \sum_{i=1}^{n} i\, C_{n-i}
\qandq
D'+ \sum_{i=1}^{n} i\, C_{n+i}.
\end{align}
Since $\bm L|_{E_0}\simeq\ms O_{E_0}(2)$,
generic section of $\bm L|_{E_0}$ vanishes at two points, so it determines
two $\CC^*_s$-orbit closures on $\tilde Z_0$.
Taking the two curves \eqref{DD'} as representatives of these orbit closures and noting the linear equivalence ${E_0}\sim \sum_{i=0}^{2n}C_i$ on $\tilde Z_0$, for the residual class, using \eqref{LZ}, we have
\begin{align*}
\bm L|_{\tilde Z_0} - 
\left(
D+ \sum_{i=1}^{n} i\, C_{n-i}
+
D'+ \sum_{i=1}^{n} i\, C_{n+i}
\right)
&\simeq 
n{E_0} - \sum_{i=1}^{n} i\, C_{n-i} -  \sum_{i=1}^{n} i\, C_{n+i}\\
&\simeq
nC_n + \sum_{i=1}^{n} (n-i) (C_{n-i}+C_{n+i}).
\end{align*}
Using that $C_i^2 = -2$ for all $0< i < 2n$,
we can see that the linear equivalent class of the last curve consists of the curve itself.
Therefore, generic section of $\bm L|_{E_0}$ uniquely extends to the whole $\tilde Z_0$ and the extension is of the form
\begin{align}\label{orbit}
\DDD + \DDD' + nC_n + \sum_{i=1}^{n} (n-i) (C_{n-i}+C_{n+i})
\end{align}
for the orbit closures $\DDD$ and $\DDD'$ that intersect $E_0$ at
the zeroes of the prescribed section of $\bm L|_{E_0}$.
Hence we have obtained the claim, and therefore $H^0(\bm L)^{\CC^*_s}\simeq\CC^{n+3}$.
\proofend

\medskip
From the proof, we readily obtain 

\begin{lemma}\label{l:LZ}
$\dim |\bm L|_{\tilde Z_0}|^{\CC^*_s} = 2$, and if $D=\bm D\cap \tilde Z_0$ and $D'= \ol {\bm D}\cap \tilde Z_0$ as before (see Figure \ref{f:twistor}), then we can take the following three curves as generators of $|\bm L|_{\tilde Z_0}|^{\CC^*_s}$:
\begin{align}\label{LZ2}
2D + \sum_{i=0}^{2n-1}(2n-i)C_i,
\quad
2D' + \sum_{i=1}^{2n}iC_i,
\qandq
D+ D' + n\sum_{i=0}^{2n}C_i.
\end{align}
\end{lemma}

\proof
Write $0$ (resp.\,$\infty$) for the intersection point $D\cap {E_0}$ (resp.\,$D'\cap {E_0}$) and take $2\cdot 0, 2\cdot \infty$
and $0 + \infty$ as the zero divisors of generators of $H^0(\bm L|_{E_0})\simeq
H^0(\ms O_{\PP^1}(2)).$ 
From \eqref{orbit}, these give the three divisors
\eqref{LZ2} as generators of the linear system $|\bm L|_{\tilde Z_0}|^{\CC^*_s}$.
\proofend

\medskip
For the proof of Proposition \ref{p:ZT}, we next prove that 
the image of 
the meromorphic map $\tilde{\Psi}:\tilde Z\lras \PP^{n+2}$ induced by the linear system $|\bm L|^{\CC^*_s}$ is isomorphic to the minitwistor space $\ms T$.
For this purpose, we need explicit generators of the
linear system $|\bm L|^{\CC^*_s}\simeq\PP^{n+2}$.
We prepare some notation to obtain them.
First, for any $\lmd\in\CC$, we define a curve $\ms C_{\lmd}$ on the surface $\FF_2$ by
\begin{align}\label{Clmd}
\ms C_{\lmd}:=\big\{(z,u)\in\FF_2\set z=\lmd u\big\}.
\end{align}
All these are $(+2)$-sections, and 
from the $\CC^*_s$-action \eqref{act-s}, they are $\CC^*_s$-invariant.
When $\lmd=\infty$, letting $\Xi_{\infty}$ be the $(-2)$-section $\{z=\infty\}$ on $\FF_2$ as before, we define
\begin{align}\label{Cinf}
\ms C_{\infty} = \Xi_{\infty}+\{u=0\}+\{u=\infty\}.
\end{align}
This is also $\CC^*_s$-invariant and linearly equivalent to the curves $\ms C_{\lmd}$.
The set $\{\ms C_{\lmd}\set \lmd\in\CC\cup\{\infty\}\}$ is a pencil on $\FF_2$ and its base locus consists of two points
$(z,u) = (0,0)$ and $(z,u)=(0,\infty)$.
(See Figure \ref{f:twistor}.)

Pulling back this pencil by $\tilde\Phi:\tilde Z\lras\FF_2$, we obtain a pencil on $\tilde Z$ consisting of $T^2_{\CC}$-invariant divisors.
%Since all members of this pencil on $\tilde Z$ is $T^2_{\CC}$-invariant, 
We call it the {\em $T^2$-invariant pencil} on $\tilde Z$.
Since the line bundle $F$ on $\tilde Z$ can be written $[\tilde\Phi\inv(\ms C_{\lmd})]$, the $T^2$-invariant pencil is a subsystem of $|F|$.
For any $\lmd\in\CC\cup\{\infty\}$, we set 
\begin{align}\label{Slmd}
S_{\lmd}:=\tilde\Phi\inv(\ms C_{\lmd}).
\end{align}
Then $S_{\infty}=\tilde Z_0 + \tilde Z_{\infty} + \bm E.$
Of course, the $T^2$-invariant pencil on $\tilde Z$ is exactly
$\{S_{\lmd}\set \lmd\in\CC\cup\{\infty\}\}$.
All $S_{\lmd}$ are (not necessarily irreducible) toric surfaces by the $T^2_{\CC}$-action.

Recall that the projective model of the twistor space $Z$ is defined by the equation 
\begin{align}\label{tw3}
xy = (z-a_1u)(z-a_2u)\dots(z-a_{2n}u).
\end{align}
For each index $1\le i\le 2n$, we put 
$$
\ms C^i:=\big\{(z,u)\in\FF_2\set z=a_iu\big\}.
$$
(See Figure \ref{f:twistor}.)
These are distinguished members of the above pencil on $\FF_2$ in the sense that $\tilde\Phi\inv(\ms C^i)$ is reducible, consisting of two irreducible components defined by
\begin{align}\label{Spm}
\{y = z-a_iu = 0\}
\qandq
\{x = z-a_iu = 0\}.
\end{align}
In the following, we denote $S^+_i$ and $S^-_i$ for the divisors in $\tilde Z$ that correspond to these divisors in the projective model respectively.
This means $\ell_i = S_i^+\cap \tilde Z_1$ and $\ol\ell_i = S_i^-\cap \tilde Z_1$.
Further, from the distinction of the two boundary divisors
$\bm D$ and $\ol{\bm D}$, 
$\bm D\cap S^+_i$ and $\ol{\bm D}\cap S^-_i$ 
are non-empty so that $\bm D\cap S^-_i=\ol{\bm D}\cap S^+_i=\emptyset$.
Of course,  $\tilde\Phi\inv(\ms C^i) = S^+_i+S^-_i$, and $S^+_i$ and $S^-_i$ are irreducible toric surface by the $T^2_{\CC}$-action.
Further, from $F\cdot L =2$ for a twistor line $L\subset Z$, $S^+_i\cdot L = S^-_i\cdot L = 1$ for any index $i$.
We will confirm that the transformation of the curve
\begin{align}\label{Li0}
S_i^+\cap S_i^-=\big\{ x=y=z-a_iu = 0\big\},
\quad 1\le i\le 2n
\end{align}
into $Z$ is the twistor line through the point $C_{i-1}\cap C_{i}$
(Proposition \ref{p:tl}).

From the equation \eqref{tw3}, for any index $1\le i\le 2n$, the curve \eqref{Li0} passes through the two singularities of the projective model of $Z$. 
From \cite[p.\,262]{Hi21}, the small resolution of the singularity lying over the point $u=0$ is explicitly given as the graph of the rational map from the projective model to $({\PP^1})^{2n-1}$ ($(2n-1)$-fold product of $\PP^1)$ whose $i$-th factor is given by
\begin{align}\label{hc}
\Big(
x:\prod_{j=1}^i (z-a_ju)
\Big).
\end{align}
This means that if we denote $0:=(1:0)\in\PP^1$ and $\infty:=(0:1)\in\PP^1$, then the exceptional curve of the resolution is given by, in $({\PP^1})^{2n-1}$,  
\begin{align}\label{ec}
\underbrace{\infty \times \dots \times \infty}_{i-1} \times \underset{\substack{
\text{$i$-th}
}}{\mathbb{P}^1} \times \underbrace{0 \times \dots \times 0}_{2n-1-i},
\end{align}
and this is exactly the component $C_i$ of the chain.
Furthermore, under the above distinction between $S^+_i$ and $S^-_i$, 
among the components $C_0,C_1,\dots, C_{2n}$, 
\begin{align}\label{incl}
\setlength{\arraycolsep}{-10pt} 
\left\{
\begin{array}{@{}l@{}}
\text{$S^+_i$ includes $C_0,C_1,\dots, C_{i-1}$,} \\
\text{$S^-_i$ includes $C_{i},C_{i+1},\dots,C_{2n}$.} \\
\end{array}
\right.
\end{align}

With these preparations, we show
\begin{lemma}\label{l:X}
Let $\lmd$ be an arbitrary real number.
The following three divisors on $\tilde Z$ belong to 
the linear system $|\bm L|^{\CC^*_s}$ and 
restrict to 
the three curves on $\tilde Z_0$ given in Lemma \ref{l:LZ} respectively:
\begin{align}\label{gen2}
X:=\sum_{i=1}^{2n} S_i^+ + 2 \bm D,\quad
\ol X:=\sum_{i=1}^{2n} S_i^- + 2 \ol{\bm D},\quad
nS_{\lmd} + \bm D + \ol{\bm D}.
\end{align}
\end{lemma}

\proof
From the above includedness \eqref{incl}, noting that both $S^+_i\cap \tilde Z_0$ and $S^-_i\cap \tilde Z_0$ are included in the chain $\cup_{i=0}^{2n}C_i$, 
we readily have $\sum_{i=1}^{2n} S_i^+|_{\tilde Z_0} = \sum_{i=0}^{2n-1}(2n-i)C_i$.
As $\bm D|_{\tilde Z_0} = D$, this implies $X|_{\tilde Z_0} =
2D+ \sum_{i=0}^{2n-1}(2n-i)C_i$.
Since this is exactly a generator of $|\bm L|_{\tilde Z_0}|^{\CC^*_s}$ of Lemma \ref{l:LZ}, this means a linear equivalence 
$X\sim\bm L + \tilde p^*\ms O_{\PP^1}(d)$ for some $d\in\ZZ$.
But both $X\cdot L$ and $\bm L\cdot L$ are $2n$
as $S_i^+\cdot L = 1$, $F\cdot L = 2$,
$\bm D\cdot L = \ol{\bm D}\cdot L = 0$
and $\bm L\cdot L = 2n$ (see \eqref{LL}), it follows that $d=0$. 
So $X\in |\bm L|$. 
As $X$ is $\CC^*_s$-invariant, this implies $X\in |\bm L|^{\CC^*_s}$.
The property   $\ol X\in |\bm L|^{\CC^*_s}$ follows from 
this and the real structure.
The remaining property
$nS_{\lmd} + \bm D + \ol{\bm D}\in|\bm L|^{\CC^*_s}$ is obvious.
\proofend

\medskip
From Lemmas \ref{l:LZ} and \ref{l:X}, we readily obtain the following

\begin{proposition}\label{p:bs}
The base locus of the linear system $|\bm L|^{\CC^*_s}$ on $\tilde Z$ is exactly the exceptional curves of the simultaneous resolution of the projective model \eqref{tw1} of the twistor space $Z$.
\end{proposition}

\proof Since the linear system in question is $\CC^*_s$-invariant and real, so is its base locus.
Hence, the base locus of $|\bm L|^{\CC^*_s}$ is contained in $\tilde Z_0\cup\tilde Z_{\infty}$.
It is easy to see that the intersection of the three curves in Lemma \ref{l:LZ} is exactly the chain of the exceptional curves of the simultaneous resolution.
From these, the base locus of $|\bm L|^{\CC^*_s}$ is included in the union of the chains of the exceptional curves.
But the inclusion has to be an equality because Lemma \ref{l:X} means that the restriction homomorphism $H^0(\tilde Z,\bm L)^{\CC^*_s}
\lras H^0(\tilde Z_0,\bm L)^{\CC^*_s}$ is surjective.
\proofend 

\medskip
Next, we give explicit generators of 
the linear system $|\bm L|^{\CC^*_s}$.
Recall that $T^2_{\CC}$ denotes $\CC^*_s\times\CC^*_t$.
\begin{proposition}\label{p:gen}
(i) The linear system $|\bm L|^{T^2_{\CC}}$ is $n$-dimensional and generated by the following $(n+1)$ real divisors
\begin{align}\label{gen1}
    mS_{\lmd}
    + (n-m)S_{\infty} + \bm D + \ol {\bm D},\quad 0\le m\le n,
\end{align}
    where $\lmd$ is any fixed real number.
 (ii)    The linear system $|\bm L|^{\CC^*_s}$ is $(n+2)$-dimensional and generated by the $(n+1)$ divisors \eqref{gen1} and the two divisors $X$ and $\ol X$ of Lemma \ref{l:X}.
\end{proposition}

\proof Taking the $\CC^*_s$-fixed part of the sequence \eqref{restG1}, we have the exact sequence
\begin{align}\label{restG2}
0 \lras H^0\big(\ms O_{\tilde Z}(\bm M)\big)^{\CC^*_s}\lras 
H^0\big(\ms O_{\tilde Z}(\bm L)\big)^{\CC^*_s}\lras 
H^0\big(\ms O_G(\bm L)\big)^{\CC^*_s}
\lras 0.
\end{align}
Recalling that $\bm M\simeq \tilde \Phi^*\ms O_{\FF_2}((n-1)\Xi_{0})$ as in \eqref{M1} and $|\bm M|^{\CC^*_s}$ is $(n-1)$-dimensional as a linear system from Lemma \ref{l:Md}, as generators of 
$|\bm M|^{\CC^*_s}$, we can take the $n$ divisors
\begin{align}\label{gM}
m S_{\lmd} + (n-1-m) S_{\infty},
\quad 0\le m < n,
\end{align}
    where ${\lmd}$ is any fixed real number.    
    Adding the divisor $G = S_{\infty} + \bm D +\ol {\bm D} $ to these, we obtain the divisors \eqref{gen1} except $nS_{\lmd}+\bm D + \ol {\bm D}$ from the case $m=n$.
    Since the sequence \eqref{restG1} is $T^2_{\CC}$-equivariant and $|\bm M|^{\CC^*_s} = |\bm M|^{T^2_{\CC}}$ as $\tilde\Phi$ is $T^2_{\CC}$-equivariant, these $n$ divisors belong to $|\bm L|^{T^2_{\CC}}$.
Further, the divisor $nS_{\lmd}+\bm D + \ol {\bm D}$ also belongs to the same linear system and the $n+1$ divisors \eqref{gen1} are linearly independent.
Thus, we obtain (i).

For (ii), from (i) and the exact sequence \eqref{restG2}, $\dim |\bm L|^{\CC^*_s} = n+3$.
Adding $G$ to each of the divisors \eqref{gM}, we obtain $n$ elements of $|\bm L|^{\CC^*_s}$ that are linearly independent.
Further, we can readily see that the restrictions to $\tilde Z_0$ of each of the three divisors in Lemma \ref{l:X} are exactly the three generators of $|\bm L|_{\tilde Z_0}|^{\CC^*_s}$ given in Lemma \ref{l:LZ}.
Therefore, the restrictions have to be generators of $|\bm L|_G|^{\CC^*_s}$ that are linearly independent.
This is exactly the assertion about the generators of $|\bm L|^{\CC^*_s}$ as in the present proposition.
\proofend

\medskip
Using the two divisors $X$ and $\ol X$ in Lemma \ref{l:X}, we are ready to complete a proof of Proposition \ref{p:ZT}.

\medskip
\noindent
{\em Completion of proof of Proposition \ref{p:ZT}.}
As before, let ${\tilde \Psi}:\tilde Z\lras \PP^{n+2}$ be the meromorphic quotient map
induced by the linear system $|\bm L|^{\CC^*_s}$.
It remains to show that ${\tilde \Psi}(\tilde Z) = \ms T$.
The following proof is quite similar to \cite[Proposition 2.10]{Hon10}
that determines the equations of the minitwistor space obtained from 
an arbitrary Joyce metric \cite{J95}.

First, since the $(n+1)$ sections of $\bm L$ that define generators \eqref{gen1} satisfy the same relations as the $(n+1)$ monomials $V^n,V^{n-1}W, \dots, W^n$ of degree $n$, 
the image ${\tilde \Psi}(\tilde Z)$ lies on the scroll of planes over a rational normal curve in $\PP^n$.
If $\Lmd\subset\PP^n$ means this curve and
$\pr:\PP^{n+2}\lras\PP^n$ is the projection which is induced by the inclusion $H^0(\bm L)^{T^2_{\CC}}\subset H^0(\bm L)^{\CC^*_s}$
under the identifications
$\PP^n=\PP(H^0(\bm L)^{T^2_{\CC}})^*$
and $\PP^{n+2}=\PP(H^0(\bm L)^{\CC^*_s})^*$,
we have ${\tilde \Psi}(\tilde Z)\subset\pr\inv(\Lmd)$.

Let $\xi$ and $\ol\xi$ be sections of $\bm L$ that define the divisors $X$ and $\ol X$ respectively,
and $\ddd$ and $\ol\ddd$ be those that define $\bm D$ and $\ol{\bm D}$
respectively.
The product $\xi\ol\xi$ belongs to $H^0(2\bm L)$ and defines the divisor $X+\ol X$.
Fix any $\lmd\in\RR$ and let $v_1,v_2\in H^0(F)^{T^2_{\CC}}$ be elements such that $(v_1) = S_{\lmd}$ and $(v_2) = S_{\infty}$.
These are basis of $H^0(F)^{T^2_{\CC}}$.
Then from Proposition \ref{p:gen} (i), if we put $z_m := v_1^mv_2^{n-m}\delta\ol\delta$ for $0\le m\le n$, then $z_0,\dots, z_n$ are 
basis of $H^0(\bm L)^{T^2_{\CC}}$.
 For each index $1\le i\le 2n$, let $e_i$ and $\ol e_i$ be defining sections of the components $S_i^+$ and $S_i^-$ of $\tilde\Phi\inv(\ms C^i)$ (see \eqref{Spm}), and 
put $s_i:=e_i\ol e_i\in H^0(F)^{T^2_{\CC}}$.
Since the reducible divisor $S_i^++S_i^-$ is defined by $z=a_iu$ as in \eqref{Spm}, 
under a suitable choice of $v_1$ (i.e.,\,the real number $\lmd$), we may write 
$s_i = v_1 - a_i v_2$ for $1\le i\le 2n$.
From the explicit form \eqref{gen2} of $X$ and $\ol X$, we have, for some real number $c\neq 0$,
\begin{align}
\xi\ol\xi &= c 
\delta^2\ol\delta^2 
\prod_{i=1}^{2n}s_i\notag\\
&= c(\delta\ol\delta)^2 
\prod_{i=1}^{2n}(v_1-a_i v_2).
\label{dd}
\end{align}
Expanding the product, we obtain a homogeneous polynomial of $v_1$ and $v_2$ of degree $2n$.
Each monomial in this polynomial can be written (not in a unique way)
as the product of two elements of the form 
$z_i/(\ddd\ol\ddd)$. %\frac{z_i}{\ddd\ol\ddd}$.
Hence, there exists a quadratic polynomial $Q$ of $(n+1)$ variables with real coefficients, such that 
$$
\prod_{i=1}^{2n}
(v_1 - \lmd_i v_2)= Q\Big(
\frac{z_0}{\ddd\ol\ddd},
\frac{z_1}{\ddd\ol\ddd},
\dots,
\frac{z_n}{\ddd\ol\ddd}
\Big).
$$
Since $Q$ is quadratic, a multiplication by $(\ddd\ol\ddd)^2$ to the RHS gives 
$Q(z_0,z_1,\dots, z_n)$.
Hence, by letting the coefficient $c$ in \eqref{dd} be absorbed in
$Q$, we obtain that the image ${\tilde \Psi}(\tilde Z)\subset\PP^{n+2}$ satisfies the equation
\begin{align}\label{T1}
z_{n+1}z_{n+2} = Q(z_0,z_1,\dots, z_n). 
\end{align}
If we write the LHS as $w_{n+1}^2 - w_{n+2}^2$ by putting $z_{n+1}=w_{n+1} - w_{n+2}$ and $z_{n+2} = w_{n+1} + w_{n+2}$, then this equation can be written $w_{n+2}^2 = w_{n+1}^2 -Q(z_0,z_1,\dots, z_n).$
Since ${\tilde \Psi}(\tilde Z)$ lies over the rational normal curve $\Lmd\subset\PP^n$, this means that ${\tilde \Psi}(\tilde Z)$ is contained in the double cover of the cone ${\rm C}(\Lmd)\subset\PP^{n+1}$ branched along 
the intersection with the quadric $\{w_{n+1}^2=Q(z_0,\dots,z_n)\}\subset\PP^{n+1}$.  
By the projection to $\Lmd$, this intersection is a double cover of $\Lmd$ branched at the intersection with the quadric $\{Q=0\}$ in $\PP^n$, and from \eqref{dd}, the last intersection consists of the points $a_1,\dots,a_{2n}$.
Therefore the projection from the double cover of ${\rm C}(\Lmd)$ to $\Lmd$ has these $2n$ points as its discriminant locus.
From the description of the minitwistor space ${\ms T}$ given at the end of Section \ref{s:2}, 
this implies that the surface $\tilde\Psi(\tilde Z)$ is exactly $\ms T$.

Next, we show that the map $\tilde\Psi$ is onto $\ms T$.
Since the base locus of $|\bm L|^{\CC^*_s}$ is contained in the two fibers $\tilde Z_0$ and $\tilde Z_{\infty}$, its restriction to the fiber $\tilde{\ms T}=\tilde p\inv(1)$ is base point free.
The generators of the restriction system $(|\bm L|^{\CC^*_s})|_{\tilde{\ms T}}$ 
can be obtained explicitly from Proposition \ref{p:gen} just as restrictions, and using them, the image $\tilde\Psi(\tilde{\ms T})$ satisfies the same equation \eqref{T1} as $\tilde\Psi(\tilde Z)$.
Further, using these generators, we readily see that the restriction $\tilde\Psi|_{\tilde{\ms T}}$ is exactly the contraction map of the two $(-n)$-curves $D_1=\bm D\cap \tilde{\ms T}$ and $\ol D_1=\ol{\bm D}\cap \tilde{\ms T}$.
In particular, $\tilde\Phi(\tilde {\ms T}) = \ms T$.
Hence, $\tilde\Phi(\tilde Z) = \ms T$
and he restriction map $\tilde\Phi|_{\tilde {\ms T}}:\tilde{\ms T}\lras\ms T$ is identified with the minimal resolution of $\ms T$.
A proof of Proposition \ref{p:ZT} is thus completed.
\proofend

\medskip
From this proposition, since $\tilde\Psi:\tilde Z\lras\ms T$ is $\CC^*_s$-equivariant, a generic fiber of $\tilde\Psi$ is irreducible. Therefore, we call the meromorphic mapping $\tilde\Psi$ as the {\em meromorphic quotient map}.

\section{The images of twistor lines to the minitwistor space}

\subsection{Equations of $\CC^*$-invariant twistor lines}\label{ss:itl}
We recall that on the fiber $\tilde Z_0=\tilde p\inv(0)$, there is a chain of smooth rational curves $C_0+ C_1+\dots + C_{2n}$ and
each component is $T^2_{\CC}$-invariant.
The group $\CC^*$ in the title means the stabilizer subgroup of a component of this chain and hence it depends on the component.
Note that $C_i\subset Z$ for every index $0\le i<2n$ but not for $i=0,2n$ because one of the two $T^2$-fixed points on each of these components does not belong to $Z$.

First, we will obtain the equations of twistor lines that intersect the above chain, by using the projection $\Phi: Z\lras\ms O(2)$.
Let $(x,y,z,u)$ be the coordinates we have used. In particular, $(z,u)$ are coordinates on $\ms O(2)$.
In the following, for any index $1\le i\le 2n$, we denote $L_i\subset Z$ for the twistor line through the $T^2$-fixed point $C_{i-1}\cap C_i$.
These $2n$ twistor lines are very special in that they are $T^2_{\CC}$-invariant.
The following statement is a slight refinement of (3.3) in \cite{Hi79}.

\begin{proposition}\label{p:tl}
For any index $1\le i\le 2n$, the twistor line $L_i\subset Z$ is 
the transformation of the curve
\begin{align}\label{Li}
\{ x=y=z-a_iu = 0\}
\end{align}
in the projective model \eqref{tw1}.
If $L\subset Z$ is a twistor line that intersects $C_i$ for some $0\le i\le 2n$ and it is none of $L_1,\dots,L_{2n}$, then it is the transformation of the curve
\begin{align}\label{xuyuzu}
x(u) = \sqrt{\prod_{1\le j\le 2n}|\lmd-a_j|}\cdot e^{\sqrt{-1}\theta}u^i,
\quad
y(u) = \sqrt{\prod_{1\le j\le 2n}|\lmd-a_j|}\cdot e^{-\sqrt{-1}\theta}u^{2n-i},
\quad 
z(u)=\lmd u
\end{align}
for some real numbers $\lmd\in (a_i,a_{i+1})$ and $\theta$.
%
%Conversely, for any index $0\le i\le 2n$, every rational curve of the form \eqref{xuyuzu} is a twistor line through a point of the curve $C_i$.
\end{proposition}

Note that from \eqref{Li} and the last equation $z=\lmd u$ in \eqref{xuyuzu}, in terms of the curve $\ms C_{\lmd}$ defined in \eqref{Clmd}, a twistor line $L$ that intersects $C_i$
satisfies
\begin{align}\label{iml}
\Phi(L) = \ms C_{\lmd},\quad
a_i\le \lmd \le a_{i+1},
\end{align}
and one of the equalities holds iff $L=L_i$ or $L_{i+1}$ respectively.

\proof
For each index $0\le i\le 2n$, 
let $G_i\simeq\CC^*$ be the subgroup of $T^2_{\CC}=\{(s,t)\in \CC^*_s\times\CC^*_t\}$ which fixes every point of $C_i$.
Also, as before, we denote 
$\CC^*_t=\{(1,t)\in T^2_{\CC}\set t\in\CC^*\}$.
The map $\Phi:Z\lras\ms O(2)$ is a quotient map of the $\CC^*_t$-action.
If a twistor line $L\subset Z$ intersects a component $C_i$, then $L$ is an orbit closure of the $G_i$-action (even when $i=0,2n$). 
Since the projection $\Phi$ is $T^2_{\CC}$-equivariant with the action on $\ms O(2)$ being the one by $T^2_{\CC}/\CC^*_t\simeq\CC^*_s$, as $G_i$ is mapped onto the quotient group, the image $\Phi(L)$ is an orbit closure of the $\CC^*_s$-action on $\ms O(2)$.
Therefore, taking the real structure \eqref{rsZ} into account, in the coordinates $(z,u)$ on $\ms O(2)$, $\Phi(L)$ is defined by an equation of the form $z=\lmd u$ for some $\lmd\in \RR$.

Substituting this into the equation \eqref{tw1} of the projective model of the twistor space, we obtain that (the transformation of) $L$ satisfies an equation
\begin{align}\label{Slmd2}
xy= (\lmd-a_1)\dots(\lmd-a_{2n})u^{2n}.
\end{align}
As $x,y\in \ms O(2n)$, 
$x = x(u)$ and $y=y(u)$ are respectively polynomials of degree $d$ and $2n-d$ for some $d$ with $0\le d\le 2n$.
If $i\neq 0,2n$, then 
because coordinates on $C_i$ is given as in \eqref{nui}, 
if $c_i\in \CC$ means the value of the coordinate of the intersection $C_i\cap L$ in non-homogeneous form and assuming that $L$ is not $L_i$ nor $L_{i+1}$ so that $c_i\neq 0,\infty$, 
we have $d=i$ and $L$ satisfies the equations
\begin{align}\label{xuyu}
x(u) = c_i\prod_{j=1}^i(\lmd-a_j)\cdot u^i
\qandq
y(u) = c_i\inv\prod_{j=i+1}^{2n}(\lmd-a_j)\cdot u^{2n-i}.
\end{align}
Noting that a coordinate on the axis $C_0$ (resp.\,$C_{2n}$) is $x$
(resp.\,$y$), it is easy to see that \eqref{xuyu} is valid even when $i=0,2n$.

Using that this is real (i.e.\,$\sigma$-invariant), from the explicit form \eqref{rsZ} of $\sigma$, we deduce the requirement
\begin{align}\label{c}
|c_i|^2 = (-1)^i \frac{\prod_{j=i+1}^{2n}(\lmd-a_j)}
{\prod_{j=1}^{i}(\lmd-a_j)}.
\end{align}
It is elementary to see that the RHS is strictly positive if $\lmd\not\in\{a_1,a_2,\dots,a_{2n}\}$,
strictly monotone (decreasing) on each interval $(a_i,a_{i+1})$ for any $0\le i\le 2n$ if we put $a_0=-\infty$ and $a_{2n+1}=\infty$,
and tends to $\infty$ as $\lmd\searrow a_i$ and to $0$ as $\lmd\nearrow a_{i+1}$.
The complex number $c_i$ takes an arbitrary value of $\CC^*$ as the intersection point varies on the whole of $C_i\minus\{C_i\cap C_{i-1}, C_i\cap C_{i+1}\}\simeq\CC^*$.
%(if we assume that the intersection is not a fixed point).
From \eqref{c} and the above properties of the RHS of \eqref{c}, 
this means that, for a fixed $i$, $\lmd$ can vary only in the interval $(a_i,a_{i+1})$ as the intersection point with $L$ moves on $C_i$, and $\lmd$ approaches $a_i$ (resp.\,$a_{i+1}$) as the intersection point approaches $C_i\cap C_{i-1}$ (resp.\,$C_i\cap C_{i+1}$).
In particular, the image $\Phi(L_i)$ of the invariant twistor line $L_i$ ($1\le i\le 2n$) has to be a curve $z=a_iu$
and $L_i$ is the transformation of the curve $\{x=y=0\}$ over this curve. This gives \eqref{Li}.
If $L$ intersects $C_i$ ($0\le i\le 2n)$ but is not $L_i$ nor $L_{i+1}$, then letting $\lmd$
be a real number belonging to the interval $(a_i,a_{i+1})$
which is uniquely determined by the equation \eqref{c} thanks to the above monotonicity of the RHS of \eqref{c},
then substituting $c_i$ that is determined from \eqref{c} into \eqref{xuyu},
we obtain \eqref{xuyuzu}. 
%
%The reverse direction is obvious from the above argument because any point of $C_i$ is passed by some twistor line
%(unless $i=0,2n$ and the point is one of the two points on $C_0$ or $C_{2n}$ which does not belong to the twistor space $Z$.)
\proofend

\medskip
%The next corollary is obvious from \eqref{c}:
%
%\begin{corollary}
%For each real number $\lmd$
%different from $a_1,\dots,a_{2n}$, there exist $S^1$-s worth of twistor lines in $Z$ which are mapped to the curve $\ms C_{\lmd}=\{z=\lmd u\}$ by the projection $\Phi:Z\lras\ms O(2)$.
%\end{corollary}

Using this, we determine the structure of a generic member of the $T^2_{\CC}$-invariant pencil $|F|^{T^2_{\CC}}$ as follows:

\begin{proposition}\label{p:S}
If $\lmd\in \CC$ is none of $a_1,a_2,\dots,a_{2n}$, then the divisor $S_{\lmd}=\tilde\Phi\inv(\ms C_{\lmd})$ is biholomorphic to a toric surface obtained from $\qdr$ by blowing up each of the four fixed points on the toric surface $\qdr$ $n$ times, 
where the iterated blowups are always done in the direction of the $(0,1)$-curve passing through the point.
\end{proposition}

\proof
For the smoothness of $S_{\lmd}$ for $\lmd$ as in the proposition,
using that $S_{\lmd}$ is the transformation of the surface defined by the equation \eqref{Slmd2}, it is easy to see that $S_{\lmd}$ is smooth except possibly at points of the exceptional curves
$C_1,\dots, C_{2n-1}$ and their conjugate.
But since the simultaneous resolution of the projective model \eqref{tw1} of $Z$ restricts to the minimal resolution of the surface \eqref{Slmd2}, $S_{\lmd}$ is smooth also on these curves.
Further, it is a toric surface since it is $T^2_{\CC}$-invariant.

Next, we identify the structure of $S_{\lmd}$ as a toric surface.
We put $D_{\lmd} = \bm D|_{S_{\lmd}}$, which is a smooth section of $\tilde\Phi|_{S_{\lmd}}:S_{\lmd}\lras\ms C_{\lmd}$.
Then the self-intersection number of $D_{\lmd}$ in $S_{\lmd}$ can be calculated as
$$
D_{\lmd}^2 = (\bm D|_{S_{\lmd}})^2 = \bm D\cdot \bm D\cdot S_{\lmd} = \bm D|_{\bm D}\cdot {S_{\lmd}}|_{\bm D}
= \bm D|_{\bm D}\cdot D_{\lmd},
$$
where the last two intersection numbers are in $\bm D\simeq\FF_2$.
As in \eqref{DD}, $\bm D|_{\bm D}\simeq\ms O_{\FF_2}(-n\Xi_0)$. 
On the other hand, $D_{\lmd}$ is a $(+2)$-section of the ruling $\bm D\simeq\FF_2\lras\PP^1$.
Hence, we obtain $D_{\lmd}^2 = -n\Xi_0^2 = -2n.$
By reality, this implies $\ol D_{\lmd}^2 = -2n$.
From Proposition \ref{p:cb} and adjunction,
noting that $F|_{S_{\lmd}}$ may be represented by two fibers of $\tilde\Phi|_{S_{\lmd}}:S_{\lmd}\lras\ms C_{\lmd}$ so that $(F|_{S_{\lmd}})^2 = 0$, 
\begin{align*}
K_{S_{\lmd}}^2 &= \big((K_{\tilde Z} + F)|_{S_{\lmd}}\big)^2
\quad(\because {\text{adjunction formula}})\\
& = \big((-F-\bm D-\ol{\bm D})|_{S_{\lmd}}\big)^2
\quad(\because {\text{Proposition \ref{p:cb}}})
\\ &=  D_{\lmd}^2 + \ol D_{\lmd}^2 + 2F|_{S_{\lmd}}\cdot D_{\lmd}
+ 2F|_{S_{\lmd}}\cdot \ol D_{\lmd}\quad(\because (F|_{S_{\lmd}})^2 = 0 {\text{ and $\bm D\cap \ol{\bm D}=\emptyset$}})
\\
&= -2n - 2n + 2\cdot 2 + 2\cdot 2\\
 & = 8 - 4n.
\end{align*}
Furthermore, the curves $C_1,\dots, C_{2n-1}$ are $(-2)$-curves on $S_{\lmd}$ as the simultaneous resolution of the projective model of $Z$ restricts to the minimal resolution of the surface \eqref{Slmd2}.
Combining these, we readily obtain that $S_{\lmd}$ is as in the proposition.
\proofend

\subsection{Elimination of the base locus}\label{ss:elm}
As we have seen, the linear system $|\bm L|^{\CC^*_s}$ on $\tilde Z$ induces the meromorphic quotient map $\tilde{\Psi}:\tilde Z\lras\ms T$ to the minitwistor space (Proposition \ref{p:ZT}) and the base locus of the linear system is exactly the exceptional curves of the simultaneous resolution of the projective model of the twistor space
(Proposition \ref{p:bs}).
In this section, we explicitly give a sequence of blowups that completely eliminates the base locus of the linear system.
Often we only indicate the operation for elimination of the base locus over the point $u=0\in\PP^1$ because the operation over the point $u=\infty$ is automatically determined from the operation over the point $u=0$ by the real structure.
As an important remark, the chain $C_0 + C_1 + \dots + C_{2n-1} + C_{2n}$ is ``symmetric'' about the central sphere $C_n$ and all operations will be given in a way that they preserve this symmetry.
We will illustrate the changes that occur on the fiber over the point $z=a_1$ in the case $n=4$.

For the first operation, we blow up $\tilde Z$ not along the base locus of 
$|\bm L|^{\CC^*_s}$ itself but along longer chains $\sum_{0\le i\le 2n}C_i$ and $\sum_{0\le i\le 2n}\ol C_i$.
These are rather the base locus of the invariant pencil $|F|^{T^2_{\CC}}$ on $\tilde Z$.
Let 
\begin{align*}
\hat Z\lras \tilde Z,\quad
E_i\text{ and } \ol E_i\quad(0\le i\le 2n)
\end{align*}
be this blowup and the exceptional divisors over $C_i$ and $\ol C_i$ respectively.
All $E_i$ and $\ol E_i$ are biholomorphic to $\qdr$ and one of the two factors may be identified with $C_i$ and $\ol C_i$ respectively by the projection to $\tilde Z$.
We put
\begin{align}\label{E2}
E:=E_n + \sum_{i=1}^{n}\big(E_{n-i} + E_{n+i}\big)
\end{align}
and similarly for $\ol E$,
where we are respecting the symmetry about the central component $E_n$.
Note that $E$ is different from the divisor $\bm E$ appearing in the compactification of the twistor space. 
The space $\hat Z$ has an ordinary double point over each singularity of the chains, and using the same symbol for the strict transform of $L_i$, they are exactly the intersection points $L_i\cap E$ and $L_i\cap \ol E$ for $1\le i\le 2n$.
The singularity $L_i\cap E$ of $\hat Z$ is shared by the four $T^2_{\CC}$-invariant divisors 
$E_{i-1}, E_i, S^+_i$ and $S^-_i$.
(See Figure \ref{f:web} for the case $n=4$ and $i=1$.)
Here and in what follows, we do not change symbols for the strict transforms of $S_i^+$ and $S_i^-$ for simplicity.

\begin{figure}
\includegraphics[height=8in]{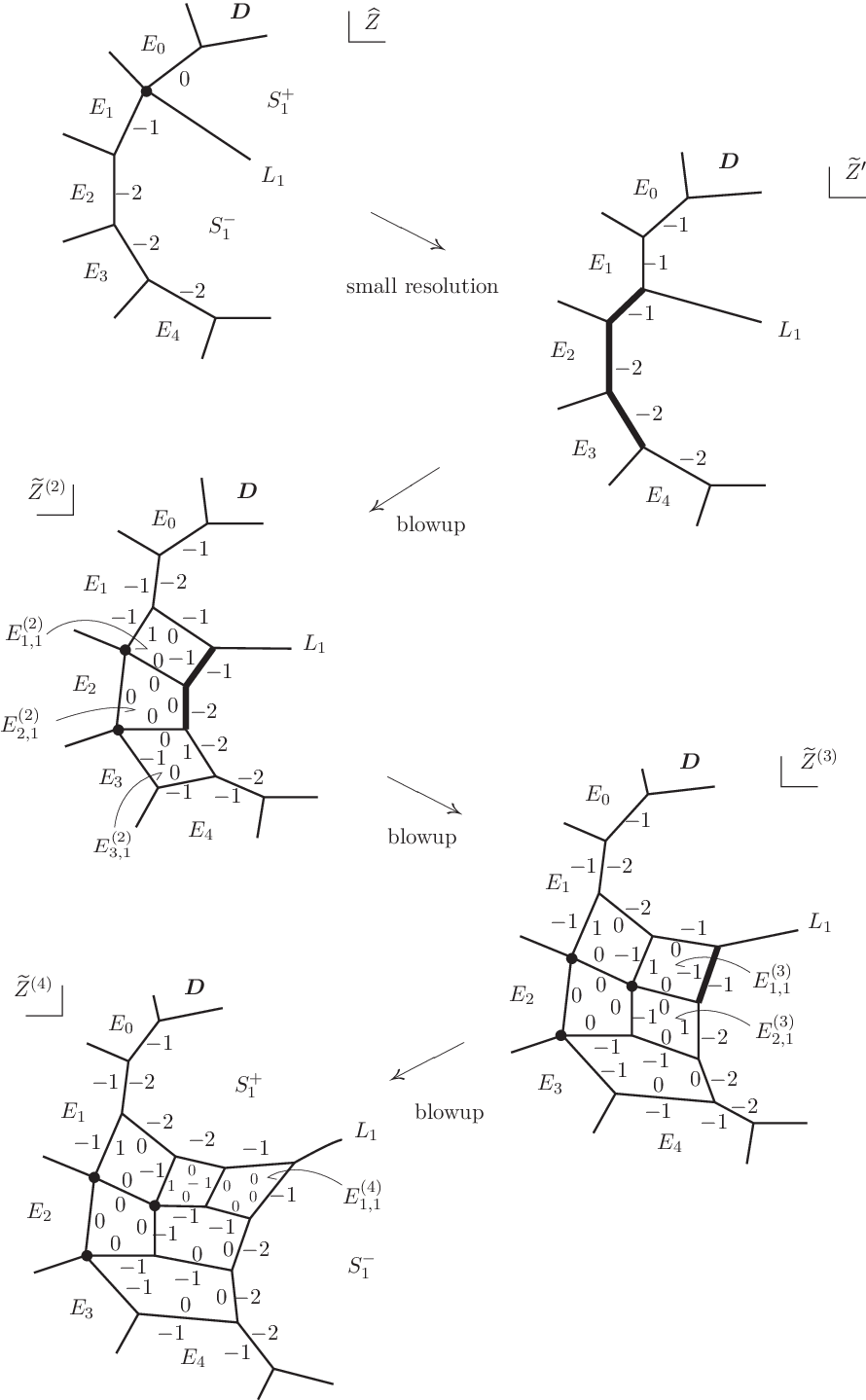}
\caption{The elimination of the base locus
}
\label{f:web}
\end{figure}

In the following, we call an ordinary double point simply a node.
For each node of $\hat Z$ that belongs to the sum $E$, taking the symmetry about $C_n$ into account, we take the small resolution such that:
\begin{align}\label{sr}
\setlength{\arraycolsep}{-10pt} 
\left\{
\begin{array}{@{}l@{}}
\text{for $1\le i\le n$, the divisors $E_i$ and $S_i^+$ are blown up,} \\
\text{for $n<i\le 2n$, the other pair $E_{i-1}$ and $S_i^-$ are blown up.} \\
\end{array}
\right.
\end{align}
(See Figure \ref{f:web} for the case $n=4$ and $i=1$.)
In particular, the end components $E_0$ and $E_{2n}$ of $E$ are not blown up,
the central component $E_n$ is blown up twice (at the intersection with $L_n$ and $L_{n+1}$),
and all other components $E_i$ are blown up once (at the intersection with $L_i$).
We denote $\tilde Z'\lras\hat Z$ for these small resolutions as well as the ones over $u=\infty$ and let
$$
\mu_1:\tilde Z'\lras \tilde Z
$$
be the composition $\tilde Z'\lras\hat Z\lras \tilde Z$.
The space $\tilde Z'$ is smooth and the exceptional locus of $\mu_1$ can be written as in \eqref{E2} and its conjugate, using the same symbols for the strict transforms of the components into $\tilde Z'$.
Since both changes $\hat Z\lras\tilde Z$ and $\tilde Z'\lras\hat Z$ are done in a $\sigma$-invariant and $T^2_{\CC}$-invariant locus, the real structure and the $T^2_{\CC}$-action on $\tilde Z$ naturally lift upto $\tilde Z'$
via $\hat Z$. 

\begin{figure}
\includegraphics{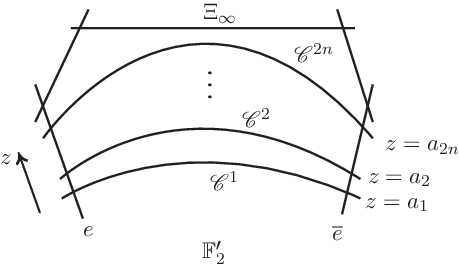}
\caption{The blowup of $\bm D\simeq\FF_2$}
\label{f:F}
\end{figure}

We recall that $\tilde Z$ has a projection $\tilde\Phi:\tilde Z\lras\FF_2$. The center of the first blowup $\hat Z\lras \tilde Z$ is precisely the two fibers of $\tilde\Phi$ over the base points of the pencil $\{\ms C_{\lmd}\set \lmd\in\CC\cup\{\infty\}\}$.
(See \eqref{Clmd} and \eqref{Cinf}.)
In terms of the coordinates $(z,u)$, the last base points are $(z,u)=(0,0)$ and $(0,\infty)$.
So if ${\FF}'_2\lras\FF_2$ is the blowup of $\FF_2$ at these two points, $\tilde\Phi$ lifts to a holomorphic map $\hat\Phi:\hat Z\lras \FF'_2$ and we obtain the commutative diagram of holomorphic maps
\begin{equation}\label{bd}
\begin{tikzcd}
\tilde Z' \arrow[r] \arrow[dr,"\tilde\Phi'"] \arrow[ddr,swap,"\tilde f'"] \arrow[drrr, bend left = 50, "\tilde p'"] 
& \hat Z \arrow[r] \arrow[d,"\hat\Phi"] & \tilde Z \arrow[d,swap,"\tilde\Phi"] \arrow[dr,swap,"\tilde p"]\\
& \FF'_2 \arrow[r] \arrow[d]& \FF_2 \arrow[r] &\PP^1\\
& \PP^1
\end{tikzcd}
\end{equation}
Here, the map $\tilde p'$ is the composition of three maps via $\hat Z$ and $\tilde Z$,
$\tilde\Phi'$ is the composition $\tilde Z'\lras\hat Z\lras \FF'_2$,
the horizontal map $\FF_2\lras\PP^1$ is the projection of the ruling $\FF_2\lras\PP^1$ which takes the $u$-coordinate, the vertical map 
$\FF'_2\lras\PP^1$ is a holomorphic map induced by the strict transform of the pencil $\{\ms C_{\lmd}\set \lmd\in\CC\cup\{\infty\}\}$ on $\FF_2$, and $\tilde f'$ is the compositioin $\tilde Z'\lras\FF'_2\lras\PP^1$.
We denote $e$ and $\ol e$ for the exceptional curves 
of the blowup $\FF'_2\lras\FF_2$ over $u=0$ and $u=\infty$ respectively.
(See Figure \ref{f:F}.)
These are sections of the vertical map 
$\FF'_2\lras\PP^1$ in \eqref{bd}.
Note that there exists a natural isomorphism between $e\simeq\PP^1$ and the fiber of the composition $\FF'_2\lras\FF_2\lras\PP^1$ over the point $u=1$,
where the second map $\FF_2\lras\PP^1$ is the ruling map.
In the intermediate space $\hat Z$, there are natural identifications $E_i\simeq C_i\times e$ and $\ol E_i\simeq\ol C_i\times\ol e$ for any index $0\le i\le 2n$.
The map $\tilde f'$ is exactly the one induced by the pencil formed by the strict transforms of members of the invariant pencil
$|F|^{T^2_{\CC}}$ on $\tilde Z$.
We also call this pencil on $\tilde Z'$ by the same name.
This is base point free.
The coordinate $z$ can be used as (non-homogeneous) coordinate on the target space of $\tilde f'$, and $\tilde f'$ has singular fibers exactly over the $(2n+1)$ points $z=a_1,a_2,\dots,a_{2n}, \infty$.
The fiber over $z=a_i$ is (the strict transform of)
the divisor $S_i^+ + S_i^-$ and the fiber over $z=\infty$ is 
(the strict transform of)
the divisor $\tilde Z_0 + \tilde Z_{\infty} + \bm E $.
From Proposition \ref{p:S}, all other fibers of $\tilde f'$ are smooth toric surfaces which are mutually isomorphic.

The $(n+1)$ generators \eqref{gen1} of the linear system $|\bm L|^{T^2_{\CC}}$ on $\tilde Z$ contains every component of the center of the first blowup $\hat Z\lras\tilde Z$ with the same multiplicity $n$.
On the other hand, from the definition of the divisors $X$ and $\ol X$ of $|\bm L|^{\CC^*_s}$ in \eqref{gen2} and the inclusion \eqref{incl}, we readily see that the two generators $X$ and $\ol X$ contain the component $C_i$ ($0\le i\le 2n$) with multiplicity exactly $(2n-i)$ and $i$ respectively. From these and the real structure, it follows that
the least multiplicity of the component $C_i$ 
for members of $|\bm L|^{\CC^*_s}$ is $n - |n-i|$ for any index $0\le i\le 2n$.
In particular, the central sphere $C_n$ is contained with the highest multiplicity $n$.
(This is consistent with the fact that the central spheres $C_n$ and $\ol C_n$ are the source and sink of the $\CC^*_s$-action \cite{Hi25}.)

Therefore, if we define a line bundle $\bm L'$ on $\tilde Z'$ as
\begin{align}\label{L'}
\bm L'
&=\mu_1^*\bm L
-
\sum_{i=1}^{2n-1}\big(n - |n-i|\big) \big(E_i+\ol E_i\big)
\end{align}
then the linear system $|\bm L'|^{\CC^*_s}$ has no fixed component.
Obviously, the base locus of $|\bm L'|^{\CC^*_s}$ is contained in the exceptional divisors of $\mu_1$.
We next determined it precisely.
We define curves on $\tilde Z'$ by
\begin{align}\label{C'}
C'_{i,j} := \begin{cases}
E_i \cap S^-_j  & {\text{if $1\le j\le i < n$}},\\
E_i\cap S^+_j & {\text{if $n<i< j\le 2n$}}.
\end{cases}
\end{align}
Each of these is a smooth rational curve,
and is mapped isomorphically onto $C_i$ by $\mu_1:\tilde Z'\lras \tilde Z$.
On $\tilde Z'$, the divisor $S^+_j+S_j^-$ is the fiber of $\tilde f'$ over $z=a_j$, and the intersections $S^+_j\cap E$ and $S^-_j\cap E$ divide the fiber chain of the projection $\tilde \Phi'|_E:E\lras e$ over the point $\ms C^j\cap e$ into two subchains.

\begin{proposition}\label{p:bs1}
The base locus of the linear system $|\bm L'|^{\CC^*_s}$ on $\tilde Z'$
is contained in the intersection chains $\cup_{j=1}^{2n}(S_j^+ \cup S_j^-)\cap E$ and $\cup_{j=1}^{2n}(S_j^+ \cup S_j^-)\cap \ol E$, and the part contained in $E$ consists of the following $(n-1)$ chains of rational curves
\begin{align}
C'_{1,1}\cup C'_{2,1}\cup C'_{3,1}\cup\dots\cup C'_{n-1,1}
 &\;\cdots\; (n-1)\text{ components, lying on }S_1^-, \label{cbc1}\\
C'_{2,2}\cup C'_{3,2}\cup\dots\cup C'_{n-1,2}
 &\;\cdots\; (n-2)\text{ components, lying on }S_2^-, \label{cbc2}\\
\vdots\notag\\
C'_{n-2,n-2}\cup C'_{n-1,n-2}
 &\;\cdots\; 2\text{ components, lying on }S_{n-2}^-, \label{cbc3}\\
C'_{n-1,n-1}
 &\;\cdots\; 1\text{ component, lying on }S_{n-1}^-, \label{cbc4}
\end{align}
and also, symmetrically about the central component $E_n$, the following $(n-1)$ chains of rational curves
\begin{align}
C'_{n+1,n+2}
 &\;\cdots\; 1\text{ component, lying on }S_{n+2}^+, \label{cbc5}\\
C'_{n+2,n+3}\cup C'_{n+1,n+3}
 &\;\cdots\; 2\text{ components, lying on }S_{n+3}^+, \label{cbc6}\\
\vdots\notag\\
C'_{2n-2,2n-1}\cup C'_{2n-3,2n-1}\cup\dots\cup C'_{n+1,2n-1}
 &\;\cdots\; (n-2)\text{ components, lying on }S_{2n-1}^+, \label{cbc7}\\
C'_{2n-1,2n}\cup C'_{2n-2,2n}\cup C'_{2n-3,2n}\cup\dots\cup C'_{n+1,2n}
 &\;\cdots\; (n-1)\text{ components, lying on }S_{2n}^+. \label{cbc8}
\end{align}
\end{proposition}

Of course, from the real structure, the part of the base locus that is contained in $\ol E$ 
is given as the complex conjugate of these loci.
Note that the proposition means that on $\tilde Z'$ there is no base point at all on the divisors $S_n^++S_n^-$ and $S_{n+1}^++S_{n+1}^-$.

\proof
From \eqref{incl} and \eqref{L'}, the transformation of the divisor $X$ in \eqref{gen2} into $\tilde Z'$ can be calculated as
\begin{align}\label{gen3}
\tilde X':=\sum_{j=1}^{2n} S_j^+ 
+ 2\sum_{i=1}^n i \big(E_{n-i} + \ol E_{n+i}\big)
+ 2\bm D. 
\end{align}
Note that this does not contain $E_{n+i}$ nor $\ol E_{n-i}$ as a component for any $i\ge 0$.
The transformation of the generator $\ol X$ is the complex conjugate of \eqref{gen3}.
The intersection of these two transformations can be confirmed to consist of the chains in the proposition exactly.

Also it is easy from \eqref{L'} to see that the transformations of the $(n+1)$ generators in \eqref{gen1} are given by
\begin{align}\label{gen4}
(n-m)S_{\infty}
    + mS_{\lmd} + \bm D + \ol {\bm D}
    + \sum_{i=0}^{2n}|n-i|(E_i + \ol E_i),\quad 0\le m\le n.\end{align}
In particular, they contain $E_i$ and $\ol E_i$ for any $0<i<2n$ with $i\neq n$.
Since the above intersection is clearly contained in these $E_i$ and $\ol E_i$, it is indeed the base locus of 
$|\bm L'|^{\CC^*_s}$.
\proofend

\medskip
Thus, the base locus on $\tilde Z'$ consists of several chains of rational curves, each of which is contained in either $S^+_j$ or $S^-_j$ for some $1\le j\le 2n$ with $j\neq n,n+1$. This means that {\em the base locus on $\tilde Z$ splits into several chains by $\mu_1:\tilde Z'\lras \tilde Z$
 but they are still on $S_j^+$ or $S_j^-$ for some $j$}.
When $n=2$, they are only on $S_1^+\cup S_1^-$ and $S_4^+\cup S_4^-$ and the chains consist of a single curve respectively.
(They are $C'_{1,1}$ and $C'_{3,4}$ and their conjugate.)

As the second step of the elimination of the base locus of $|\bm L|^{\CC^*_s}$, let 
$$
\mu_2:\tilde Z\uptwo \lras\tilde Z'
\qandq
E\uptwo_{i,j}, \ol E\uptwo_{i,j}
$$
be the blowup along all chains of the base locus of 
$|\bm L'|^{\CC^*_s}$ obtained in Proposition \ref{p:bs1} and the exceptional divisors over the curves $C'_{i,j}$ and $\ol C'_{i,j}$ respectively. 
Again, the real structure and the $T^2_{\CC}$-action naturally lift on $\tilde Z\uptwo$. 
By $\mu_2$, a string of the exceptional divisors
\begin{align}\label{exe01}
E\uptwo_{i,i}\cup E\uptwo_{i+1,i}\cup \dots \cup E\uptwo_{n-1,i}
\end{align}
are inserted in between $S^-_i$ with $i<n$ and the exceptional divisor $E$, and 
\begin{align}\label{exe02}
E\uptwo_{i-1,i}\cup E\uptwo_{i-2,i}\cup\dots\cup E\uptwo_{n+1,i}
\end{align}
are inserted in between $S_i^+$ with $i>n+1$ and $E$.
(See Figure \ref{f:web} for the case $n=4$ and $i=1$.)
As a consequence, these divisors are separated in $\tilde Z\uptwo$.
In terms of the map $\tilde f':\tilde Z'\lras\PP^1$ in the diagram \eqref{bd}, 
the strings \eqref{exe01} and \eqref{exe02} are inserted in the fiber $(\tilde f')\inv(a_i)$.

Since the blowup centers of $\mu_2$ are chains if $n>2$, just like $\hat Z$,  $\tilde Z\uptwo$ has nodes over the singularities of the chains if $n>2$.
(In Figure \ref{f:web}, they are indicated by the dotted points.)
If $n=2$, $\tilde Z\uptwo$ has no node.
On the other hand, this time, a component of the exceptional divisor of $\mu_2$ is not isomorphic to $\PP^1\times\PP^1$ if it is any one of the two end components of the string.
(For the cases $i=n-1$ and $i=n+2$, where the string consists of a single curve, it is not interpreted as an end component and it is isomorphic to $\qdr$.)
Namely, for the former string \eqref{exe01}, the components $E\uptwo_{i,i}$ and $E\uptwo_{n-1,i}$ are isomorphic to $\FF_1$ if $i\neq n-1$, while all other components are isomorphic to $\PP^1\times\PP^1$.
The same thing holds for the latter string \eqref{exe02}.

We use the same symbols for the exceptional divisors $E_i, \ol E_i\subset \tilde Z'$ and their strict transforms into $\tilde Z\uptwo$.
Then in $\tilde Z\uptwo$, the central components $E_n$ and $\ol E_n$ have the following property, which will be important later.
\begin{proposition}\label{p:El}
In the variety $\tilde Z\uptwo$, the central divisors $E_n$ and $\ol E_n$ are isomorphic to the general fiber $\tilde{\ms T}=\tilde p\inv(1)$.
\end{proposition}

\proof
Since all the blowups and the small resolutions preserve the real structure, it suffices to show the claim for $E_n$.
On the space $\hat Z$, the divisor $E_n$ is isomorphic to $C_n\times e\simeq\qdr$, where $e$ was one of the exceptional curves of the blowup $\FF'_2\lras\FF_2$.
%naturally identified with the base $\PP^1$ of the map $f$. 
From the choice of the small resolutions of $\hat Z$, by $\tilde Z'\lras \hat Z$, the divisor $E_n$ is blown up at two points 
$(C_{n-1}\cap C_n)\times \{a_{n}\}$ and $(C_{n}\cap C_{n+1})\times \{a_{n+1}\}$ under the above identification.
Next, by the blowup $\mu_2:\tilde Z\uptwo \lras \tilde Z'$, 
$E_n$ is further blown up at distinct $2(n-1)$ points;
explicitly, the $(n-1)$ points $(C_{n-1}\cap C_n)\times \{a_{i}\}$
for $1\le i<n$ and other $(n-1)$ points $(C_{n}\cap C_{n+1})\times \{a_{i}\}$
for $n<i\le 2n$, still under the above identification.

Therefore, if $\tilde f\uptwo:\tilde Z\uptwo\lras\PP^1$ denotes the composition $\tilde Z\uptwo\stackrel {\mu_2}\lras\tilde Z'\stackrel {\tilde f'}\lras\PP^1$,
then in $\tilde Z\uptwo$, the restriction of $\tilde f\uptwo$
to $E_n$ has reducible fibers exactly over $z=a_i$
($1\le i\le 2n$), and the strict transforms of the two sections
$(C_{n-1}\cap C_n)\times\PP^1$ and 
$(C_{n}\cap C_{n+1})\times\PP^1$ are sections of $\tilde f'|_{E_n}:E_n\lras\PP^1$ whose self-intersection numbers are $(-n)$.
On the other hand, the surface $\tilde{\ms T}$ also has the same structure from the equation \eqref{Hi1} (or Figure \ref{f:mt}).
Hence $E_n$ is isomorphic to $\tilde{\ms T}$
\proofend

\medskip
Next we investigate the base locus on $\tilde Z\uptwo$.
From \eqref{gen3}, each chain given in Proposition \ref{p:bs1} is contained in either $\tilde X'$ or $\sigma(\tilde X')$ with multiplicity exactly one.
For example, each component of the first chain
$C'_{1,1}\cup C'_{2,1}\cup C'_{3,1}\cup\dots\cup C'_{n-1,1}$
is contained in $\sigma(\tilde X')$ with multiplicity exactly one because they are included only in the component $S_1^-$.
This means that the least multiplicity for generators of the system $|\bm L'|$ along components of all chains of base cuves is exactly one.
Hence, defining a line bundle 
$\bm L\uptwo$ on $\tilde Z\uptwo$ as the pullback $\mu_2^*\bm L'$ minus the sum of all exceptional divisors of $\mu_2$ with multiplicity one,
then the linear system $|\bm L\uptwo\,|^{\CC^*_s}$ has no fixed component.

Since the multiplicities of the two components 
$E_{n-1}$ and $E_{n+1}$ for the generators \eqref{gen4} are both one,
their transformations into $\tilde Z\uptwo$ 
do not contain $E\uptwo_{n-1,j}$ ($1\le j<n$) and $E\uptwo_{n+1,j}$ ($n+1<j\le 2n$) as components. 
Using this, again by calculating the intersection of the generators of $|\bm L\uptwo|^{\CC^*_s}$, we can see that, if we define curves in $\tilde Z\uptwo$ by $C\uptwo_{i,j}:=E\uptwo_{i,j}\cap S^-_{j}$ for $j< n$ and
 $C\uptwo_{i,j}:=E
 \uptwo_{i,j}\cap S^+_j $ for $j> n$, then 
the base locus of $|\bm L\uptwo\,|^{\CC^*_s}$ consists of the chains
\begin{align}\label{ch1}
C\uptwo_{i,i}\cup C\uptwo_{i+1,i}\cup\dots\cup C\uptwo_{n-2,i}\quad
i< n-1, 
\end{align}
which is one shorter than the string \eqref{exe01}, and
\begin{align}\label{ch2}
C\uptwo_{i-1,i}\cup C\uptwo_{i-2,i}\cup\dots\cup C\uptwo_{n+2,i}\quad
i>n+2,
\end{align}
which is one shorter than the string \eqref{exe02},
and of course the complex conjugate of these chains.
Note that none of the nodes of $\tilde Z\uptwo$ belong to these chains.
Note also that the divisor $E\uptwo_{i,j}$ intersects $S_m^-$ or $S_m^+$ only when $m=j$.
The chains \eqref{ch1} and \eqref{ch2} are again $T^2_{\CC}$-invariant.
Thus, {\em by the blowup $\mu\uptwo:\tilde Z\uptwo\lras \tilde Z'$, new base curves appear on $S^+_j$ and $S^-_j$ again as chains, but with length one shorter than the previous chain, with the lost component being at the side of the central component.}
(See Figure \ref{f:web} for the case $n=4$ and $i=1$.)
In particular, the base curves on $E_{n-1}$ and $E_{n+1}$
(or equivalently, those on $S^-_{n-1}$ and $S^+_{n+2}$), both of which consist of a single curve, are eliminated by $\mu_2$.
In particular, the base locus is completely eliminated on $\tilde Z\uptwo$ in the case $n=2$.

We can repeat this procedure of blowup until there is no base curve.
In fact, because on $\tilde Z'$ the longest chains of base curves lies on $S^+_1\cup S^-_1$ and $S_{2n}^+\cup S_{2n}^-$, this process of blowup finishes when the base curves on these components are eliminated, and since the length of the longest chain is exactly $(n-1)$ as in \eqref{cbc1} and \eqref{cbc8}, we finish the process of elimination when we blowup $(n-1)$ times from $\tilde Z'$; namely, in the above notation, when we reach the space $\tilde Z^{(n)}$.
The exceptional divisor on $\tilde Z^{(n)}$ that lies over the curve $C'_{i,j}\subset \tilde Z'$ is a string of divisors which is explicitly given by
\begin{align}\label{exe1}
E\uptwo_{i,j}\cup E^{(3)}_{i,j}\cup\dots\cup E^{(n-i+1)}_{i,j}&
{\text{ if $1\le j\le i < n$}},\\
E\uptwo_{i,j}\cup E^{(3)}_{i,j}\cup\dots\cup E^{(i-n+1)}_{i,j}&
{\text{ if $n<i\le j\le 2n$}}.\label{exe2}
\end{align}
These are strings of divisors which grow ``orthogonally'' to the strings of the exceptional divisors of each blowup $\mu_m$.

In each step of the blowup $\mu_m:\tilde Z\upm\lras\tilde Z^{(m-1)}$
all the exceptional divisors $E\upm_{i,j}$ are ruled surfaces, but one of the end components is blown up through the next blowup $\mu_{m+1}:\tilde Z^{(m+1)}\lras\tilde Z\upm$ because an end component of the exceptional divisor of 
$\mu_m$ intersects an end component of the center of $\mu_{m+1}$ transversally at one point. (See Figure \ref{f:web}.)
So in $\tilde Z\upnn$, each component is not necessarily isomorphic to a ruled surface and some of them are blown-up ruled surfaces.

Thus, the base locus of the linear system $|\bm L|^{\CC^*_s}$ on $\tilde Z$ is completely eliminated on $\tilde Z^{(n)}$.
This space has many nodes (if $n>2$) because we have blown up chains of curves and did not resolve the resulting nodes at all.
We also note that in the blowups $\mu_3,\mu_4,\dots,\mu_n$, no change occurs on the central components $E_n$ and $\ol E_n$ because all blown-up chains do not intersect these components.
So from Proposition \ref{p:El}, in $\tilde Z^{(n)}$ also, 
the divisors $E_n$ and $\ol E_n$ are isomorphic to $\tilde{\ms T}$.
Thus, {\em the (blown up) minitwistor space $\tilde{\ms T}$ appears as a component of the exceptional divisors of the blowups for eliminating the indeterminacy locus of the meromorphic quotient map $\tilde{\Psi}:\tilde Z\lras \ms T$.}

Now it is not difficult to prove the following.
Let $\tilde\Psi\upnn:\tilde Z\upnn\lras \ms T$ be the composition $\tilde\Psi\circ\mu_1\circ\mu_{2}\circ\dots\circ \mu_n$. 
This is the elimination of the indeterminacy of $\tilde\Psi$ and hence has no base point.

\begin{proposition}\label{p:mr}
The restriction of the holomorphic map 
$\tilde\Psi\upnn:\tilde Z\upnn\lras \ms T$ to (the strict transforms of) the central components $E_n$ and $\ol E_n$ may also be identified with the minimal resolution of $\ms T$.
\end{proposition}

\proof
From reality, it suffices to prove for $E_n$.
Since every blowup appearing after the space $\tilde Z\uptwo$ does not affect $E_n$ as above, 
it is enough to prove the same assertion for $\tilde\Psi\uptwo:=\tilde\Psi\circ\mu_1\circ\mu_2:\tilde Z\uptwo\lras\ms T$.
In the argument for the elimination of the base locus, we obtained concrete generators of the linear system $|\bm L\uptwo|^{\CC^*_s}$, and it is immediate to express their restrictions to $E_n\subset\tilde Z\uptwo$ explicitly.
From this, we obtain that the restrictions are the same as the restrictions of the generators of $|\bm L|^{\CC^*_s}$ obtained in Proposition \ref{p:gen} to the fiber $\tilde{\ms T} = \tilde p\inv(1)$.
As in the final part of the proof of Proposition \ref{p:ZT}, these restrictions induce the contraction mapping of the two $(-n)$-curves
on $E_n\simeq\tilde{\ms T}$ and this is identified with the minimal resolution of $\ms T$.
\proofend

\medskip
Let $\tilde p\upnn:\tilde Z\upnn\lras\PP^1$ be the composition $\tilde p\circ \mu_1\circ\dots\circ\mu_n$. Since all the modifications to obtain $\tilde Z\upnn$ from $\tilde Z$ are done in the fibers $\tilde p\inv(0)$ and $\tilde p\inv(\infty)$.
we have a holomorphic isomorphism $(\tilde p\upnn)\inv (\CC^*)\simeq\CC^*\times\tilde{\ms T}$ by the $\CC^*_s$-action.
From the previous proposition, the remaining fibers (over $0$ and $\infty$) contain 
the component which is biholomorphic to
$\tilde{\ms T}$.
From these, it would be natural to expect that the space $\tilde Z\upnn$ is bimeromorphic to the product $\tilde{\ms T}\times\PP^1$ through some explicit bimeromorphic changes.
In the next two subsections, we show that this is really the case.

\subsection{Modifications using the nodes}
In this subsection, we apply some bimeromorphic changes to $\tilde Z\upnn$, using all its nodes.
%The operations are illustrated in Figure \ref{f:web2} again in the case $n=4$ for the fiber over the point $z=a_1$.

\begin{figure}
\includegraphics{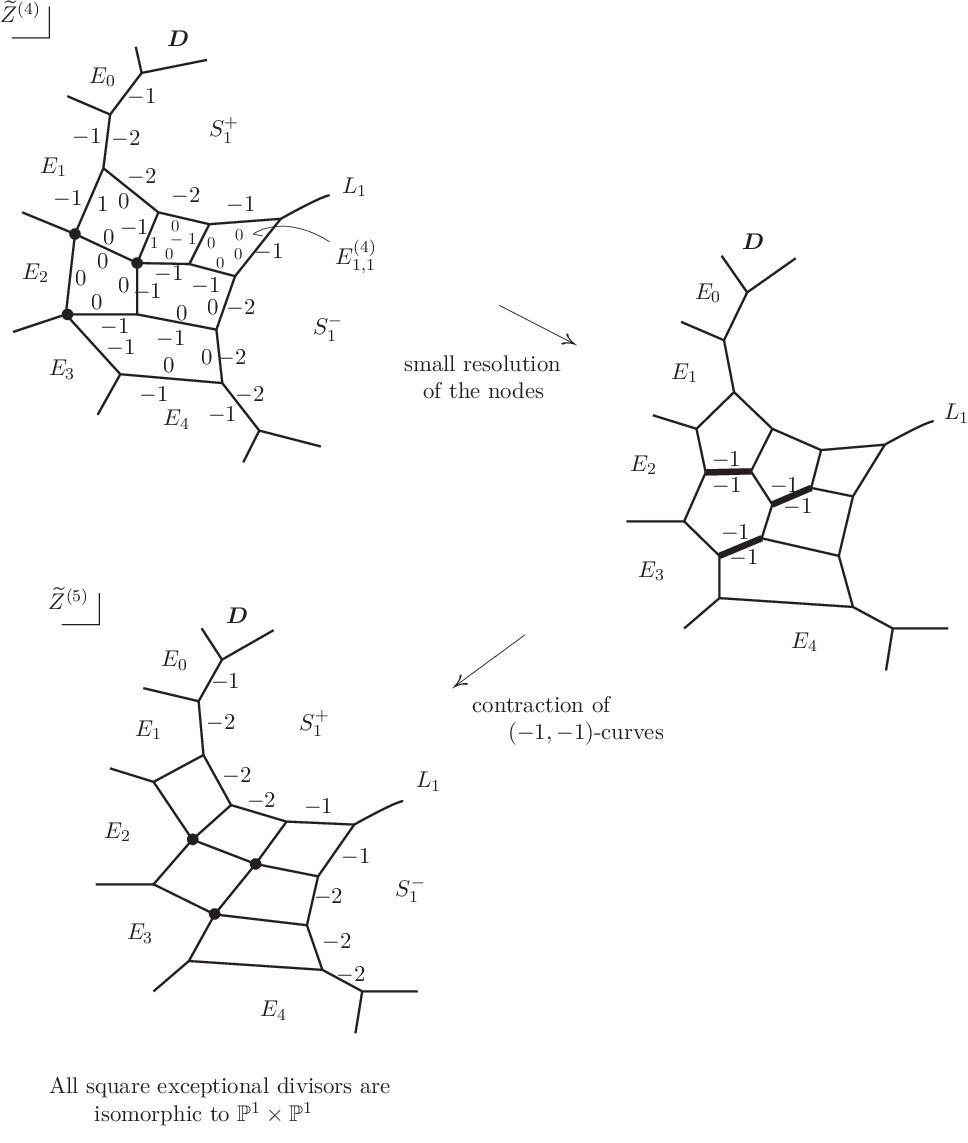}
\caption{The transformations using the nodes.
}
\label{f:web2}
\end{figure}

We recall that each node of $\tilde Z\upnn$ arises by the blowup $\mu_m:\tilde Z\upm\lras
\tilde Z^{(m-1)}$ for some uniquely determined number $1<m< n$,
where the case $m=n$ is excluded because the center of $\mu_n$ consists of smooth rational curves, 
and the node appears over a singularity of the chain
which is a connected component of the base locus of $|\bm L^{(m-1)}|^{\CC^*_s}$. 
Using the lift of the map $\tilde f':\tilde Z'\lras\PP^1$ in \eqref{bd} to $\tilde Z\upnn$ which is given by $\tilde f\upnn:= \tilde f'\circ \mu_1\circ\dots\circ\mu_n$, the nodes belong to the reducible fibers $(\tilde f\upnn)\inv(a_j)$ with
$j=1,2,\dots,2n$, but the fibers over the two points $z=a_{n+1}, a_{n+2}$ have no node because, as we have already mentioned, only the blowup $\mu_2:\tilde Z\uptwo\lras\tilde Z$ changes the fibers over these points and further, all the blowup centers on these two fibers are irreducible (and smooth).
In particular, in the case $n=2$, no change will be made.
In the following, we discuss only the nodes
that lie over the chain $C_0\cup C_1\cup\dots \cup C_{2n}\subset Z$, because the operation for the nodes lying over the conjugate chain is automatically determined from the former by the real structure.
The operations below are illustrated in Figure \ref{f:web2} in the case $n=4$ for the fiber over the point $z=a_1$.

Any singularity of the chains in $\tilde Z\upm$ lying over $C_0\cup\dots\cup C_{2n}$ can be written $C^{(m)}_{i,j}\cap C^{(m)}_{i+1,j}$ for some $i,j$ and $m$.
This node belongs to the intersection curve
$E\upm_{i,j}\cap E\upm_{i+1,j}$, which is the inverse image of $C^{(m-1)}_{i,j}\cap C^{(m-1)}_{i+1,j}$.
If $i<n$, then we take the small resolution of the node $C^{(m)}_{i,j}\cap C^{(m)}_{i+1,j}$ which blows up the component $E\upm_{i,j}$, and 
if $i>n$, then we take the small resolution of the node which blows up another component $E\upm_{i+1,j}$.
Then the symmetry about the central component $E_n$ is preserved.
This small resolution makes
the intersection curve $E\upm_{i,j}\cap E\upm_{i+1,j}$ a $(-1,-1)$-curve 
on the small resolution.
Namely, it becomes a smooth rational curve with normal bundle $\ms O(-1)\oplus\ms O(-1)$. Hence, we can contract it to a node.
Let $\tilde Z^{(n+1)}$ be the threefold obtained by applying this change (of small resolution and contraction of the $(-1,-1)$-curve) to all these nodes as well as all their images by the real structure.
The number of the nodes is preserved from $\tilde Z\upnn$ to $\tilde Z^{(n+1)}$.
As we have mentioned, not all components $E\upm_{i,j}$ in $\tilde Z\upnn$ are isomorphic to a ruled surface, but {\em the transformation from $\tilde Z\upnn$ to $\tilde Z^{(n+1)}$ makes all these components isomorphic to $\qdr$}.
The real structure and the $T^2_{\CC}$-action is succeeded on $\tilde Z^{(n+1)}$.

\subsection{Blowdown to $\tilde{\ms T}\times\PP^1$}
In this subsection, we show that the variety $\tilde Z^{(n+1)}$ can be successively blowdown to the product $\tilde{\ms T}\times\PP^1$.
The operations are illustrated in Figure \ref{f:web3} again in the case $n=4$ for the fiber over the point $z=a_1$.

\begin{figure}
\includegraphics{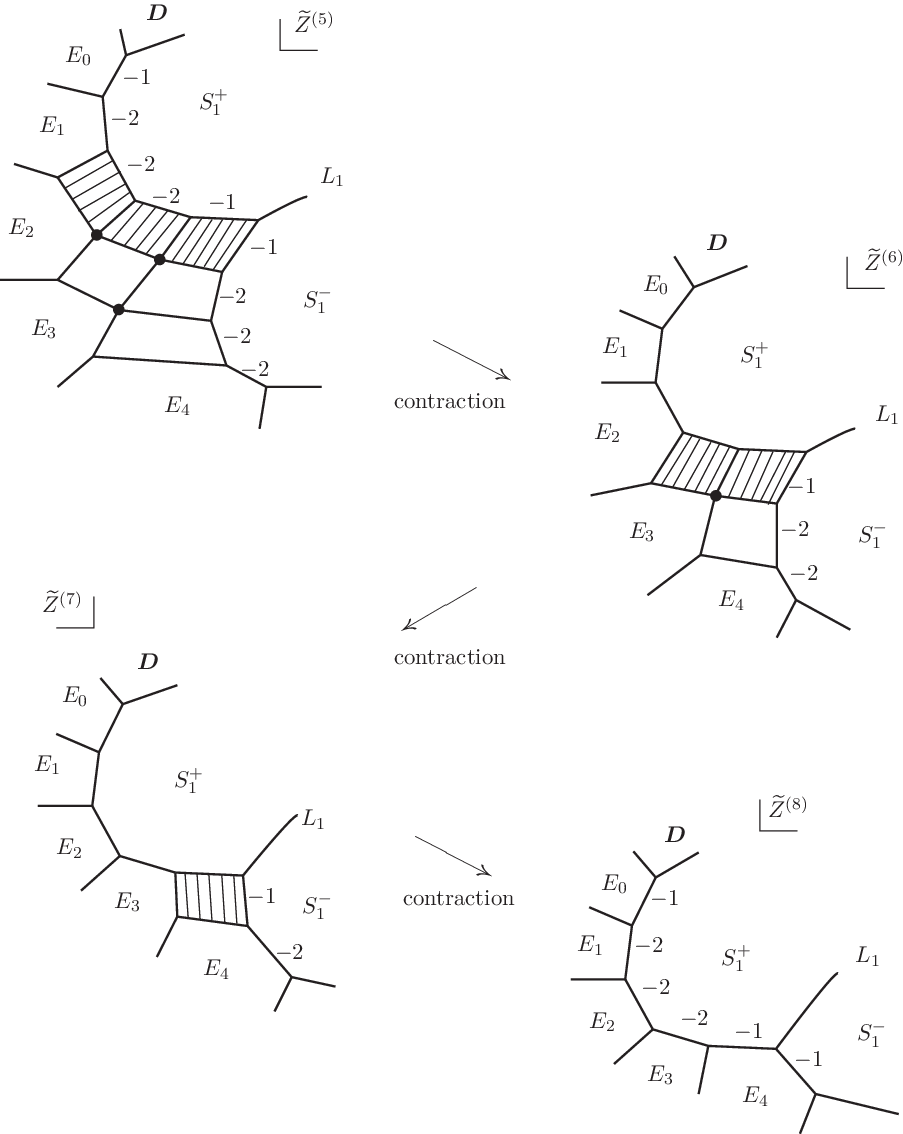}
\caption{Blowdown to $\tilde Z^{(2n)}$}
\label{f:web3}
\end{figure}

For this, among many of these $\qdr$ in $\tilde Z^{(n+1)}$, we take a look at the string of components 
that are over the curve $C'_{i,i}\subset \tilde Z'$, for each $i<n$ and $C'_{i,i+1}\subset \tilde Z'$ for $i>n$.
These curves are one of the end components of the chain of base curves on $\tilde Z'$ given in Proposition \ref{p:bs1}, and they are written as the first component in \eqref{cbc1}--\eqref{cbc4} for the former and \eqref{cbc5}--\eqref{cbc8} for the latter.
These strings of the exceptional divisors 
are already shown in \eqref{exe1} and \eqref{exe2} 
putting $j=i$ for the former and $j=i+1$ for the latter.
In particular, they consist of $|n-i|$ components.
The former string (resp.\,the latter string) faces $S^+_i$ (resp.\,$S^-_i$) in the sense that each component
intersects $S_i^+$ (resp.\,$S^-_i$) in a curve.

As above, for every $i,j,m$, the transformation of the component $E\upm_{i,j}\subset \tilde Z\upnn$ into $\tilde Z^{(n+1)}$ 
is isomorphic to $\qdr$.
By calculating the normal bundles of these divisors in 
$\tilde Z^{(n+1)}$, we can see that 
{\em the string \eqref{exe1} with $j=i$
can be simultaneously blown down to a chain of rational curves in $S^+_i$ of the same length}. (See Figure \ref{f:web3}.)
From the symmetry about the central component $E_n$, 
{\em the string \eqref{exe2} with $j=i+1$ can also be blown down to a chain of rational curves in $S_i^-$ of the same length.}
By the real structure, the conjugate of these two strings
can also be blown down to chains of rational curves in $S_i^+$ or $S_i^-$.
Let $\tilde Z^{(n+2)}$ be the variety obtained by blowing down these four strings of $\qdr$. 
The real structure and the $T^2_{\CC}$-action descend on $\tilde Z^{(n+2)}$.
This blowdown does not change the structure of the remaining $\qdr$ and therefore they remain to be isomorphic to $\qdr$. 
Among these $\qdr$ in $\tilde Z^{(n+2)}$, if $i<n$, then
the ones lying over the curve $C'_{i+1,i}\subset \tilde Z'$ 
(i.e.,\,the second components in \eqref{cbc1}--\eqref{cbc3})
constitute a string of length $(n-i-1)$, which is facing $S_i^+$, just like the former string over $C'_{i,i}$.
By calculating the normal bundles, this string of divisors can also be simultaneously blown down to a chain of rational curves in $S_i^+$ of the same length. (See Figure \ref{f:web3}.)
By the symmetry about the central component,
if $i>n$, then the string of $\qdr$ lying over $C'_{i,i+2}$ 
(i.e.,\,the second components in \eqref{cbc6}--\eqref{cbc8})
can also be blown down to a chain of rational curves in $S_i^-$ of the same length.
By the real structure, the string over $\ol C'_{i+1,i}$ for $i<n$
and the string over $\ol C'_{i,i+2}$ for $i>n+1$ can also be blown down to chains of rational curves in $S_i^-$ and $S_i^+$ respectively.
Let $\tilde Z^{(n+3)}$ be the blowdown of these four strings of $\qdr$ in $\tilde Z^{(n+2)}$.
Again, $\tilde Z^{(n+3)}$ remains to have a $T^2_{\CC}$-action and the real structure.

We can repeat this blowdown process for the strings of $\qdr$ that face $S_i^+$ or $S_i^-$ until all the $\qdr$ are blown down to rational curves in $S_i^+$ or $S_i^-$.
Again, this process finishes in $(n-1)$ times, so we denote
$\tilde Z^{(2n)}$ for the resulting space, so that in this space, no $\qdr$, namely no component of the exceptional divisor of 
$\tilde Z^{(n)}\lras \tilde Z'$, remains.
Note that {\em all nodes of $\tilde Z^{(n+1)}$ disappear in this process and the space $\tilde Z^{(2n)}$ is smooth.}

The manifold $\tilde Z^{(2n)}$ remains to have a $T^2_{\CC}$-action and a real structure, as well as holomorphic maps $\tilde f^{(2n)}:\tilde Z^{(2n)}\lras\PP^1$ and 
$\tilde\Phi^{(2n)}:\tilde Z^{(2n)}\lras\FF'_2$
which are naturally induced from
the holomorphic maps $\tilde f':\tilde Z'\lras\PP^1$ in the diagram \eqref{bd}
and the projection $\tilde\Phi':\tilde Z'\lras\FF'_2$ respectively.
Since all the operations from $\tilde Z'$ to $\tilde Z^{(2n)}$ are done in the fibers of $\tilde f'$ over the $2n$ points $z=a_1,\dots,a_{2n}$, 
the two spaces $\tilde Z'$ and $\tilde Z^{(2n)}$ are biholomorphic to each other outside these fibers.
Next, we take a look at the effect of the transformations on the components $S_i^+$ and $S_i^-$ of these reducible fibers.
In the process from $\tilde Z'$ to $\tilde Z\upnn$, 
if $i\le n$ (resp.\,if $i\ge n+1$), then the divisor $S^+_i$ is blown up $(n-i)$ times (resp.\,$(i-n-1)$ times) and each blowup is done at the intersection of the twistor line $L_i$ with the newest exceptional divisor of the blowup,
which is over $C_1\cup\dots\cup C_{2n}$ (resp.\,over $\ol C_1\cup\dots\cup \ol C_{2n}$); see Figure \ref{f:web}.
In particular, the two divisors $S_n^+$ and $S_{n+1}^+$ receive no effect from the blowups up to $\tilde Z\upnn$.
The effect on the other components $S_i^-$ is known just by taking the complex conjugate.

In the next process from $\tilde Z\upnn$ to $\tilde Z^{(n+1)}$,
$S_i^+$ and $S_i^-$ receive no effect for any $i$. (See Figure \ref{f:web2}.)
In the subsequent process from $\tilde Z^{(n+1)}$ to $\tilde Z^{(2n)}$, 
if $i\le n$ (resp.\,if $i\ge n+1$), then $S_i^-$ is blown down $(n-i)$ times (resp.\,$(i-n-1)$ times) and each of the blowdown contracts the $(-1)$-curve which intersects $L_i$ and 
which is over $C_1\cup\dots\cup C_{2n}$ (resp.\,$\ol C_1\cup\dots\cup \ol C_{2n}$); see Figure \ref{f:web3}.
By reality, analogous blowdowns occur for $S_i^+$.
In particular, again, the divisors $S_i^+$ and $S_i^-$ receive no changes for $i=n,n+1$.

From these, it is not difficult to see that {\em 
in the space $\tilde Z^{(2n)}$, all reducible fibers $S_i^++ S_i^-$ ($1\le i\le 2n$) are biholomorphic to each other.} 
Therefore, as a whole, recalling that $E_i$ and $\ol E_i$ were the exceptional divisors over the curves $C_i$ and $\ol C_i$ in $\tilde Z$ respectively, it would be possible to say:

\begin{quote} 
In the initial space $\tilde Z'$, for every index $1\le i\le 2n$, the $T^2_{\CC}$-invariant twistor line $L_i$  is a `bridge' that connects the components $E_i$ and $\ol E_i$, and in the last space $\tilde Z^{(2n)}$, $L_i$ is a `bridge' that connects the central components $E_n$ and $\ol E_n$, for any index $i$.
(Compare $\tilde Z'$ in Figure \ref{f:web} and $\tilde Z^{(8)}$ in Figure \ref{f:web3}.)
\end{quote}

The structure of $E_i$ and $\ol E_i$ ($0\le i\le 2n)$ in the space
$\tilde Z^{(2n)}$ can also be seen from the series of procedures, and we obtain that if $i\neq n$, then they are isomorphic to $\qdr$, whose two factors can be identified with the base of the projection
$\tilde f^{(2n)}:\tilde Z^{(2n)}\lras\PP^1$ and the original curve $C_i$ or $\ol C_i$.
For the central divisors, $E_n$ and $\ol E_n$ in $\tilde Z^{(2n)}$ are isomorphic to $\tilde{\ms T}$ because in $\tilde Z\upnn$ they were isomorphic to $\tilde{\ms T}$ and  the change from $\tilde Z\upnn$ to $\tilde Z^{(n+1)}$ and the blowdowns
from $\tilde Z^{(n+1)}$ to $\tilde Z^{(2n)}$ do not change the structure of $E_n$ and $\ol E_n$ at all.
On $\tilde Z^{(2n)}$, the intersection $(S_j^+\cup S_j^-)\cap E_i$ has the following structure:
\begin{itemize}\setlength{\itemsep}{-2pt}
  \item if $i=0,2n$, then it is a $(-1)$-curve on $S_j^+$ or $S_j^-$,
  \item if $i\neq 0,n,2n$, then it is a $(-2)$-curve on $S_j^+$ or $S_j^-$,
  \item  if $i=n$, then it consists of
two components, one of which is a $(-1)$-curve in $S_j^+$ and another one is a $(-1)$-curve in $S_j^-$,
\end{itemize}
(See Figure \ref{f:web3}.) This means that the mutually adjacent components $E_0,E_1,\dots, E_{n-1}\simeq\qdr$ can be successively blown down to curves in this order in the direction of fibers of $\tilde f^{(2n)}$, and similarly for 
the components $E_{2n},E_{2n-1},\dots, E_{n+1}\simeq\qdr$.
Let $\tilde Z^{(2n)}\lras \tilde Z^{(2n+1)}$ be the composition of all these blowdowns and their complex conjugate.
The variety $\tilde Z^{(2n+1)}$ has no exceptonal divisor of the form $E_i$ and $E^{(m)}_{i,j}$
except for $E_n$ and $\ol E_n$, and is also smooth, admitting a $T^2_{\CC}$-action and the real structure.
The maps $\tilde f^{(2n)}$ and $\tilde\Phi^{(2n)}$ descends to holomorphic maps
$\tilde f^{(2n+1)}:\tilde Z^{(2n+1)}\lras\PP^1$ and 
$\tilde\Phi^{(2n+1)}:\tilde Z^{(2n+1)}\lras\FF'_2$, and recalling that $e$ and $\ol e$ are the exceptional curves of the blowup $\FF'_2\lras\FF_2$,
we have $(\tilde\Phi^{(2n+1)})\inv(e) = E_n$
and $(\tilde\Phi^{(2n+1)})\inv(\ol e) = \ol E_n$.
Further, $(\tilde f^{(2n+1)})\inv(a_i) = S_i^++ S_i^-$ for any $i$ but 
$S_i^+$ and $S_i^-$ are biholomorphic to $\qdr$ as a consequence of the contraction of $E_i$ and $\ol E_i$ with
$i\neq n$.

For any $m\le 2n+1$, let $\tilde p^{(m)}:\tilde Z^{(m)}\lras\PP^1$ be the holomorphic map induced from the map $\tilde p:\tilde Z\lras\PP^1$ that takes the $u$-coordinate.
The map $\tilde p'=\tilde p\upone:\tilde Z'\lras\PP^1$ has reducible fibers exactly over the two points $u=0,\infty$, and they are
$\tilde Z_0+\sum_{i=0}^{2n}E_i$ and $\tilde Z_{\infty}+\sum_{i=0}^{2n}\ol E_i$ respectively.
The map $\tilde p^{(2n+1)}$ also has reducible fibers exactly over
$u=0,\infty$, but as a consequence of the blowdowns from $Z^{(n+1)}$ to $Z^{(2n+1)}$ they are $\tilde Z_0 + E_n$
and $\tilde Z_{\infty} + \ol E_n$ respectively, and in $\tilde Z^{(2n+1)}$, the components $\tilde Z_0$ and $\tilde Z_{\infty}$ are also isomorphic to $\qdr$.
Moreover, with the help of the intersection with (the transformation of) the divisors $\bm D$
or $\ol{\bm D}$, we readily obtain that the normal bundle
of $\tilde Z_0$ and $\tilde Z_{\infty}$ in $\tilde Z^{(2n+1)}$ is isomorphic to $\ms O(-1,0)$, where $\ms O(1,0)$ refers to the fiber class of 
$\tilde\Phi^{(2n+1)}|_{\tilde Z_0}:\tilde Z_0\lras\PP^1$ and 
$\tilde\Phi^{(2n+1)}|_{\tilde Z_{\infty}}:\tilde Z_{\infty}\lras\PP^1$.
Therefore, these two divisors can be blown down in another direction.
Let $\tilde Z^{(2n+1)}\lras \tilde Z^{(2n+2)}$ be these blowdowns.
$\tilde Z^{(2n+2)}$ is still smooth and has a $T^2_{\CC}$-action and the real structure. We then have:

\begin{proposition}\label{p:prod}
%As before, let $\tilde{\ms T}$ be the minimal resolution of the minitwistor space $\ms T$.
The manifold $\tilde Z^{(2n+2)}$ is biholomorphic to $\tilde{\ms T}\times\PP^1$.
\end{proposition}

\proof
For simplicity, we denote $Y$ for the space $\tilde Z^{(2n+2)}$ and 
$Y_u$ for the fiber of the map $\tilde p^{(2n+1)}:Y\lras\PP^1$ over $u\in\PP^1$. 
For any point $u\in\PP^1\minus\{0,\infty\}$, $Y_u$ is biholomorphic to the fiber $\tilde Z_u$ of the original map $\tilde p:\tilde Z\lras \PP^1$.
Therefore $Y_u\simeq\tilde{\ms T}$ for any $u\in\PP^1\minus\{0,\infty\}$.
For the fiber over $u=0$, as we have already noted, 
the divisor $E_n$ in $\tilde Z^{(2n)}$ is biholomorphic to $\tilde{\ms T}$.
Also, from the above description of the last two blowdowns
$\tilde Z^{(2n)}\lras\tilde Z^{(2n+1)}\lras \tilde Z^{(2n+2)}$, 
these do not change the structure of $E_n$ and $\ol E_n$.
Hence, the fibers $Y_0$ and $Y_{\infty}$ are also isomorphic to $\tilde{\ms T}$.
Thus, all fibers of $\tilde p^{(2n+2)}:Y\lras\PP^1$ are biholomorphic to $\tilde{\ms T}$. Therefore, by a theorem of Fischer-Grauert \cite{FG}, $Y$ is a holomorphic $\tilde{\ms T}$-bundle over $\PP^1$.

Further, $Y$ admits a $\CC^*_s$-action which is just the resriction of the $T^2_{\CC}$-action to the subgroup $\CC^*_s$, and this $\CC^*$-action covers the standard $\CC^*$-action on $\PP^1$ that fixes $0$ and $\infty$.
Therefore, there is a holomorphic isomorphism $Y\minus(Y_0\cup Y_{\infty})\simeq\tilde{\ms T}\times\CC^*$.
On the other hand, since the central spheres $C_n,\ol C_n\subset Z$ are sink and source of the $\CC^*_s$-action \cite{Hi21},
$\CC^*_s$ acts trivially on the exceptional divisor $E_n$ and $\ol E_n$.
Hence, the isomorphism $Y\minus(Y_0\cup Y_{\infty})\simeq\tilde{\ms T}\times\CC^*$ extends to an isomorphism $Y\simeq\tilde{\ms T}\times\PP^1$, as required.
\proofend

\medskip
The product of the minimal resolution map and the identity map gives a holomorphic map $\tilde{\ms T}\times\PP^1\lras\ms T\times\PP^1$.
The isomorphism $\tilde Z^{(2n+2)}\lras\tilde{\ms T}\times\PP^1$
in the previous proposition, followed by this product map gives a holomorphic map $\tilde Z^{(2n+2)}\lras\ms T\times\PP^1$.
Obviously, this is also bimeromorphic.
By composition with the bimeromorphic map between $\tilde Z$ to $\tilde Z^{(2n+2)}$, we obtain a bimeromorphic identification between $\tilde Z$ and $\ms T\times\PP^1$.
From the argument so far, under this identification, the meromorphic quotient map
$\tilde\Psi:\tilde Z\lras\ms T$ can be regarded as just a projection to a factor:
\begin{proposition}\label{p:pr}
There exists the following commutative diagram of meromorphic maps:
\begin{equation}\label{cdm}
\begin{tikzcd}
\tilde Z \arrow[r,"\tilde\Psi"] \arrow[d] &
\ms T \\
\tilde {\ms T}\times\PP^1 \arrow[ur,swap,"\pr'"]&
\end{tikzcd}
\end{equation}
where the vertical map is the bimeromorphic transformation which is the composition of all bimeromorphic changes from $\tilde Z$ to $\tilde Z^{(2n+2)}\simeq \tilde {\ms T}\times\PP^1$,
and $\pr'$ is the composition of the projection to the $\tilde{\ms T}$-factor and the minimal resolution map $\tilde{\ms T}\lras\ms T$.
\end{proposition}

\iffalse
\proof
From the $\CC^*_s$-action on $\tilde Z$, there exists a holomorphic isomorphism $\tilde Z\minus (\tilde Z_0\cup \tilde Z_{\infty})
\simeq \tilde {\ms T}\times \CC^*$. 
Let $\tilde\Psi\uc$ be the restriction of $\tilde\Psi$ to the subset 
$\tilde Z \minus (\tilde Z_0\cup \tilde Z_{\infty})$.
So we may consider the following diagram of holomorphic maps
\begin{equation}\label{cdm2}
\begin{tikzcd}
\tilde Z \minus (\tilde Z_0\cup \tilde Z_{\infty})
\arrow[r,"\tilde\Psi\uc"] \arrow[d,"\simeq"] &
\ms T \\
\tilde {\ms T}\times\CC^* \arrow[ur,swap,"\pr'"]&
\end{tikzcd}
\end{equation}
From Proposition \ref{p:ZT}, $\tilde\Psi\uc$ restricts to the minimal resolution map
$\tilde Z_1=\tilde p\inv(1)=\tilde{\ms T}\lras\ms T$.
Since all blowups and blowdowns from $\tilde Z$ to 
$\tilde Z^{(2n+2)}$ are done in the two fibers $\tilde Z_0\cup \tilde Z_{\infty}$, the vertical map can be identified with the restriction of the bimeromorphic map from $\tilde Z$ to 
$\tilde Z^{(2n+2)}$ to the complement of $\tilde Z_0\cup \tilde Z_{\infty}$.
It follows from these that the restriction of the diagram \eqref{cdm2} to the fiber $\tilde Z_1$ is commutative if we replace $\tilde{\ms T}\times\PP^1$ by $\tilde{\ms T}\times\{1\}$. 
Therefore, since both the vertical map and $\Psi'$ are $\CC^*_s$-equivariant, the diagram \eqref{cdm2} is commutative.
Hence, the diagram \eqref{cdm}, which restricts to \eqref{cdm2}, is also commutative.
\proofend
\fi

\section{The images of twistor lines}\label{s:mtl}
We recall from Proposition \ref{p:bs} that the indeterminacy locus of the meromorphic quotient map $\tilde\Psi:\tilde Z\lras\ms T$
consists of the chain $C_1\cup\dots\cup C_{2n-1}$ (considered as a subset of the fiber $\tilde Z_0 = \tilde p\inv(0)$) and their conjugate.
In this section, we investigate the images of twistor lines 
that meet a slightly longer chain $C_0\cup C_1\cup\dots\cup C_{2n}$,
under the map $\tilde\Psi$.
If a twistor line $L\subset Z$ intersects the end component $C_0$ or $C_{2n}$, then $\tilde\Psi$ has no indeterminacy on $L$ and the image $\tilde\Psi(L)$ makes sense as usual.
But if $C$ intersects $C_1\cup\dots \cup C_{2n-1}$, then $\tilde\Psi$ has a pair of indeterminacy points on $L$ and the image has to be in the bimeromorphic sense, which will be explained later.
We will show that all these images are curves in $\ms T$ which can be explicitly described.

Since the map $\tilde\Psi$ is a quotient map under the $\CC^*_s$-action on $\tilde Z$ and elements of its $S^1$-subgroup map a twistor line to a twistor line, the images of twistor lines through the same $S^1$-orbit are all equal.
Since $\CC^*_s$ acts non-trivially on every $C_i$ except the central sphere $C_n$, this implies that the images of twistor lines through a component $C_i$ with $i\neq n$ constitute a family parameterized by an interval.
In contrast, the images of twistor lines through $C_n$ will be curves parameterized by $C_n\simeq S^2$ itself.

We denote $f_{\lmd}$ for the fiber of the projection $\tilde{\ms T}\lras\PP^1$ over the point $z=\lmd\in(\PP^1)^{\sigma}$.
Recall from Section \ref{ss:itl} that for each index $1\le i\le 2n$, $L_i$ denotes the twistor line through $C_{i-1}\cap C_i$, which is defined by $x=y=z-a_i=0$ in the projective model.

We begin with the case where a twistor line $L\subset Z$ intersects one of the two ``axes''; namely the end components $C_0$ and $C_{2n}$.
In this case, as above, the image $\tilde\Psi(L)$ makes sense as usual unless $L=L_1$ or $L_{2n}$.

\begin{proposition}\label{p:a}
Let $L\subset Z$ be a twistor line that intersects the axis $C_0$ (resp.\,the axis $C_{2n}$)
but not equal to $L_1$ (resp.\,$L_{2n}$).
Let $\lmd\in (-\infty,a_1)$ (resp.\,$\lmd\in (a_{2n},\infty)$) be
the real number such that $\Phi(L) = \ms C_{\lmd}$ (see \eqref{iml}).
Then 
$
\tilde\Psi(L) = f_{\lmd}.
$
Further, the map $\tilde\Psi|_L:L \lras f_{\lmd}$ is $n:1$. 
\end{proposition}

\proof
Let $G_i\simeq\CC^*$ be the stabilizer subgroup of $T^2_{\CC}$ along the component $C_i$.
The map $\tilde\Psi:\tilde Z\lras\ms T$ is $T^2_{\CC}$-equivalent and hence $G_i$-equivariant. Up to a finite subgroup of $\CC^*$, if $i\neq n$, then the $G_i$-action on $\ms T$ is identified with the $\CC^*$-action which 
is given by $(x,y,z)\longmapsto (tx,t\inv y, z)$ from \eqref{act-t}, which preserves each fiber of the rational map $\ms T\lras\PP^1$ that takes the $z$-coordinate.
Hence, the image $\tilde\Psi(L)$ has to be a fiber of this map, and the fiber has to be irreducible and real since $L$ is so.
To show that this fiber is exactly $f_{\lmd}$ as in the proposition,
consider the intersection point $q:=L\cap \tilde Z_1$, where $\tilde Z_1=\tilde p\inv(1)$.
Since $\Phi(L)=\ms C_{\lmd}=\{z=\lmd u\}$, the $z$-coordinate of the point $q$ is $\lmd$.
So $q\in f_{\lmd}$, which means $\tilde\Psi(L)=f_{\lmd}$.
 
As in \eqref{LL}, $\bm L\cdot L = 2n$.
Since $\tilde{\Psi}(L)$ is the image in the usual sense and it is a conic as above, this means that the map $\tilde\Psi|_L:L\lras f_{\lmd}$ is of degree $2n/2 = n$, as required.
\proofend

\medskip
Recall that $\pi:\PP^{n+1}\lras\PP^n$ and $\Pi:\PP^{n+2}\lras\PP^{n+1}$ denote the projection from a point, which restricts to the double covering ${\ms T}\lras {\rm C}(\Lmd)$ and the rational projection $\ms T\lras\Lmd$, respectively. (See the diagrams \eqref{cdm0}.)
Regarding $\Psi(\tilde L)$ as a multiple curve $nf_{\lmd}$ from the previous proposition, we immediately obtain the following.

\begin{corollary}\label{c:im1}
Let $L$ and $\lmd$ be as in the previous proposition, and
$h\subset\PP^{n+1}$ the inverse image of
the osculating hyperplane of $\Lmd\subset\PP^n$ at the point 
$\lmd$, under the projection $\pi:\PP^{n+1}\lras\PP^n$ .
We then have
$$\tilde\Psi(L)=\Pi\inv\big(h\cap {\rm C}(\Lmd)\big).$$
\end{corollary}

In general, if $\Phi:X\lras Y$ is a meromorphic mapping between compact complex spaces,
then for any irreducible analytic subset $Z\subset X$ which is not entirely contained in the indeterminacy locus of $\Phi$, 
the image
$\Phi(Z)\subset Y$ is defined as follows.
Take an elimination of the indeterminacy locus $\mu:\tilde X\lras X$ of $\Phi$, so that the composition $\Phi\circ\mu:\tilde X\lras Y$ has no indeterminacy point.
Then 
$$
\Phi(Z):= (\Phi\circ\mu)\big(\mu\inv(Z)\big).
$$
This is independent of the choice of the elimination $\mu$.
When $\Phi$ has no indeterminacy point on $Z$, by letting $\mu=\id$, this coincides with the image in the usual sense.
If $\Phi$ has an indeterminacy point on $Z$, then this is often called the {\em meromorphic image} of $Z$.
This is always connected if $Z$ is connected, but although $Z$ is supposed to be irreducible, the image can be reducible because it can have a component from the inverse image of the intersection of $Z$ and the indeterminacy locus of $\Phi$. On the other hand, the closure of the image $\Phi(Z\minus V)$ in $Y$, where $V$ is the indeterminacy locus of $\Phi$, is an irreducible analytic subset of $Y$ (if $Z$ is irreducible). We call this the {\em ``real image''} of $Z$, where the quotation marks are used to avoid confusion with $\sigma$-invariance.
Note that this is not necessarily an irreducible component of the meromorphic image because it can be entirely included in an irreducible component of the meromorphic image that comes from the indeterminacy locus.
In our situation that will be studied below, we will find that this really happens.

Back to our meromorphic map $\tilde\Psi:\tilde Z\lras\ms T$, if $L\subset Z$ is a twistor line that intersects the chain $C_1\cup C_2\cup\dots\cup C_{2n-1}$, then the image $\tilde\Psi(L)$ should be considered as the meromorphic image.
The case where $L$ intersects the central sphere $C_n$ is special because $L$ is $\CC^*_s$-invariant and will be discussed later.
We next determine the meromorphic image of $L$ which intersects the above chain at a point not belonging to $C_n$ and which is different from any $L_i$.
In this case, the ``real image'' of $L$ remains to be a curve because the twistor line is not $\CC^*_s$-invariant.
Recall that the projection $\pi\circ\Pi:\ms T\lras\Lmd$ has reducible fibers precisely over the points $z=a_1,\dots,a_{2n}$ and we denote them 
$$f_{a_i} = \ell_i + \ol\ell_i.$$
Both $\ell_i$ and $\ol\ell_i$ are lines with respect to the embedding $\ms T\subset\PP^{n+2}$.

\begin{proposition}\label{p:i1}
Suppose that $L$ is a twistor line that intersects the curve $C_i$ for some $i$ with $0< i<n$ or $n<i<2n$ and is equal to none of $L_j$
($1\le j\le 2n$).
Let $\lmd\in(a_i,a_{i+1})$ be the real number that satisfies $\Phi(L) = \ms C_{\lmd}$.
Then the ``real image'' of $L$ by $\tilde\Psi$ is the fiber conic $f_{\lmd}$, while the meromorphic image $\tilde\Psi(L)$ is the reducible curve 
\begin{align}\label{mi1}
f_{\lmd} \cup \bigcup_{j=1}^{i} \,(\ell_j \cup \ol \ell_j)
\quad{\text{or}}\quad 
f_{\lmd} \cup \bigcup_{j=i}^{2n} \,(\ell_j \cup \ol \ell_j),
\end{align}
according as $0< i<n$ or $n<i<2n$ respectively.
Further, by $\tilde\Psi$, $L$ is $|n-i|:1$ over the conic $f_{\lmd}$.
\end{proposition}

\proof
We only prove the case $i<n$, as the case $i>n$ can be shown similarly.
The ``real image'' of $L$ is the fiber conic $f_{\lmd}$ for a similar reason to the case $L\cap C_0\neq\emptyset$ as in 
Proposition \ref{p:a}, and we omit the detail.

We put $q_i:= L\cap C_i$, and let $\CCC_i\subset E_i$ be the fiber over $q_i$ of the projection $E_i\lras C_i$, where as before $E_i$ is the exceptional divisor over $C_i$. 
From the list of the chains of base curves on $\tilde Z'$
in Proposition \ref{p:bs1},
all base curves of $|\bm L'|^{\CC^*_s}$ that intersect $\CCC_i$ are $C'_{i,1}, C'_{i,2}, \dots, C'_{i,i}$.
Put $q'_{i,j}:=\CCC_i\cap C'_{i,j}$ for these intersection points.
These points determine curves on all the exceptional divisors $E\uptwo_{i,j},
\dots E^{(n-j+1)}_{i,j}$ over the curve $C'_{i,j}$ as the fibers over the point $q'_{i,j}$. 
Among these fibers, only the one on the last component $E^{(n-j+1)}_{i,j}$ is mapped onto a curve on the central component $E_n\subset \tilde Z^{(2n+2)}$ by the sequence of blowdowns $\tilde Z\upnn\lras\dots\lras\tilde Z^{(2n+2)}$ in the previous section, and the last curve is exactly the line $\ell_i\subset E_n\simeq\tilde{\ms T}$.
Therefore, the meromorphic image $\tilde\Psi(L)$ includes the $i$ lines $\ell_1,\dots, \ell_i$. 
On the other hand, the strict transform of $\CCC_i$ into $\tilde Z\upnn$ is mapped to one of the two $(-n)$-curves in $E_n\simeq\tilde{\ms T}$ (namely, $\bm D\cap E_n\subset \tilde Z^{(2n+2)}$)
by the composition $\tilde Z\upnn\lras\dots\lras\tilde Z^{(2n+2)}%\simeq \tilde{\ms T}\times\PP^1
$ and it is contracted to a singularity of $\ms T$ by the map $\pr'$ in the diagram \eqref{cdm}.
This belongs to the line $\ell_i$, so (by Proposition \ref{p:pr}) the image of $\CCC_i\subset \tilde Z\upnn$ does not appear as a component of the meromorphic image of $L$.
Furthermore, the other fibers of $E^{(m)}_{i,j}\lras C'_{i,j}$
with $m< n-j+1$ is mapped to a point on the line $\ell_i$ by the sequence of blowdowns and therefore it also does not appear as a component of $\tilde\Psi(L)$.
Combining all these, we obtain that the meromorphic image $\tilde\Psi(L)$ is as in the former of \eqref{mi1}.

For the degree of $L$ over $f_{\lmd}$,
let $L'\subset \tilde Z'$ be the strict transform of $L$ into $\tilde Z'$ under $\mu_1:\tilde Z'\lras\tilde Z$.
From the explicit form \eqref{L'} of the transformation $\bm L'$ on $\tilde Z'$, as $L'$ only intersects the components $E_i$ and $\ol E_i$ respectively at one point and the intersection consists of one point and is transversal, we have
\begin{align*}
\bm L' \cdot L' &= \bm L\cdot L - (n-|n-i|) (E_i+\ol E_i)\cdot  L'\\
&= 2n - 2(n-|n-i|)\\
&= 2|n-i|.
\end{align*}
Since the blowup after $\tilde Z'$ does not affect $L$, dividing by the degree of the image conic $f_{\lmd}$, $L$ is $|n-i|:1$ over $f_{\lmd}$.
\proofend

\medskip
From this, similarly to the previous case of $L$ intersecting $C_0$ or $C_{2n}$ as in Corollary \ref{c:im1}, we immediately obtain the following

\begin{corollary}\label{c:im2}
Let $L$ and $\lmd$ be as in the previous proposition.
Then again
$$
\tilde\Psi(L)=\Pi\inv\big(h\cap {\rm C}(\Lmd)\big)
$$ 
holds, where $h\subset\PP^{n+1}$ is the inverse image by $\pi:\PP^{n+1}\lras\PP^n$ of the  hyperplane $\underline h$ in $\PP^n$ passing through the points
\begin{itemize}\setlength{\itemsep}{-2pt}
\item $a_1,a_2,\dots, a_i$ if $i<n$,
\item $a_{2n},a_{2n-1},\dots, a_{i+1}$ if $i>n$,
\end{itemize}
as well as the point $\lmd$, where at $\lmd$ it is tangent to $\Lmd$ with order $|n-i|$. 
\end{corollary}

From the well-known property of the rational normal curve, the hyperplane $\underline h\subset\PP^n$ in the proposition uniquely exists.
Obviously, the hyperplane varies continuously as the point $\lmd$
varies on the arc $(a_i,a_{i+1})$, but from the explicit description in Corollaries \ref{c:im1} and \ref{c:im2}, 
we readily obtain that {\em the hyperplane $\underline h$ (and hence $h=\pi\inv(\underline h)\subset\PP^{n+1}$ also) varies continuously even when the point $\lmd$ moves across an endpoint of the arc.}

The remaining case that a twistor line $L\subset Z$ intersects the central sphere $C_n$ is most interesting.
In this case, as we mentioned, the ``real image'' of $L$ by $\tilde\Psi$ is a single point as $L$ is an orbit closure of the $\CC^*_s$-action.
Identifying the minimal resolution $\tilde{\ms T}$ of $\ms T$ with the fiber $\tilde Z_1=\tilde p\inv(1)$, the ``real image'' of $L$ by $\tilde\Psi$ is the intersection point $L\cap \tilde Z_1$, and this point is mapped to a special connected component of the real locus $\ms T\us$ of $\ms T$ which is also called the central sphere \cite{Hi21,Hi25}.
If $L\neq L_n,L_{n+1}$, then the meromorphic image $\tilde\Psi(L)$ consists of two smooth rational curves which are mutually conjugate by the real structure as follows.

\begin{proposition}\label{p:l}
If $L\subset Z$ is a twistor line that intersects the central sphere $C_n$ but is different from $L_n$ and $L_{n+1}$, then the meromorphic image $\tilde\Psi(L)$ is a reducible curve of the form $\CCC\cup\ol\CCC$, where $\CCC$ is a section of $\pi\circ\Pi:\ms T\lras\Lmd$ that intersects the $2n$ lines $\ol\ell_i$ and $\ell_{n+i}$ at one point for $i=1,\dots,n$ 
and that do not intersect the residual $2n$ lines 
$\ell_i$ and $\ol\ell_{n+i}$ for $i=1,\dots,n$.
\end{proposition}

Since $\tilde\Psi$ preserves the real structure, the ``real image'' of $L$, which has to be a point as above, is real. From the proposition, this point has to be one point of $\CCC\cap\ol\CCC$ which belongs to the central sphere in $\ms T$.

\medskip\noindent
{\em Proof of Proposition \ref{p:l}.}
Again we use the sequence of explicit blowups of $\tilde Z$ that eliminates the indeterminacy of $\tilde\Psi$.
Put $q:=L\cap C_n$. Then $\ol q = L\cap \ol C_n$,
and $q$ and $\ol q$ are all indeterminacy points on $L$.
As above, the image of $L\minus \{q,\ol q\}$ has to be a point that belongs to the central sphere.
Let $\CCC$ be the fiber of the projection $E_n\lras C_n$ over the point $q$ by the first blowup $\hat Z\lras\tilde Z$.
Then $\ol\CCC\subset \ol E_n$ is the fiber of $\ol E_n\lras \ol C_n$ over the point $\ol q$.
Among the rational curves which are the exceptional curves of the small resolution $\tilde Z'\lras\hat Z$,
only the ones over the two points $z=a_n$ and $z=a_{n+1}$ are inserted in $E_n$, and these two curves are finally mapped onto the lines $\ell_n$ and $\ol\ell_{n+1}$ in $\ms T$ respectively.
The next map $\mu_2:\tilde Z\uptwo\lras \tilde Z'$ blows up $2(n-1)$ points of $E_n$ and the exceptional curves are finally mapped onto the lines $\ell_1,\dots,\ell_{n-1}$ and $\ol\ell_1,\dots,\ol\ell_{n-1}$ in $\ms T$.
Therefore, the image of $\CCC$ to $\ms T$ does not intersect the $2n$ lines $\ell_1,\dots,\ell_n$
and $\ol\ell_{n+1},\dots,\ol\ell_{2n}$ and does intersect the residual $2n$ lines $\ol\ell_1,\dots,\ol\ell_n$ and $\ell_{n+1},\dots,\ell_{2n}$.

Since the meromorphic map $\tilde\Psi$ preserves the real structure,
the meromorphic image $\tilde\Psi(L)$ is real. Therefore, the conjugate curve $\ol\CCC$ also has to be included in $\tilde\Psi(L)$ and it has to be from the point $\ol q$.
Thus, $\CCC\cup\ol\CCC\subset \tilde\Psi(L)$.
Since $\tilde\Psi$ has the two points $q$ and $\ol q= L\cap \ol C_n$ as its all base points on $L$, the sum $\CCC+\ol\CCC$ is the entire meromorphic image.
\proofend

\medskip
The meromorphic images of $L_n$ and $L_{n+1}$ are given as follows:

\begin{proposition}\label{p:itl}
If a twistor line $L\subset Z$ is $L_n$ or $L_{n+1}$, then the meromorphic image $\tilde\Psi(L)$ is, respectively 
\begin{align}\label{mi2}
\bigcup_{j=1}^{n} \,(\ell_j \cup \ol \ell_j)
\quad{\text{or}}\quad 
\bigcup_{j=n+1}^{2n}  (\ell_j \cup \ol \ell_j).
\end{align}
\end{proposition}

\proof
We only prove the case $L=L_n$ as the case
$L=L_{n+1}$ can be shown similarly.
We use the notations in the previous proof.
This time, the fiber $\CCC$ of the projection $E_n\lras C_n$ over the point $q=C_{n-1}\cap C_n$, where $E_n\subset \hat Z$, passes through the $n$ points over $z=a_1,\dots,a_n$ to be blown up by $\tilde Z'\lras\hat Z$ and $\tilde Z\uptwo\lras\tilde Z'$.
So the transformation of $E_n$ into $\tilde Z\uptwo$ includes the exceptional curves 
which are finally mapped onto $\ell_1,\ell_2,\dots,\ell_n$ in $\ms T$.
Hence, these $n$ lines are included in the meromorphic image of $L$.
Further, this time, the curve $\CCC$ in $E_n\subset\tilde Z^{(2n+2)}$ is contracted to a singularity of $\ms T$.
From these, taking the contribution from another base point $\ol q$ into account, we obtain that the meromorphic image $\tilde\Psi(L)$ is as in the former of \eqref{mi2}.
\proofend

\medskip
Since the curve $\Gamma$ in the proof of Propositions \ref{p:l} and \ref{p:itl} faithfully varies as the intersection point $L\cap C_n$ varies, the meromorphic image of twistor lines through the central sphere $C_n$ constitute a real 2-dimensional family of reducible curves on $\ms T$, and the parameter space of this family is naturally identified with the curve $C_n$ itself.
Except for the two members that are from $L_n$ and $L_{n+1}$, each member of this family consists of two smooth rational curves $\CCC$ and $\ol\CCC$, and they constitute a pair of pencils in the usual sense which are mutually $\sigma$-conjugate.
This pair of pencils is exactly the one discussed in \cite{Hi21} and \cite[Section 6]{Hi25}.

We note that although the two curves \eqref{mi2} look quite different from $\CCC+\ol\CCC$ in Proposition \ref{p:l}, each of them can be naturally regarded as a limit of the latter curves.
More concretely, if we put $\CCC= \sum_{i=1}^n\ell_i$ or $\CCC= \sum_{i=n+1}^{2n}\ell_i$, then 
the curves \eqref{mi2} can be written as $\CCC+\ol\CCC$ respectively, and moreover, these $\Gamma$ are limits of the curve $\Gamma$ in Proposition \ref{p:l}.
Hence, in the following, we do not treat the case $L= L_n, L_{n+1}$ separately.

So as in Propositions \ref{p:l} and \ref{p:itl}, let $L\subset Z$ be a twistor line that intersects $C_n$ and write $\tilde\Psi(L) = \CCC + \ol\CCC\subset\ms T$ including the case $L=L_n, L_{n+1}$ as above.
Next we show that the images of $\CCC$ and $\ol\CCC$ by the double covering $\Pi:\ms T\lras {\rm C}(\Lmd)$ have some nice property.
Recall that this covering has the hyperelliptic curve $\Sigma\subset {\rm C}(\Lmd)$ branched at $a_1,\dots,a_{2n}$ as the branch divisor. (See \eqref{Sgm}.)
It is convenient to introduce the following
\begin{definition}
{\em
We say that a hyperplane $h\subset\PP^{n+1}$ is {\em evenly tangent to $\Sigma$} if $h$ is tangent to $\Sigma$ at every intersection point and if the contact order is even at every tangent point.
}
\end{definition}

Since $\deg(\Sigma) = 2n$, most generically, an evenly tangential hyperplane has exactly $n$ tangent points to $\Sigma$.
Note that the inverse image of a hyperplane in $\PP^n$ passing through $n$ points among $a_1,a_2,\dots,a_{2n}$ under the projection $\pi:\PP^{n+1}\lras\PP^n$ is always evenly tangent to $\Sigma$, because such a hyperplane contains the generating lines of the cone over the $n$ points and therefore is tangent to $\Sigma$ at ramification points.

\begin{proposition}\label{p:tan}
If a twistor line $L\subset Z$ intersects the central sphere $C_n$ and write $\tilde\Psi(L)=\CCC+\ol\CCC$ as above,
then the images $\Pi(\CCC)$ and $\Pi(\ol\CCC)$ are sections of ${\rm C}(\Lmd)$ by hyperplanes which are
evenly tangent to the hyperelliptic curve $\Sigma$.
\end{proposition}

Note that we are not asserting that the hyperplane is real, or equivalently, that $\Pi(\CCC) = \Pi(\ol\CCC)$ holds.
If $\Pi(\CCC) \neq \Pi(\ol\CCC)$, then the meromorphic image $\tilde\Psi(L)$ is not a hyperplane section of $\ms T$.
We will soon show that the coincidence $\Pi(\CCC) = \Pi(\ol\CCC)$ holds only for $L$ such that the intersection point $L\cap C_n$ belongs to a particular circle in $C_n$, which goes through the two points $C_n\cap C_{n-1}$ and $C_n\cap C_{n+1}$.
(But the role of this circle will be auxiliary.)

\medskip
\noindent
{\em Proof of Proposition \ref{p:tan}.}
If $L=L_n$, then as above $\CCC= \sum_{i=1}^n\ell_i$ and hence $\Pi(\CCC)$ is the sum of the generating lines of the cone ${\rm C}(\Lmd)$ over the points $a_1,\dots,a_n$.
This is cut of ${\rm C}(\Lmd)$ by the hyperplane $h_0$ which is obtained from the hyperplane in $\PP^n$ spanned by the $n$ points $a_1,\dots,a_n$ by taking the inverse image under $\pi:\PP^{n+1}\lras\PP^n$.
The hyperplane $h_0$ is evenly tangent to $\Sigma$ as above.
The case $L=L_{n+1}$ can be shown similarly just by replacing $a_1,\dots,a_n$ with 
the residual $n$ points $a_{n+1},\dots,a_{2n}$.

In the sequel,  we suppose that $h\subset\PP^{n+1}$ is (not necessarily real) hyperplane that is evenly tangent to $\Sigma$ and which does not pass through the vertex of the cone ${\rm C}(\Lmd)$,
and find a property of the inverse image $\Pi\inv(h)\subset\ms T$
that characterizes a component of the curve $\Pi\inv(h)$.

The cut ${\rm C}(\Lmd)\cap h$ is a smooth rational curve. From the evenly tangential condition, $\Pi\inv(h)$ consists of two irreducible components and both of them are smooth rational curves.
Let $\CCC_1$ and $\CCC_2$ be these components, so that $\Pi\inv(h) = \CCC_1+\CCC_2$.
Then for any $1\le i\le 2n$, $\CCC_1$ and $\CCC_2$ intersect at least one of the two lines $\ell_i$ and $\ol \ell_i$.
If $\CCC_1$ or $\CCC_2$ would intersect both $\ell_i$ and $\ol \ell_i$ for some $1\le i\le 2n$, then $h$ would pass through the branch point of $\pi:\Sigma\lras\Lmd$ over the point $z=a_i$ from tangency, which implies that $h$ passes through the vertex of ${\rm C}(\Lmd)$. Hence, $\CCC_1$ and $\CCC_2$ intersect exactly one of $\ell_i$ and $\ol\ell_i$ for any $1\le i\le 2n$.
Of course, $\Gamma_1$ and $\Gamma_2$ do not intersect the same line.
Further, using the realization of the minimal resolution $\tilde{\ms T}$ as a blowup of $\qdr$ as explained at the end of Section \ref{s:2}, we can prove that the lines that intersects $\CCC_1$ 
are $n$ lines $\ell_{i_1},\dots,\ell_{i_n}$ 
for some indices $1\le i_1< \dots< i_n\le 2n$ and also $n$ lines 
$\ol\ell_{j_1},\dots,\ol\ell_{j_n}$ where $\{i_1,\dots,i_n\}\cup\{j_1,\dots,j_n\} = \{1,2,\dots,2n\}$.
This is the property of the components of $\Pi\inv(h)$ for an evenly tangential hyperplane $h$ that does not pass through the vertex of ${\rm C}(\Lmd)$.

Conversely, let $\Gamma_1$ be an irreducible curve on $\ms T$ that does not pass through the singularities of $\ms T$ and that intersects $2n$ lines among the $4n$ lines on $\ms T$ in the above way.
Then the curve $\Gamma_2:=\Pi\inv(\Pi(\Gamma_1))-\Gamma_1$ intersects the $4n$ lines on $\ms T$ in the complementray way to $\Gamma_1$.
Moreover, the way of the intersection with the lines means that there exists a hyperplane $h\subset\PP^{n+1}$ such that $\Gamma_1+\Gamma_2=\Pi\inv(h)$.
Furthermore, the reducibility of $\Gamma_1+\Gamma_2$ means that any intersection point of $h\cap\Sigma$ is tangential and that the contact order is even.
Therefore, $h$ is evenly tangent to $\Sigma$ and $\Gamma_1$ is an irreducible component of $\Pi\inv(h)$.
Thus, we have obtained the desired characterization.

Back to the components $\CCC$ and $\ol\CCC$ of $\tilde\Psi(L)$, by Proposition \ref{p:l}, $\CCC$ intersects the $n+n$ lines $\ol\ell_1,\dots,\ol\ell_n$ and $\ell_{n+1},\dots,\ell_{2n}$ and does not intersect the residual $n+n$ lines. 
This is a special case of the above way of intersection.
Further, from the proof of the same proposition, $\CCC$ does not pass through the two singularities of $\ms T$.
Therefore, from the above characterization in this proof, $\CCC$ is a component of $\Pi\inv(h)$ for a hyperplane $h\subset\PP^{n+1}$ that is evenly tangent to $\Sigma$ and that is not through the vertex of ${\rm C}(\Lmd)$.
By the same argument, $\ol\CCC$ is also obtained in such a way
(but it is not necessarily obtained from the same $h$ as $\CCC$.)
\proofend

\medskip
In order to show the existence of the circle in $C_n$ mentioned right after the previous proposition, we prove the following property about real evenly tangent hyperplanes.

\begin{proposition}\label{p:gm}
{\em Real} hyperplanes in $\PP^{n+1}$ that are evenly tangent to the hyperelliptic curve $\Sigma$ constitute (various) real 1-dimensional families parameterized by $S^1$.
There exists a natural one-to-one correspondence between the set of such $S^1$-families and the set of equal divisions of the set of $2n$ generating lines of the cone ${\rm C}(\Lmd)\subset\PP^{n+1}$ over the points $a_1,a_2,\dots, a_{2n}$.
\end{proposition}

\proof
First, we show that real evenly tangential hyperplanes in $\PP^{n+1}$ constitute families parameterized by $S^1$. Let $h$ be such a hyperplane that does not pass through the vertex of ${\rm C}(\Lmd)$. 
As $h$ is real, we may write $\Pi\inv(h) = \Gamma + \ol\Gamma$.
Since $\deg(\Sigma) = 2n$, the sum of all contact orders of $h$ to $\Sigma$ is $2n$.
From the even tangency, this means that $\Gamma\cdot\ol\Gamma = n$ for the intersection number on $\ms T$.
Therefore, again using the reality, we have $(\Gamma+\ol\Gamma)^2 = 2\Gamma^2 + 2\Gamma\cdot\ol\Gamma
= 2\Gamma^2 + 2n$. 
On the other hand, as $\Pi\inv(h) = \Gamma+\ol\Gamma$, we have
$$
(\Gamma+\ol\Gamma)^2 = 2 h^2 = 2\deg\Lmd = 2n.
$$
Hence, we obtain $\Gamma^2 = 0$ and hence $\ol\Gamma^2 =0$.
From this, as $\CCC$ and $\ol\CCC$ are smooth rational curves and $\ms T$ is a rational surface, the complete linear systems $|\CCC|$ and $|\ol\CCC|$ are both pencils. 
We write $\{\CCC_t\set t\in\PP^1\}$ for the former pencil.
Then $|\ol\CCC|=\{\ol\CCC_t\set t\in\PP^1\}$. 

To obtain the $S^1$-family of real and evenly tangential hyperplanes from these pencils on $\ms T$, we next let $\chi:\ms T\lras\ms T$ be the covering transformation of 
the double covering $\Pi:\ms T\lras {\rm C}(\Lmd)$.
In terms of the coordinates used in \eqref{Hi1},
$\chi(x,y,z) = (-y,-x,z)$.
From this and \eqref{rs1}, we readily see that $\chi\circ\sigma=\sigma\circ\chi$.
Leting $\CCC$ be the curve on $\ms T$ determined from a real evenly tangential hyperplane $h$ as above, since $\chi(\CCC) = \ol\CCC$, $\chi$ induces a holomorphic identification $|\CCC|\lras|\ol\CCC|$, for which we use the same letter $\chi$.
Similarly, we have an anti-holomorphic identification 
$\sigma:|\CCC|\lras|\ol\CCC|$.
Then the composition $\chi\circ\sigma$ gives an anti-holomorphic isomorphism from $|\Gamma|$ to itself which is involutive using 
$\chi\circ\sigma=\sigma\circ\chi$.
Hence, as a map from $|\CCC|\simeq\PP^1$ to itself, $\chi\circ\sigma$ is identified with either the standard complex conjugation or the antipodal map.
But since it has the curve $\CCC$ as an invariant element, 
it has to be identified with the complex conjugation.
Thus, the pencil $|\Gamma|\simeq\PP^1$ has a natural real structure, which has ``real'' members parameterized by $S^1$.
By the construction, such a member is characterized by the property 
(not $\ol\CCC_t=\CCC_t$ but) $\chi(\CCC_t) = \ol\CCC_t$.
For such a member $\CCC_t$, the image $\Pi(\CCC_t)=\Pi(\ol\CCC_t)$ is a real hyperplane section of ${\rm C}(\Lmd)$ which is evenly tangent to $\Sigma$.
This is the $S^1$-family of such hyperplanes as in the proposition.

To obtain the equal division of the $2n$ generating lines of the cone ${\rm C}(\Lmd)$ from each of these $S^1$-families, the group $\CC^*_t$ in \eqref{act-t} acts on the parameter space of the above pencil $|\Gamma|$ and its two invariant members can be written
\begin{align}\label{Gm}
\CCC_0:=\sum_{m=1}^n \ell_{i_m} \qandq 
\CCC_{\infty}:=\sum_{m=1}^n \ol\ell_{j_m}
\end{align}
for some indices $1\le i_1< \dots< i_n\le 2n$ and $1\le j_1<\dots<j_n\le 2n$ such that $\{i_1,\dots,i_n\}\cup\{j_1,\dots,j_n\} = \{1,2,\dots,2n\}$.
(Note that in $\CCC_{\infty}$, the complex conjugation is taken.)
Then the equal division $\{a_1,a_2,\dots, a_{2n}\} = \{a_{i_1},\dots, a_{i_n}\}\cup\{a_{j_1},\dots, a_{j_n}\}$ gives the required one of the $2n$ generating lines of ${\rm C}(\Lmd)$.
Conversely, for a given equal division $\{a_1,a_2,\dots, a_{2n}\} = \{a_{i_1},\dots, a_{i_n}\}\cup\{a_{j_1},\dots, a_{j_n}\}$, the sums \eqref{Gm} generate a complete pencil on $\ms T$, and choosing the $S^1$-subfamily of this pencil by using the covering transformation $\chi$ as above, we obtain an $S^1$-family $\{\Gamma_t+\ol{\Gamma}_t\set t\in S^1\}$ whose image to ${\rm C}(\Lmd)$ is a restriction to ${\rm C}(\Lmd)$ of an $S^1$-family of real evenly tangential hyperplanes in $\PP^{n+1}$.
Obviously, these two directions are converse to each other and we obtain the desired one-to-one correspondence.
\proofend

\medskip
Let $\{h_t\set t\in S^1\}$ be any one of the $S^1$-families of real evenly tangential hyperplanes to $\Sigma$.
Then for any $t\in S^1$ we can write $\Pi\inv(h_t) = \CCC_t+ \ol\CCC_t$ where $\ol\CCC_t = \chi(\CCC_t)$ as in the previous proof.
Any point of the intersection $\CCC_t\cap\ol\CCC_t$ is identified with a tangent point of the hyperplane $h_t$ to $\Sigma$ and the contact order is twice the local intersection number between $\CCC_t$ and $\ol\CCC_t$.
Since $\CCC_t\cdot\ol\CCC_t=n$ as in the previous proof, this implies that we can write 
$h_t|_{\Sigma} = 2D_t$ for some effective divisor $D_t$ of degree $n$ on $\Sigma$. Since $h_t$ is real, $D_t$ is also real.
For any $1\le i\le 2n$, put $r_i = \ell_i\cap \ol\ell_i$. These are the ramification points of the double covering $\pi:\Sigma\lras\Lmd\simeq\PP^1$.
As in \cite[Proposition 2.3]{H25}, if $\{i_1,\dots,i_n\}\cup\{j_1,\dots,j_n\} = \{1,2,\dots,2n\}$ is an equal division of the set $\{1,2,\dots,2n\}$, then we have a linear equivalence
$\sum_{l=1}^n r_{i_l}\sim \sum_{l=1}^n r_{j_l}$ and 
$\dim |\sum_{l=1}^n r_{i_l}|=\dim|\sum_{l=1}^n r_{j_l}|=1$.
Since all these points are real, this pencil has a natural real structure.

\begin{proposition}\label{p:Dt}
If $1\le i_1<i_2<\dots<i_n\le n$ are the indices that correspond to an $S^1$-family $\{h_t\set t\in S^1\}$ of real evenly tangential hyperplanes to $\Sigma$ and $D_t=\frac12 h_t|_{\Sigma}$ as above, then $D_t \sim \sum_{l=1}^n r_{i_l}$ and the $S^1$-family $\{D_t\set t\in S^1\}$ is the set of real members of the pencil $|\sum_{l=1}^n r_{i_l}|$.
\end{proposition}

\proof
Let $h_0$ (resp.\,$h_{\infty}$) be the inverse image under $\pi:\PP^{n+1}\lras\PP^n$ of the hyperplane in $\PP^n$ spanned by the $n$ points $a_{i_1},\dots, a_{i_n}$ (resp.\,$a_{j_1},\dots, a_{j_n}$).
Then using the notation from \eqref{Gm},  $\Pi\inv(h_t) = \Gamma_t+\ol\Gamma_t$ for $t=0,\infty$, $h_0|_{\Sigma} = \sum_{l=1}^n 2r_{i_l}$ and $h_{\infty}|_{\Sigma} = \sum_{l=1}^n 2r_{j_l}$
Hence, $D_0 = \sum_{l=1}^n r_{i_l}$ and $D_{\infty} = \sum_{l=1}^n r_{j_l}$.
Since all $h_t$ are mutually linearly equivalent, so are the restrictions $h_t|_{\Sigma}$. Hence, the fact $\dim|\sum_{l=1}^n r_{i_l}|=1$ and a continuity of the family $\{D_t\set t\in S^1\}$ forces that all $D_t$ belong to this pencil.
Since all $D_t$ are real as above, the assertion of the proposition follows.
\proofend

\medskip
With this understanding, the existence of a special circle in $C_n$ is easy to prove:
\begin{proposition}\label{p:ccl}
There exists a smooth circle $\gamma_n\subset C_n$ such that the meromorphic image $\tilde\Psi(L)=\Gamma+\ol\Gamma$ of a twistor line $L$ through $C_n$ satisfies $\Pi(\Gamma) = \Pi(\ol\Gamma)$, or equivalently, $\tilde\Psi(L) = \Pi\inv(h)$ for some real evenly tangential hyperplane $h$, if and only if 
$L$ intersects $\gamma_n$.
This circle $\gamma_n$ passes through the two points $C_n\cap C_{n-1}$ and $C_n\cap C_{n+1}$.
\end{proposition}

\proof
As in the proof of Proposition \ref{p:tan}, the curve $\Gamma$ as in the proposition generates a complete pencil $|\Gamma|$ on $\ms T$, and the parameter space of this pencil is naturally identified with the component $C_n$ as the intersection point with $L$.  From the proof of Proposition \ref{p:gm}, this pencil has the $S^1$-subfamily $\{\Gamma_t\set t\in S^1\}$ such that $\ol\Gamma_t=\chi(\Gamma_t)$ for any $t\in S^1$. Hence, if $L_t\subset Z$ is the twistor line that satisfies $\tilde\Psi(L_t) = \Gamma_t + \ol\Gamma_t$, then the circle $\gamma_n$ is given as the set of intersection points $\{L_t\cap C_n\set t\in S^1\}$.
The circle $\gamma_n$ passes through the point $C_n\cap C_{n-1}$ because $\tilde\Psi(L_n) = \Gamma_0+\ol\Gamma_0$ and $\ol\Gamma_0 = \chi(\Gamma_0)$ so that $\Gamma_0$ belongs to the $S^1$-family. Similarly, $\gamma_n$ passes through another point $C_n\cap C_{n+1}$ using $\tilde\Psi(L_{n+1}) = \Gamma_{\infty}+\ol\Gamma_{\infty}$ and $\ol\Gamma_{\infty} = \chi(\Gamma_{\infty})$.
\proofend

\medskip
From Proposition \ref{p:itl} or the last part of the proof of Proposition \ref{p:tan}, the two limiting members in the $S^1$-family of real divisors on $\ms T$ determined by twistor lines through $\gamma_n$ are given by 
$$\Gamma_0+\ol\Gamma_0 = \sum_{i=1}^n (\ell_i+\ol\ell_i)
\qandq
\Gamma_{\infty}+\ol\Gamma_{\infty}
=\sum_{i=1}^n (\ell_{n+i}+\ol\ell_{n+i}).
$$
Therefore, the equal division of the set $\{1,2,\dots,2n\}$ that corresponds to this $S^1$-subfamily is 
\begin{align}\label{ed}
\{1,2,\dots,n\}\cup\{n+1,n+2,\dots,2n\}.
\end{align}

From Proposition \ref{p:Dt}, if $\{h_t\set t\in S^1$ is an $S^1$-family of real evenly tangential hyperplanes to $\Sigma$, then the divisors $D_t$ ($t\in S^1$) of tangent points constitute a real pencil, which contains two divisors $\sum_{l=1}^nr_{i_l}$ and $\sum_{l=1}^nr_{j_l}$ as special members.
But we have to note:
\begin{itemize}\setlength{\itemsep}{-2pt}
\item There may exist $t\in S^1$ such that $D_t$ contains a multiple point. In other words, the contact order of $h_t$ to $\Sigma$ at a tangent point can be $2m$ with $m>1$.
\item The divisor $D_t$ is always real as a whole, but it can contain a pair of points of the form $q_t, \ol q_t$ with $q_t\neq \ol q_t$.
\end{itemize}
These issues are mutually related.
In fact, if $t_0\in S^1$ is such that $D_{t_0}$ contains $2q$ with $q=\ol q$, then typically, the following holds in a neighborhood of $t_0$.
\begin{itemize}\setlength{\itemsep}{-2pt}
\item If $t<t_0$ then $D_t$ contains $q_t+q'_t$ with both $q_t,q'_t$ real and $q_t\neq q'_t$.
\item If $t>t_0$, then $D_t$ contains $q_t+\ol q_t$ with $q_t\neq\ol q_t$.
\end{itemize}
In the present situation where the limiting divisors are $\sum_{i=1}^nr_{i}$ and $\sum_{i=1}^nr_{n+i}$, if $g=1$ (i.e.\,$n=2$), there exists no such $t_0$, because the points $r_1$ and $r_2$ (and also $r_3$ and $r_4$) belong to mutually distinct components of the real circle of the branch hyperelliptic curve $\Sigma$; see Section \ref{ss:he}). 
On the other hand, if $g>1$, such a $t_0$ really exists and the above phoenomena actually happen; see Section \ref{ss:sing}.
In contrast, in the Lorenzian case studied in \cite{H25}, the divisor $D_t$ never contains a multiple point and all its points are always real; see the next remark.

\begin{remark}
{\em
In the Lorenzian case studied in \cite{H25}, the key in the construction of the EW space was also the $S^1$-family of tangential real hyperplanes to the (same) hyperelliptic curve $\Sigma\subset\rm C(\Lmd)$.
The most typical division of the set $\{1,2,\dots,2n\}$ used there was
$$\{i\set i:{\rm{ odd}}\}\cup \{i\set i:{\rm{ even}}\}.$$
Including this case, all points of the divisor $D_t$ are always distinct and real.
This follows from the fact that each tangent point belongs to a mutually distinct connected component of the real locus $\Sigma\us$ of $\Sigma$ and therefore they cannot be equal.
}\end{remark}

In this section, we investigated the (meromorphic) images of twistor lines in $Z$ that intersect the chain $C_0\cup C_1\cup \dots\cup C_{2n}$, under the meromorphic quotient map $\tilde\Psi:\tilde Z\lras\ms T$.
It turns out that, if we consider the circle $\gamma_n$ obtained in Proposition \ref{p:ccl} instead of the entire central sphere $C_n$, then all these images are real hyperplane sections of $\ms T\subset\PP^{n+2}$, and moreover, all these hyperplanes are inverse images of real hyperplanes in $\PP^n$ under the projection $\pi:\PP^{n+1}\lras\PP^n$.

About the continuity of these hyperplanes,
as mentioned right after Corollary \ref{c:im2}, the (meromorphic) images of minitwistor lines through the chain $C_0\cup C_1\cup \dots\cup C_{2n}$ vary continuously except possibly when the intersection point varies across the points $C_{n-1}\cap C_n$ or $C_n\cap C_{n+1}$.
But since the hyperplane in $\PP^{n+1}$ which cuts out the meromorphic image $\tilde\Psi(L_n)$ (resp.\,$\tilde\Psi(L_{n+1})$) is the inverse image under $\pi$ of the hyperplane in $\PP^n$ spanned by the $n$ points $r_1,\dots,r_n$ (resp.\,$r_{n+1},\dots,r_{2n}$) as we have seen, and since these are limiting members of the $S^1$-family of real evenly tangential hyperplanes to $\Sigma$ arising from twistor lines intersecting $\gamma_n$, the jumping of the hyperplanes do not happen at the points $C_{n-1}\cap C_n$ or $C_n\cap C_{n+1}$.
Namely, the continuity of hyperplanes holds even at the intersections $C_{n-1}\cap \gamma_n$ and $\gamma_n\cap C_{n+1}$.

Moreover, a further continuity holds as follows.
By Proposition \ref{p:a}, if a twistor line $L$ intersects the end component $C_0$ or $C_{2n}$ of the chain (i.e.,\,the `axis'), then the image $\tilde\Psi(L)$ 
is of the form $n f_{\lmd}$, where $\lmd\in(-\infty,a_1)$ or $\lmd\in(a_{2n},\infty)$ according as $L\cap C_0\neq\emptyset$
or $L\cap C_{2n}\neq\emptyset$ respectively.
Hence, if $f_{\infty}$ means the fiber of $\pi\circ\Pi:\ms T\lras\Lmd$ over the point $\infty=-\infty\in\Lmd\us$, then we have the coincidence of the limits 
\begin{align}\label{infty}
\lim_{L\cap C_0\to \infty_0} \tilde\Psi(L) 
=
\lim_{L\cap C_{2n}\to \infty_{2n}} \tilde\Psi(L) = nf_{\infty},
\end{align}
where $\infty_0:=C_0\cap \bm D$ and $\infty_{2n} = C_{2n}\cap \ol{\bm D}$, which are the points of $C_0$ and $C_{2n}$ respectively that do not belong to $Z$.

As before, let $\tilde Z_0$ and $\tilde Z_1$ be the fibers of $\tilde p:\tilde Z\lras\PP^1$ over the points $u=0$ and $u=1$ respectively,
and $Z_0$ and $Z_1$ be its intersection with $Z$ respectively.
Then there is a natural diffeomorphism 
\begin{align}\label{alpha}
\aaa:Z_0\lras Z_1
\end{align}
that maps the point $Z_0\cap L$ to the point $Z_1\cap L$ where $L$ is an arbitrary (real) twistor line.
Also, on $Z_1$, there is a projection $\Phi|_{Z_1}:Z_1\lras\Lmd$ that takes the $z$-coordinate.
Write the point $-\infty=\infty$ of $\Lmd\us$ as $a_0$ and $a_{2n+1}$. Namely, $a_0 = a_{2n+1} =-\infty=\infty$.
From Proposition \ref{p:tl}, the composition $\Phi\circ\alpha:Z_0\lras\Lmd$ maps each component $C_i$ to the arc $I_i=[a_i,a_{i+1}]\subset\Lmd\us$ and this descends to a map from the quotient $C_i/S^1$ by the scalars.
If $i\neq n$, this map from $C_i/S^1$ to the arc $[a_i,a_{i+1}]$ is homeomorphic. 
If $i=n$, the restriction of the map to the circle $\gamma_n\subset C_n$ is two-to-one over the arc $[a_n,a_{n+1}]$ branched at the boundary points. 
Let $\ccc'_n$ be any one of the two semicircles in $\gamma_n$ bounded by the points $C_n\cap C_{n-1}$ and $C_n\cap C_{n+1}$. Then the map $\alpha|_{\gamma'_n}$ is a homeomorphism from $\gamma'_n$ to the arc $[a_n,a_{n+1}]\subset\Lmd\us$.

Therefore, we have obtained a continuous family of real hyperplanes in $\PP^{n+1}$ parameterized by the entire circle $\Lmd\us$, formed by twistor lines through the components $C_i$ with $i\neq n$ and the semicircle $\gamma'_n$ in $C_n$, and also the limit hyperplane as the points on $C_0$ and $C_{2n}$ go to infinity.
In the next section, we use this $\Lmd\us\simeq S^1$-family to give the family of minitwistor lines on $\ms T$ which induces the EW structure on the quotient space of the gravitational instanton by the scalar $S^1$-action.

\section{The complete family of minitwistor lines}\label{s:EW}
In this section, we first give a 2-dimensional family of real hyperplanes in $\PP^{n+1}$ whose generic member is tangent to the hyperelliptic curve $\Sigma\subset {\rm C}(\Lmd)$ at exactly $(n-1)$ points.
The parameter space of this family will be a quarter of the hyperelliptic curve $\Sigma$ and is
constructed using the (meromorphic) images of special twistor lines obtained in the previous section as boundary data.
This family will give minitwistor lines in $\ms T$
by pulling back their hyperplane sections of the cone by the double covering $\Pi:\ms T\lras {\rm C}(\Lmd)$.
We will show that this family constitutes a global slice of the residual $S^1$-action on the EW space of the minitwistor space induced from the tri-holomorphic $S^1$-action (Proposition \ref{p:ext}).
Then rotating the quarter, we will obtain the 3-dimensional family of minitwistor lines in $\ms T$ that corresponds to the EW space obtained from the gravitational instanton by the scalar $S^1$-action.

\subsection{The real and the pure imaginary circles of $\Sigma$}
\label{ss:he}
First, we define the notation and terminology.
Recall that for the toric ALE gravitational instanton of type $A_{2n-1}$, the genus $g$ of the hyperelliptic curve $\Sigma$ is $(n-1)$.
%The letter $g$ will always mean the genus of $\Sigma$, which equals $n-1$.
For each index $1\le i< 2n$, we denote $I_i =[a_i,a_{i+1}]$, a closed arc
in the real circle $\Lmd\us$ 
Further, we define a closed (connected) arc $I_0\subset\Lmd\us$ 
by $[a_{2n},a_1]:=[a_{2n},\infty]\cup [-\infty,a_1]$.
So $\Lmd\us$ is the union of the $2n$ arcs
$I_0,I_1,\dots,I_{2n-1}$.
We denote $I_i\uc$ for the interior of $I_i$.
Put $f(z):=\prod_{i=1}^{2n}(z-a_i)$, so the minitwistor space $\ms T$ is defined by $xy=f(z)$.
Recalling that the real structure on $\ms T$ is given by 
$(x,y,z)\longmapsto \big((-1)^n\ol y,(-1)^n\ol x,\ol z\big)$ as in \eqref{rs1}, for fixed $z\in\Lmd\us$, the fiber conic $xy =f(z)$ has a real circle iff 
\begin{equation}
  \begin{minipage}{0.9\linewidth}
  \begin{itemize}
    \item $z\in I\uc_1\cup I\uc_3\cup I\uc_5\cup\dots\cup I\uc_{2n-1}$ if $n$ is odd,
    \item $z\in I\uc_0\cup I\uc_2\cup I\uc_4\cup\dots\cup I\uc_{2n-2}$ if $n$ is even.
  \end{itemize}\label{itv}
  \end{minipage}
\end{equation}
In both cases these are exactly half of the $2n$ intervals and the $n$-th arc $I_n\uc$ is always included.
In the following, we denote 
%\begin{align}\label{Ts}
$\ms T\us_i$
%\end{align}
for the real locus of $\ms T$ lying over $I_i$.
Note that this is different from the notation used in \cite{H25}; this time the real spheres are $\ms T\us_{\rm odd}$ (resp.\,$\ms T\us_{\rm odd}$) if $n$ is odd (resp.\,even) and there are no $\ms T\us_{\rm even}$ (resp.\,$\ms T\us_{\rm even}$) in that case.
Namely, the real spheres are of the form $\ms T\us_i$, where $i$ has the same parity as $n$.
We call these the {\em real spheres}.
There are exactly $n$ real spheres and all of them are smooth spheres in $\ms T$. 
Moreover, all these are invariant under the residual $S^1$-action 
$(x,y,z)\longmapsto(tx,t\inv y,z)$.
The sphere $\ms T\us_n$ is also called the central sphere \cite{Hi25} and plays a distinguished role in the sequel.
The real locus $\ms T\us$ of $\ms T$ consists of these $n$ spheres.

As before, let $\pi:\Sigma\lras\Lmd\simeq\PP^1$ be the double covering of the hyperelliptic curve, which is the projection $(v,z)\longmapsto z$ in the coordinates used in \eqref{Sgm}.
Also, we denote $$r_1,r_2\dots, r_{2n}$$ for the ramification points of $\pi$ over
$a_1,\dots,a_{2n}\in\Lmd\us$ respectively.
Then if we let $\ell_i$ and $\ol\ell_i$ be the lines on $\ms T$ lying over the generating line over $a_i\in\Lmd$ as before, $r_i = \ell_i\cap \ol\ell_i$ for any index $i$.
The real structure on the cone is given by $(v,z)\longmapsto
((-1)^n\ol v,\ol z)$ in the above coordinates, and it preserves $\Sigma$.
For any index $i$, the inverse image $\pi\inv(I_i)\subset\Sigma$ is a smooth circle in $\Sigma$, and it is contained in the real locus $\Sigma\us$ of $\Sigma$ if and only if $i\equiv n\,(2)$.
In that case we denote it $\Sigma\us_i$ and call it a {\em real circle} of $\Sigma$.
So the real locus $\Sigma\us$ of $\Sigma$ consists of these $n$ circles. 
We call $\Sigma\us_n$ the {\em central circle}.
The central sphere $\ms T\us_n$ is a double cover of the disk in ${\rm C}(\Lmd)$ bounded by this circle.
For an index $i$ such that $\pi\inv(I_i)$ is not a real circle (i.e.,\,if $i\not\equiv n\,(2)$), we denote it by $\Sigma_i$ and call it 
{\em pure imaginary circle} because the real structure $\sigma$ acts on these circles as a reflection which fixes the two ramification points $r_i$ and $r_{i+1}$.
These are exactly the locus where $\sigma$ and the hyperelliptic involution coincide, and also the fixed points set of another anti-holomorphic involution $\sigma\circ\tau$ on $\Sigma$.
From these definitions, $\pi\inv(\Lmd\us)\subset\Sigma$ is a union of the $n$ real circles and $n$ pure imaginary circles.
Explicitl, from \eqref{itv},
\begin{equation}
  \begin{minipage}{0.9\linewidth}
  \begin{itemize}
    \item $\pi\inv(\Lmd\us) = 
    \Sigma_0\cup \Sigma_1\us \cup\Sigma_2\cup\Sigma\us _3\cup\dots\cup\Sigma_{2n-2}\cup\Sigma\us_{2n-1}$ if $n$ is odd,
    \item $\pi\inv(\Lmd\us) = 
    \Sigma\us_0\cup \Sigma_1 \cup\Sigma\us_2\cup\Sigma _3\cup\dots\cup\Sigma\us_{2n-2}\cup\Sigma_{2n-1}$ if $n$ is  even.
  \end{itemize}\label{itv2}
  \end{minipage}
\end{equation}
This is a cycle of the $2n$ circles joined at the ramification points, so that the two end circles are also joined. See Figure \ref{f:Rsurf}.

\begin{figure}
\includegraphics{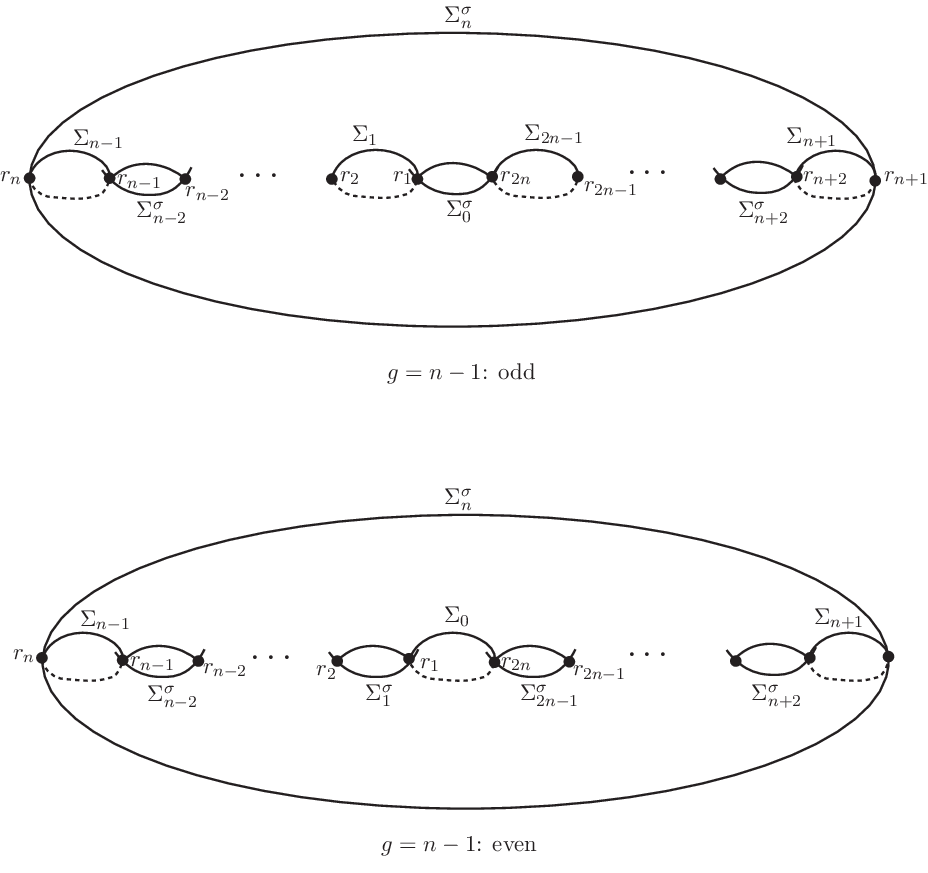}
\caption{The real and pure imaginary circles in $\Sigma$}
\label{f:Rsurf}
\end{figure}

\subsection{Minitwitor lines parameterized by the quarter of $\Sigma$}\label{ss:2mtl}
The curve $\Sigma$ is divided by $\sigma$ into two parts whose common boundary consists of the real circles.  
We denote by $\Sigma'$ one of these halves and we let $\Sigma'$ include the boundary of $\Sigma'$.  
Then
\[
\Sigma = \Sigma' \cup \sigma(\Sigma'), 
\qquad 
\Sigma' \cap \sigma(\Sigma') = \partial \Sigma',
\]
where $\partial \Sigma'$ denotes the boundary of $\Sigma'$.
In Figure~\ref{f:Rsurf}, the Riemann surface $\Sigma$ is drawn so that the $n$ ramification points $r_1,\dots,r_n$ (resp.\,$r_{n+1},\dots,r_{2n}$) lie on the left half (resp.\,the right half) of the surface.  
With this convention, the real structure $\sigma$ acts on $\Sigma$ as the reflection across the horizontal plane containing the real circles, while the hyperelliptic involution $\tau$ acts as a half-rotation about the line passing through $r_1,\dots,r_{2n}$.  
The half $\Sigma'$ may be regarded as the upper half of $\Sigma$.  
Note that in \cite[Figure~1]{H25}, $\Sigma$ was drawn slightly differently; as a result, $\sigma$ exchanged the front and back sides.

The circle over the interval $I_0$ is always located in the middle.  
If $g=n-1$ is even, then it is a real circle $\Sigma\us_0$ going once around the middle hole of $\Sigma$,
whereas if $g$ is odd, then it is a purely imaginary circle $\Sigma_0$.
On the other hand, the outer great circle is always real; it is the central circle $\Sigma\us_{g+1}=\Sigma\us_n$.  
Like the central sphere $C_n$, this circle will be directly related to the conformal infinity of the EW space.

We define two divisors $D_{\rm L}$ and $D_{\rm R}$ of degree $n=g+1$ on $\Sigma$ by 
\begin{align}\label{pen1}
D_{\rm L}= \sum_{i=1}^n r_i
\qandq 
D_{\rm R}= \sum_{i=n+1}^{2n} r_i.
\end{align}
As in Proposition \ref{p:Dt}, these are two limiting divisors of the real pencil arising from twistor lines intersecting the circle $\gamma_n\subset C_n$.
In particular, $D_{\rm L}$ and $D_{\rm R}$ are linearly equivalent and $\dim |D_{\rm L}|=1\,(=\dim |D_{\rm R}|)$.
The next proposition means that the parameter space of this pencil is naturally identified with the central circle $\Sigma\us_n$.

\begin{proposition}\label{p:rp}
Any member of the real pencil $|D_{\rm L}|\us$ has a unique point $q$ belonging to the central circle $\Sigma\us_n$, and $q$ is not a multiple point of the member.
\end{proposition}

\proof
From \cite[Proposition 2.3]{H25}, the pencil $|D_{\rm L}|$ is base point free.
Hence, for any point $q\in\Sigma$, there exists a unique member $D\in|D_{\rm L}|$ such that $q\le D$. But the pencil has the member $\sum_{i=1}^n r_i$ and among these $n$ points, only the point $r_n$ belongs to the central circle $\Sigma\us_n$.
Since the real circles $\Sigma\us_1,\dots,\Sigma\us_n$ are mutually disjoint, by continuity, this implies that for any real member $D\in|D_{\rm L}|$, there exists a unique point $q\in \Sigma\us_n$ such that $q\le D$. Further, the multiplicity of $q$ is obviously one. 
\proofend

\medskip

Next, for the circle $\Sigma_i$ which is a pure imaginary one, we put
$$
\Sigma'_i:=\Sigma_i\cap \Sigma'.
$$
This is half of $\Sigma_i$.
These $g+1$ arcs divide the half $\Sigma'$ into further
halves.
Let $\Sigma''$ be (any) one of these halves and call it the {\em quarter} of $\Sigma$.
If $\Sigma\us_i$ is a real circle (i.e.,\,if $i\equiv n\,(2)$), then 
we define half of $\Sigma\us_i$ by  
$$(\Sigma_i^{\sigma})':=\Sigma\us_i\cap \Sigma''.$$
Then the boundary of the quarter $\Sigma''$ is a  ``cycle'' of $2g+2=2n$ semicircles, half of which are real and the remaining half are pure imaginary. (See Figure \ref{f:quarter}.)

\begin{figure}
\includegraphics{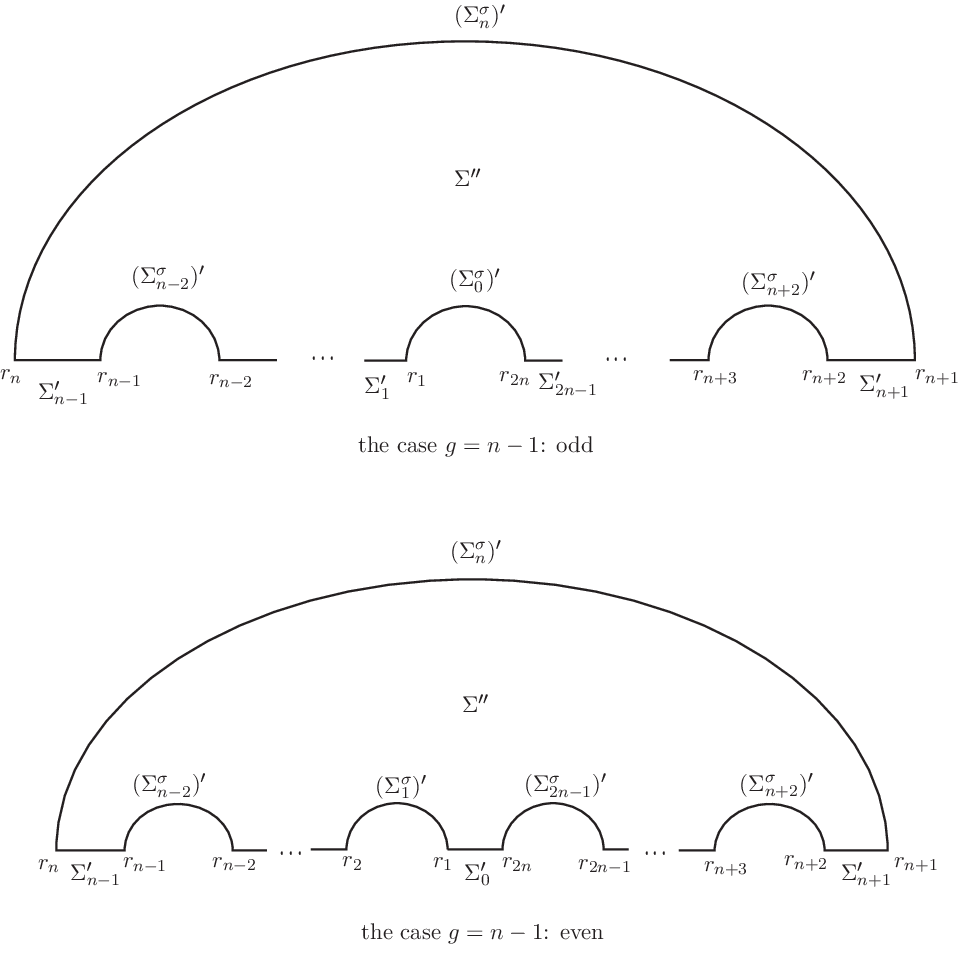}
\caption{The quarter $\Sigma''$ of $\Sigma$}
\label{f:quarter}
\end{figure}

We define 
\begin{align}\label{bmA}
\bm A':=(\ptl\Sigma'')\minus(\Sigma\us_n)'.
\end{align}
Among the semicircles of $\bm A'$,
$\Sigma'_0$ or $(\Sigma\us_0)'$ is at the middle.
Note that this is at the opposite side of the central semicircle. 
Evidently the double covering $\pi:\Sigma\lras\Lmd$ gives a homeomorphism from $(\Sigma\us_i)'$ or $\Sigma'_i$
to the arc $I_i=[a_i,a_{i+1}]$ for any index $0\le i<2n$.
Hence, $\ptl\Sigma''$ is half of $\pi\inv(\Lmd\us)$ and $\pi$ induces a homeomorphism $\ptl\Sigma''\simeq\Lmd\us$.
Note that, if we consider $\Sigma$ as a subset of $\tilde Z_1=\tilde p\inv(1)=\tilde{\ms T}$, then $\pi:\Sigma\lras\Lmd$
is identified with the restriction $\Phi|_{\Sigma}$, where as before $\Phi:Z\lras\ms O(2)$ is the projection to the standard minitwistor space.

We recall that at the end of Section \ref{s:mtl}, after choosing a half $\ccc'_n$ of the circle $\ccc_n$ in the central sphere $C_n$, we have obtained 
the family of real hyperplanes in $\PP^{n+1}$ which are parameterized by $\Lmd\us\simeq S^1$.
The sections of the cone ${\rm C}(\Lmd)$ by these hyperplanes were the images under the double covering map $\Pi:\ms T\lras {\rm C}(\Lmd)$ of the meromorphic images $\tilde\Psi(L)$ of $L\subset Z$ intersecting components $C_i$ with $i\neq n$ or the semicircle $\gamma'_n\subset C_n$.
From the above identification $\Sigma'_i\simeq I_i$ or $(\Sigma\us_i)'\simeq I_i$ induced by the cone projection, these together give a homeomorphism $\ptl\Sigma''\lras\Lmd\us$.
In the following, we regard the parameter space of the above family of hyperplanes as $\ptl\Sigma''$ rather than $\Lmd\us$ through this homeomorphism.

In the rest of this paper, for any point $q\in \ptl\Sigma''$, we denote by $h_q\subset\PP^{n+1}$ the real hyperplane in the family, determined by the point $q$ (i.e.,\,by the point $\pi(q)\in\Lmd\us$).
Then from their explicit description given in the previous section, $h_q$ always passes through $q$.
Moreover, the hyperplane $h_q$ satisfies the following property, which is important for us.

\begin{proposition}\label{p:bdry}
For any point $q\in\ptl\Sigma''$, the real hyperplane $h_q\subset\PP^{n+1}$ satisfies
\begin{align}\label{hqS}
h_q|_{\Sigma} = q + \ol q + 2D'_q
\end{align}
for some real effective divisor $D'_q$ of degree $g=n-1$ on $\Sigma$.
\end{proposition}

\proof
The case $q\in (\Sigma\us_n)'$ is immediate since $h_q$ is a real evenly tangential hyperplane to $\Sigma$ through $q=\ol q$ and hence we can write $h_q|_{\Sigma}$ as in \eqref{hqS}.

For other cases, we write the condition \eqref{hqS} as $D'_q=\frac12(h_q|_{\Sigma} - q-\ol q)$. The RHS is real as $h_q$ is real. So $D'_q$ is real.
Since $\deg\Lmd=n$ and $\pi:\Sigma\lras\Lmd$ is a double cover, $\deg
(h_q|_{\Sigma} ) = 2n$. Therefore, the degree of $D'_q$ is $(n-1)$.
Hence, to prove the proposition, it is enough to show that $\frac12(h_q|_{\Sigma} - q-\ol q)$ is effective and integral. 

In the following, for simplicity of notation, even if the index $i$ is such that $\pi\inv(I_i)$ is a real circle (i.e.,\,if $i\equiv n\, (2)$), we write its half by $\Sigma'_i$ (instead of $(\Sigma\us_i)'$).
Let $\tau:\Sigma\lras\Sigma$ be the hyperelliptic involution.
Suppose that the index $i$ satisfies $0\le i<n$ and take any point $q\in \Sigma'_i$.
From Corollaries \ref{c:im1} and \ref{c:im2}, 
\begin{align}\label{hq}
h_q|_{\Sigma} = \sum_{j=1}^{i}2r_j+ (n-i)\big(q+\tau(q)\big).
\end{align}
(If $i=0$, then the sum disappears.)
This means 
\begin{align}\label{D'q}
D'_q &= \sum_{j=1}^{i}r_j + \frac{n-i}2\big(q+\tau(q)\big)-\frac12(q+\ol q)\\
&= \sum_{j=1}^{i}r_j + \frac{n-i-1}2q - \frac12\ol q + \frac{n-i}2\tau(q)
\end{align}
Assume moreover that $q$ is a real point. Then this can be written $\sum_{j=1}^{i}r_j + \frac12(n-i-2)q+\frac{n-i}2\tau(q)$.
On the other hand, from \eqref{itv}, the reality of $q$ means that the parities of $n$ and $i$ are equal.
Hence, noting that $i\neq n-1$ as $q$ is a real point, $D'_q$ is integral and effective. 

Next, assume that $q$ belongs to a pure imaginary circle. Then $\tau(q) = \ol q$. So from \eqref{D'q}, $D'_q = \sum_{j=1}^{i}r_j + \frac{n-i-1}2q+\frac{n-i-1}2\ol q$.
On the other hand, again from \eqref{itv}, the parities of $n$ and $i$ are not equal.
Again this readily means that the divisor $D'_q$ is integral and effective.

The case where the index $i$ satisfies $n<i< 2n$ can be seen in a similar way using
$$
h_q|_{\Sigma} = \sum_{j=i+1}^{2n}2r_j+ (i-n)\big(q+\tau(q)\big)
$$
instead of \eqref{hq}, which again follows from Corollaries \ref{c:im1} and \ref{c:im2}.
\proofend

\iffalse
\begin{proposition}\label{p:tbd}
In the situation of the previous proposition, if the point $q\in \ptl\Sigma''$ belongs to the central circle $\Sigma\us_n$, then the degree $n$ divisor $q + D'_q$ is linearly equivalent to the pencil $|D_{\rm L}|=|D_{\rm R}|$ on $\Sigma$.
\end{proposition}
\fi

\medskip
To extend the family $\{h_q \mid q\in \partial\Sigma''\}$ to a family parameterized by the whole quarter $\Sigma''$ while preserving the property~\eqref{hqS}, like the method employed in~\cite{H25}, we make use of the Jacobian variety and the Abel--Jacobi map of~$\Sigma$.
Let ${\rm J}_{\Sigma}$ be the Jacobian variety of the hyperelliptic curve $\Sigma$, and let $\mathfrak a:{\rm Div}(\Sigma)\to {\rm J}_{\Sigma}$ be the Abel--Jacobi map with base point $r_1$, one of the ramification points.
Since $r_1$ is a real point, the real structure $\sigma$ on $\Sigma$ induces a real structure on ${\rm J}_{\Sigma}$, which is simply complex conjugation.
Let $({\rm J}_{\Sigma})\us_{\mathfrak o}$ be the identity component of the real locus $({\rm J}_{\Sigma})\us$ of ${\rm J}_{\Sigma}$.
This is a $g$-dimensional real torus; see \cite[Section~2.2]{H25}.
Let $({\rm J}_{\Sigma})\us_{\mathfrak o}$ be the identity component 
of the real locus $({\rm J}_{\Sigma})\us$.
This is a $g$-dimensional real torus; see \cite[Section~2.2]{H25}.
Let $\mathfrak t:{\rm J}_{\Sigma}\to {\rm J}_{\Sigma}$ be 
the doubling morphism defined by $\mathfrak t(x)=2x$.
Since the base point $r_1$ is real, this map preserves the real structure and hence maps $({\rm J}_{\Sigma})\us_{\mathfrak o}$ to itself.
Let 
\[
\mathfrak t_{\mathfrak o}:({\rm J}_{\Sigma})\us_{\mathfrak o}\lras ({\rm J}_{\Sigma})\us_{\mathfrak o}
\]
be the restriction of\/ $\mathfrak t$ to $({\rm J}_{\Sigma})\us_{\mathfrak o}$.
This is the quotient map by the $2$-torsion subgroup and is therefore a $2^{g}$-fold covering of the torus.
The same statement holds for $-\mathfrak t_{\mathfrak o}$.
In the proof of the next proposition, we use $-\mathfrak t_{\mathfrak o}$ rather than $\mathfrak t_{\mathfrak o}$.

\begin{proposition}\label{p:ext}
The family of hyperplanes in Proposition~\ref{p:bdry}, originally parameterized by the boundary $\partial\Sigma''$, admits a natural extension to a family of real hyperplanes parameterized by the entire surface $\Sigma''$, while preserving the tangential property~\eqref{hqS}.
\end{proposition}
\proof
We define a mapping $\beta:\Sigma\lras {\rm J}\us_{\Sigma}$ by $\beta(q) = \mf a(q+\ol q)$.
This is a continuous mapping and therefore the image is contained in the component 
$({\rm J}_{\Sigma})\us_{\mf o}$.

Suppose $g=1$. Then we can realize the elliptic curve $\Sigma$ as the quotient $\CC/(\ZZ+\sqrt{-1}b\ZZ)$, with some real $b\ge 1$. In this case, fixing $r_1$ as a base point, ${\rm J}_{\Sigma}$ is identified with $\Sigma$ itself by the Abel-Jacobi map $\mf a$, and $\mf a$ is the identity map.
Therefore, $\beta$ is concretely written 
\begin{align}\label{beta}
q=\big[x+\sqrt{-1}b\, y\big]\stackrel{\beta}\longmapsto [2x]\in ({\rm J}_{\Sigma})\us_{\mf o}=\big\{[x']\set x'\in\RR\big\}\simeq S^1.
\end{align}
If $h$ denotes the hyperplane class on $\Sigma\subset\PP^{n+1}=\PP^3$, then under the present choice of the base point,
the equation $q+\ol q + 2p\sim h$ (linear equivalence) is equivalent to $\beta(q)=-2\mf a(p)$.
From \eqref{beta}, if we write $q= [x+\sqrt{-1}by]$ and $p=[\xi+\sqrt{-1}b\eta]$, this can be rewritten as $[2x] = [-2\xi]$. This equation, where the unknown is $[\xi+\sqrt{-1}b\eta]$, has an obvious solution $[\xi]=[-x]$, and it is not difficult to see that this is exactly the solution we have obtained in Proposition \ref{p:bdry} for the case $g=1$.
The continuous extendability of this solution to the whole quarter $\Sigma''$ 
can be seen by gradually shrinking the lower-left quarter square in the fundamental domain 
$
\{\,x+\sqrt{-1}\,b y \mid 0\le x,y\le 1\,\}
$
of $\Sigma$, which represents the quarter $\Sigma''$, 
to the midpoint $\tfrac14(1+\sqrt{-1}\,b)$ of that square.

Next, suppose $g>1$. 
If a point $q\in\Sigma''$ belongs to a pure imaginary semicircle, then $\tau(q)=\ol q$.
As $\mf a(\tau(q)) = -\mf a(q)$, this implies $\mf a(q+\ol q) = \mf a(q+\tau( q)) =\mf a(q) - \mf a(q) = \mf o$. Hence, the map $\beta$ maps all points on pure imaginary semicircles to the origin.
Conversely, if two points $q$ and $q'$ of $\Sigma''$ satisfy $\beta(q)=\beta(q')$,
then by Abel's theorem, we have a linear equivalence $q+\ol q\sim q'+\ol q'$.
From \cite[Proposition 2.1]{H25}, this implies that $\ol q = \tau(q)$, which means that $q$ belongs to a pure imaginary semicircle.
Hence, $\beta$ is injective away from the pure imaginary semicircles.
So if $g>1$, then $\beta:\Sigma''\lras ({\rm J}_{\Sigma})\us_{\mf o}$ is simply the map that identifies all points on the pure imaginary semicircles.
We denote $\mf V:=\beta(\Sigma'')\subset ({\rm J}_{\Sigma})\us_{\mf o}$ for the image. %, and call it the Seifert surface 
This has the images of the real semicircles as its boundary, and all real semicircles of $\ptl\Sigma''$ become loops that contain $\mf o$.

Therefore, any point of the boundary $\ptl\mf V$ may be written $\mf a(2q)$ where $q$ belongs to a real semicircle of $\ptl\Sigma''$, and
from Proposition \ref{p:bdry}, the hyperplane $h_q$ satisfies $h_q|_{\Sigma} = 2q + 2D'_q$ for an effective divisor $D'_q$ of degree $g$.
Since $\mf a(h_q|_{\Sigma})=\mf a(2nr_1) = \mf o$, taking the image under $\mf a$, we obtain $\mf a(2D'_q) = -\mf a(2q)$.
As $\mf a(2q) = \mf t_{\mf o}(\mf a(q))$, this means $\mf a(2D'_q) = -\mf t_{\mf o}(\mf a(q))$. 
Thus, Proposition \ref{p:bdry} means that the mapping $\mf a(2q)\longmapsto \mf a(D'_q)$ from $\ptl\mf V$ to $({\rm J}_{\Sigma})\us_{\mf o}$ provides a lift of $\ptl\mf V$ to $({\rm J}_{\Sigma})\us_{\mf o}$ for the $2^g$-fold covering $-\mf t_{\mf o}:({\rm J}_{\Sigma})\us_{\mf o}\lras ({\rm J}_{\Sigma})\us_{\mf o}$.
Since the divisor $D'_q$ varies continuously as seen at the end of Section \ref{s:mtl}, this lift is continuous.
Furthermore, the lift is a closed curve.
Since the fundamental group of $\mf V$ is generated by the loops of the images under $\mf a$ of the real semicircles of $\Sigma$ from the above topological discription of $\mf V$, this implies that the lift bounds a simply connected domain in $({\rm J}_{\Sigma})\us_{\mf o}$.
Hence, the lift along $\ptl\Sigma''$ uniquely extends to the entire surface $\Sigma''$.
If $D'_q$ is a degree $g=n-1$ divisor such that $\mf a(D'_q)$ is the lift of a point $\mf a(\tau(q)+\tau(\ol q))=-\mf a(q+\ol q)\in\mf V$, then it satisfies $-\mf t_{\mf o}(\mf a(D'_q)) = \mf a(q+\ol q)$, which is equivalent to the linear equivalence $2D'_q+q+\ol q\sim h$, where $h$ is the hyperplane class.
This means the existence of the hyperplane $h_q$ in the proposition.
\proofend

\medskip
The surface $\mathfrak V$ in the Jacobian is an analogue of the surface $\mathfrak S_k$ used in~\cite{H25} to obtain a two-dimensional family of real minitwistor lines in the same minitwistor space $\ms T$, and we again call it the \emph{Seifert surface}, 
since it is determined by its boundary curve.

\subsection{The complete family of minitwistor lines}\label{ss:cplt}
To obtain all minitwistor lnes that correspond to the present EW space, let $M$ be a toric ALE gravitational instanton of type $A_{2n-1}$ we have considered and $\hat M:=M\cup\{\infty\}$ be its orbifold compactification. 
This has the scalar $S^1$-action in particular. We denote $\hat W:=\hat M/S^1$ for the quotient space by this action, and $\varpi:\hat M\lras\hat W$ for the quotient map.
The components $C_1,C_2,\dots, C_{2n-1}$ of the exceptional curves of the minimal resolution of the $A_{2n-1}$-singularity
are invariant under the scalar $S^1$-action. Only the central component $C_n$ is pointwise fixed by the scalar $S^1$-action and it is mapped isomorphically to the boundary $\ptl \hat W$ of $\hat W$ by $\varpi$.
Hence, there is a natural identification 
\begin{align}\label{bdry}
\ptl \hat W\simeq C_n \,\,(\simeq S^2).
\end{align}
We denote $W:=\hat W\minus\ptl \hat W$ for the interior of $\hat W$.
For each $i=1,2,\dots, 2n-1$ with $i\neq n$, we denote $\gamma_i:=\varpi(C_i)\simeq C_i/S^1$. These are segments in $\hat W$ and contained in the interior $W$ if $i\neq n-1,n+1$.
The two images $\ccc_{n-1}$ and $\ccc_{n+1}$ may also be identified with segments, but the images of the two points $C_{n-1}\cap C_n$ and $C_{n+1}\cap C_n$ belong to $\ptl \hat W$.
Both of the unions $\ccc_1\cup\ccc_2\cup\dots\cup \ccc_{n-1}$ and $\ccc_{n+1}\cup\ccc_{n+2}\cup\dots\cup \ccc_{2n-1}$ are connected but these two unions are mutually disjoint.
On the other hand, noting that the two components $C_0$ and $C_{2n}$ in $\hat M$ share the point $\infty$ at infinity,  the image $\varpi(C_0\cup C_{2n})$ is a single segment which has the point $\varpi(\infty)$ as an interior point.
We denote $\ccc_0:=\varpi(C_0\cup C_{2n})$ for this segment. 
(So $\ccc_0$ is not $\varpi(C_0)$.)
This connects the endpoint of $\ccc_{n+1}\cup\dots\cup \ccc_{2n-1}$ and the head point of $\ccc_1\cup\dots\cup \ccc_{n-1}$, and the union 
\begin{align}\label{axis}
\bm A:=(\ccc_{n+1}\cup\dots\cup \ccc_{2n-1})\cup\ccc_0\cup(\ccc_1\cup\dots\cup \ccc_{n-1})
\end{align}
is a connected long path whose endpoints belong to $\ptl \hat W$. (See Figure \ref{f:W}.)
The quotient group $T^2/S^1\simeq S^1$, where $S^1$ in the quotient is the subgroup of scalars, acts on $\hat W$ and this $S^1$-action can be considered as a rotation around the axis \eqref{axis}.
We call \eqref{axis} the {\em rotational axis} of $\hat W$ or $W$.

\begin{figure}
\includegraphics{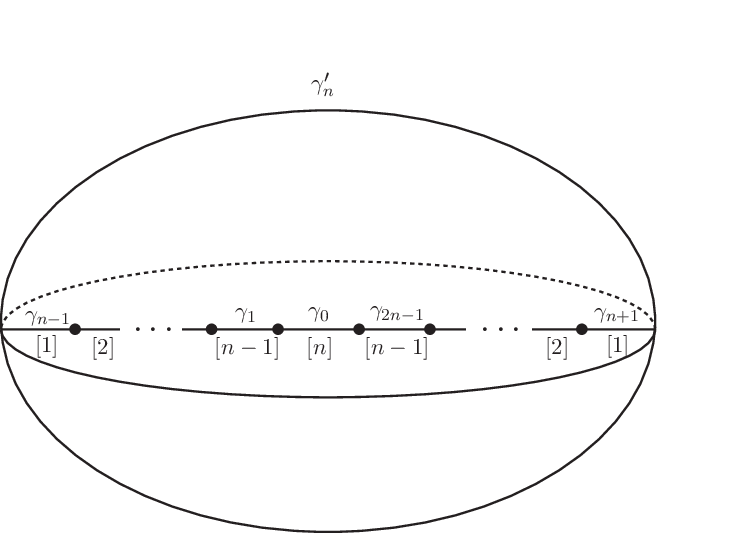}
\caption{The EW space $W$ and the axis of rotation. (The numbers in the brackets refer to the order of singularity.)}
\label{f:W}
\end{figure}

In Section \ref{s:mtl} we have defined a circle $\gamma_n$ and its half $\gamma'_n$ in the central sphere $C_n$.
By the quotient map $\varpi$, the latter is mapped isomorphically to its image and we denote it by the same letter $\gamma'_n$.
Then the union of the axis \eqref{axis} and $\gamma'_n$ is homeomorphic to $S^1$. We denote this circle in $\hat W$ by $\hat\gamma$.

As discussed at the end of Section \ref{s:mtl}, for any $0\le i\le 2n$, the composition $\Phi\circ\alpha:Z_0\lras \Lmd$ (see \eqref{alpha}) maps the sphere $C_i$ to the arc $I_i=[a_i,a_{i+1}]\subset\Lmd\us$ and if $i\neq n$ this descends to an isomorphism from $C_i/S^1\simeq\gamma_i$ to $I_i$.
Hence, using the isomorphism from $\Sigma'_i$ or $(\Sigma\us_i)'$ to $I_i$ induced by the double covering $\pi:\Sigma\lras\Lmd$, the path \eqref{axis} is naturally identified with the set $\bm A'=(\ptl\Sigma'')\minus(\Sigma\us_n)'$ defined in \eqref{bmA},
where each segment $\gamma_i$ is identified with the semicircle $\Sigma'_i$ or $(\Sigma\us_i)'$ depending on the parity of $n$.
Namely, there is a natural identification $\bm A\simeq\bm A'$.

For any index $i$, the weight of the scalar $S^1$-action on $C_i$ is $|n-i|$. 
Hence, for $i\neq n$, the curve $C_i$ has a non-trivial stabilizer subgroup unless $i=n\pm1$. 
Accordingly, the space $W$ has cyclic orbifold singularities of order $|n-i|$ along $\ccc_i$. 
The space $W$ is smooth along the two segments $\ccc_{n\pm1}$, except at their intersection points with the adjacent segments $\gamma_{n\pm2}$, respectively.
On the middle segment $\gamma_0$ of the axis~\eqref{axis}, the order of the singularity is exactly $n$, which is the largest among all segments.
Any intersection point of two adjacent segments of~\eqref{axis} is not an orbifold point, since the orders of the singularities along the two adjacent segments are different.

By a theorem of Jones--Tod \cite{JT85}, the interior $W$ admits an EW structure away from the non-orbifold points.
Since $\ms T$ is obtained by a dimensional reduction of the twistor space $Z$ of the gravitational instanton, $\ms T$ may be regarded as the minitwistor space corresponding to this EW space.
The minitwistor lines on $\ms T$ that induce the EW structure on $W$ are the images of twistor lines in $Z$.
In Section \ref{s:mtl}, we identified the images of twistor lines through the chain $C_0\cup\dots\cup C_{2n}$.
The images of twistor lines through $C_i$, $0\le i\le 2n$ and $i\neq n$, are minitwistor lines corresponding to points on the rotational axis \eqref{axis}.
In particular, from Propositions \ref{p:a} and \ref{p:i1}, which determine the images of twistor lines through $C_i$ for $i\neq n$, the multiplicity of the conic component $f_{\lmd}$ of the image is equal to the order $|n-i|$ of the singularities along $\gamma_i$.
This is consistent with a natural expectation that the minitwistor lines corresponding to orbifold points of an EW space should have multiplicity equal to the order of the orbifold singularities.
Also, in Propositions \ref{p:l} and \ref{p:tan}, we identified the (meromorphic) images of twistor lines through the central sphere $C_n$, and they consist of two rational curves that are mutually $\sigma$-conjugate.
They may be regarded as the ``minitwistor lines'' corresponding to points of the boundary $\partial \hat W$.
Here, we use quotation marks because they correspond to boundary points rather than interior points.
Thus, the minitwistor lines corresponding to points on the rotational axis \eqref{axis} and the ``minitwistor lines'' corresponding to the boundary \eqref{bdry} are explicitly identified.
Using the hyperplanes $h_q\subset\PP^{n+1}$ as in Proposition \ref{p:bdry}, if we put 
\begin{align}\label{Hq}
H_q:=\Pi\inv(h_q)\subset\PP^{n+2},
\end{align}
then the collection of minitwistor lines which correspond to points of the axis \eqref{axis} and the semicircle $\gamma'_n$ is the collection $\{H_q\cap \ms T\set q\in\ptl\Sigma''\}$.
From Proposition \ref{p:ext}, this family has a continuous extension $\{H_q\cap \ms T\set q\in\Sigma''\}$ to the whole quarter, where $H_q$ is still defined by $h_q\subset\PP^{n+1}$ as \eqref{Hq}.
For members of this family, we have the following

\begin{proposition}\label{p:inn}
Let $D'_q$ be the real degree $(n-1)$ divisor on $\Sigma$ that satisfies $h_q|_{\Sigma} = q+\ol q + 2D'_q$ as in Proposition \ref{p:ext}.
If a point $q\in \Sigma''\minus\ptl\Sigma''$ does not satisfy $q\le D'_q$, then $H_q\cap\ms T$ is a real minitwitsor line on $\ms T$ in the sense that it defines a smooth EW structure on the smooth locus of the space $W$.
\end{proposition}

See Remark \ref{r:cusp} for the assumption $q\not\le D'_q$.
If $g=1$, then this assumption is satisfied for any $q\in \Sigma''\minus\ptl\Sigma''$ because $\deg D'_q=1$ as $g=1$, which means that $D'_q$ consists of a single {\rm real} point $p$ of multiplicity one and therefore $q\neq p$ since $q\neq\ol q$ as $q\not\in\ptl\Sigma''$.
In the case $g>1$, we do not know whether $q\not\le D'_q$ holds for all $q\in \Sigma''\minus\ptl\Sigma''$.
Apart from this problem, we note that the proposition does not assert that $H_q\cap\ms T$ always has exactly $g\,(=n-1)$ nodes.
In fact, we will prove that it can have other types of singularities like a tacnode (Proposition \ref{p:re}). In that case the number of singularities of $H_q\cap\ms T$ will be smaller; see the proof of Proposition \ref{p:inn} below.
For a proof of Proposition \ref{p:inn}, we prepare a lemma.

\begin{lemma}\label{l:inn1}
If $q\in \Sigma''\minus\ptl\Sigma''$, then 
the real hyperplane $h_q\subset\PP^{n+1}$ does not pass through the vertex of the cone ${\rm C}(\Lmd)$.
\end{lemma}

\proof 
If $h_q$ would pass through the vertex of the cone, then $h_q$ would be of the form $\pi\inv(\ul h)$ for some hyperplane $\ul h\subset\PP^n$.
Therefore, the multiplicities of $h_q|_{\Sigma}$ of the points $q$ and $\tau(q)$ are equal.
Using this as well as the fact that the points $q,\ol q,\tau(q)$ and $\tau(\ol q)$ are mutually distinct as $q\not\in\ptl\Sigma''$,
from the condition $h_q|_{\Sigma} = q + \ol q + 2D'_q$, we readily obtain $h_q|_{\Sigma}\ge m(q + \ol q+\tau(q) + \tau(\ol q))$ for arbitrary $m>0$.
Of course, this cannot happen. Hence, $h_q$ does not pass through the vertex of the cone.
\proofend

\medskip
\noindent{\em Proof of Proposition \ref{p:inn}.}
By the previous lemma, the cut $h_q\cap {\rm C}(\Lmd)$ is a smooth rational curve.
Write the divisor $D'_q$ as $D'_q = m_1p_1+ \dots + m_lp_l$ where $p_1,\dots, p_l\in \Sigma$ are distinct. Then $\sum_{1\le i\le l} m_i = g$.
As $h_q|_{\Sigma} = q+\ol q+ 2 D'_q$, the branch points of the double covering $H_q\cap \ms T\lras h_q\cap {\rm C}(\Lmd)$ consists of $q, \ol q$ and $p_1,\dots, p_l$.
Since $q\neq \ol q$ as $q\not\in\ptl\Sigma''$, by the assumption $q\not\le D_q$, $q$ and $\ol q$ are simple branch points.
Hence, $C:= H_q\cap \ms T=\Pi\inv(h_q\cap {\rm C}(\Lmd))$ is an irreducible curve.
Using that $C$ is a double cover of $h_q\cap {\rm C}(\Lmd)\simeq\PP^1$ whose branch divisor is of degree $2n$, 
the arithmetic genus of $C$ is $n$. 
It is not difficult to see (without using the assumption $q\not\le D_q$) that the sum of genus drops of the present singularities of $C$ is also $n$.
Therefore, $C$ is a rational curve.

If the divisor $D'_q$ has no multiple components, then $C$ has exactly $g$ nodes.
Further, we always have $C^2 = 2\deg {\rm C}(\Lmd) = 2n = 2g+2$.
From \cite{HN11}, this implies that $C$ is a nodal minitwistor line and that the space of such rational curves is smooth 3-dimensional and has a natural (smooth) EW structure.
To investigate the general case, write $D'_q = m_1p_1+ \dots + m_lp_l$ for distinct points $p_1,\dots,p_l$ as above, with $\sum_{i=1}^l m_i = g$.
Then over a point $p_i$, using that $q\neq p_i$ for any $i$,  $C$ has $A_{2m_i}$-singularity.
Even if $m_i>1$ for some $i$, because the multiplicity of $h_q|_{\Sigma}$ at $p_i$ is even and hence $C$ has two components around $p_i$, the argument of taking the ``normalization'' of a neighborhood of $C$ in \cite{HN11} still works by just noticing that the self-intersection number in the ``normalization'' drops by $2m_i$ for each $i$.
Hence, the compact component of the inverse image of $C$ into the ``normalization'' has 
$$
2n - \sum_{i-1}^l 2m_i = 2(n - g) = 2
$$
as the self-intersection number.
Consequently, such a curve can also be regarded as a minitwistor line which corresponds to a smooth point of the EW space. \proofend

\begin{remark}\label{r:cusp}{\em
If a point $q\in \Sigma''\minus\ptl\Sigma''$ would satisfy $q\le D'_q$, then the divisor $h_q|_{\Sigma}$ would contain the points $q$ and $\ol q$ with odd ($\ge 3$) multiplicity. This implies that $\ms T\cap H_q$ has (not necessarily ordinary) cuspidal singularities over $q$ and $\ol q$. This seems to mean that the EW structure does not extend to the corresponding point of the quotient space, at least smoothly. However, by a theorem of Jones-Tod \cite{JT85}, the quotient space has a smooth EW structure on the smooth locus, and the space $W$ is indeed smooth at the point determined by $q\in \Sigma''\minus\ptl\Sigma''$.
Therefore, we suspect that the appearance of a point $q\in\Sigma''\minus\ptl\Sigma''$ that satisfies $q\le D'_q$.
}\end{remark}

We recall that the tri-holomorphic $S^1$-action on $\tilde Z$, which is given by \eqref{act-t}, induces an $S^1$-action on the linear system $|\bm L|^{\CC^*_s}\simeq\PP^{n+2}$ and it induces the residual $S^1$-action on $\ms T$ and $\PP^{n+2}$.

\begin{proposition}\label{p:slice}
Let $q_1$ and $q_2$ be any points of the interior $\Sigma''\minus\ptl\Sigma''$. Then the coincidence $H_{q_1} = t(H_{q_2})$ holds for some $t\in S^1\minus\{1\}$ only when $q_1=q_2$. 
\end{proposition}

\proof
Let $\bm a\in \PP^{n+2}$ be the center of the projection $\Pi:\PP^{n+2}\lras\PP^{n+1}$ and $l_{\infty}\subset\PP^{n+2}$ the line which is the center of the projection $\pi\circ\Pi:\PP^{n+2}\lras\PP^{n}$.
Then $\bm a$ belongs to $l_{\infty}$ and a hyperplane $H\subset\PP^{n+2}$ is of the form $\Pi\inv(h)$ for some hyperplane $h\subset\PP^{n+1}$ iff $\bm a\in H$.
The $S^1$-action preserves the line $l_{\infty}$ and its weight on $l_{\infty}$ is two.
Therefore only $\pm 1\in S^1$ fix the point $\bm a$. Hence, since $H_{q_1}$ and $H_{q_2}$ pass through $\bm a$, $t(H_{q_1}) = H_{q_2}$ implies $t=\pm 1$. 
Writing $H_{q_i}=\Pi\inv( h_{q_i})$ for $i=1,2$,
$(-1)H_{q_2} = (-1)\Pi\inv( h_{q_2}) = \Pi\inv((-1)h_{q_2}).
$
Since $\Pi\inv$ is injective, this implies that if $t=-1$ then $h_{q_1} = -h_{q_2}$. 
As in \cite[Proof of Proposition 3.6]{H25}, the element $-1\in S^1$ acts on the branch hyperelliptic curve $\Sigma\subset\PP^{n+1}$ as the hyperelliptic involution $\tau$.
Therefore, restricting the last equation to $\Sigma$, we obtain 
$$q_1+\ol q_1+ 2D'_{q_1} = \tau(q_2) + \tau(\ol q_2)+2\tau(D'_{q_2}).$$
Pulling all $q_i$ and $\ol q_i$ from $D'_{q_i}$ if any to obtain a divisor $D''_{q_i}\ge 0$ for $i=1,2$, this can be rewritten 
$$
(2m_1+1)\big(q_1+\ol q_1\big) + 2D''_{q_1}
=
(2m_2+1)\big(\tau(q_2) + \tau(\ol q_2)\big)+2\tau(D''_{q_2})
$$
for some non-negative integers $m_1,m_2$.
From parity of the coefficients, this means $m_1 = m_2$ and
$q_1+\ol q_1 = \tau(q_2) + \tau(\ol q_2)$.
But since $\Sigma''\minus\ptl\Sigma''$ is the interior of a fundamental domain of the action generated by $\tau$ and $\sigma$,
the four points $q_1,\ol q_1,\tau(q_2)$ and $\tau(\ol q_2)$ belong to mutually different domains. Hence $(-1)H_{q_1}\neq H_{q_2}$. Namely, $t\neq -1$. Therefore $H_{q_1}=H_{q_2}$. This means $h_{q_1} = h_{q_2}$, which implies $q_1=q_2$.
\proofend

\medskip
We have a (real) 2-dimensional family of hyperplanes $\{H_q=\Pi\inv(h_q)\set q\in\Sigma''\minus\ptl\Sigma''\}$ in $\PP^{n+2}$.
Using the residual $S^1$-action, we obtain from this a 3-dimensional family $\{t(H_q)\set q\in\Sigma''\minus\ptl\Sigma'', t\in S^1\}$.
From Proposition \ref{p:slice}, the former 2-dimensional family is a {\em slice} of the latter 3-dimensional family with respect to the residual $S^1$-action in the strict sense; namely all their orbits intersect the former family at exactly one point. 
Hence, the parameter space of this 3-dimensional family is $(\Sigma''\minus\ptl\Sigma'')\times S^1$. 
The boundary $\ptl\Sigma''$ is naturally attached to $\Sigma''\minus\ptl\Sigma''=(\Sigma''\minus\ptl\Sigma'')\times\{1\}$ and again by rotation, this gives a compactification of $(\Sigma''\minus\ptl\Sigma'')\times S^1$ such that the boudnary is $S^2$ which is formed by $S^1$-orbits through the central semicircle $(\Sigma\us_n)'=(\Sigma\us_n)'\times\{1\}$, and such that the residual locus $\bm A'=\ptl\Sigma''\minus (\Sigma\us_n)'$ is identified with the axis $\bm A$ of rotation. 
Consequently, this has the same structure as the quotient space $\hat W$ by the scalar $S^1$-action on $\hat M$.
We denote this space (obtained as the compactification of $(\Sigma''\minus\ptl\Sigma'')\times S^1$) by $\hat W'$.
The added locus in the compactification also parameterizes hyperplane sections of $\ms T$, and it provides a continuous extension of the family parameterized by $(\Sigma''\minus\ptl\Sigma'')\times S^1$.
The quarter $\Sigma''$ of $\Sigma$ is naturally embedded in $\hat W'$ as the closure of $(\Sigma''\minus\ptl\Sigma'')\times \{1\}$, and in the following we regard $\Sigma''$ as a subset of $\hat W'$.
By construction, the space $\hat W$ is just a quotient space of $\hat M$, while $\hat W'$ is a parameter space of curves in $\ms T$.

By construction, it would be natural to expect that there exists a natural isomorphism between $\hat W$ and $\hat W'$. In the following, we show that this is the case.
For this, we first show the following proposition. Let $W'$ be the interior of $\hat W'$ and recall that $\bm A'=(\ptl\Sigma'')\minus(\Sigma\us_n)'$ is the axis of rotation of $W'$ as above. 

\begin{proposition}\label{p:im}
Let $L\subset Z$ be any twistor line which does not pass through the chain $C_0\cup C_1\cup\dots\cup C_{2n}$. Then the image $\tilde\Psi(L)$ belongs to $W'\minus \bm A'$.
\end{proposition}

\proof
It suffices to show that there exists an element $t\in S^1$ from the tri-holomorphic action such that $\tilde\Psi(t(L))=t(\tilde\Psi(L))$ belongs to $\Sigma''\minus\ptl\Sigma''\subset W'$. By Hitchin \cite[Theorem 4.1]{Hi25}, the image $\tilde\Psi(L)$ belongs to the hyperplane section class on $\ms T\subset\PP^{n+2}$. Let $H\subset\PP^{n+2}$ be the hyperplane such that $\tilde\Psi(L)=H\cap\ms T$. This is real as $\tilde\Psi$ preserves the real structure, and moreover, it does not pass through the two singularities of $\ms T$ as $L$ is assumed not to pass through the chain $C_0\cup C_1\cup\dots\cup C_{2n}$.
Hence, if as before $l_{\infty}\subset\PP^{n+2}$ denotes the line of the center of the projection $\pi\circ\Pi:\PP^{n+2}\lras\PP^n$, then since this is the line through the two singularities of $\ms T$, $H$ does not contain $l_{\infty}$. 
Therefore, $H$ intersects $l_{\infty}$ at one point.
Let $\bm a'$ be this point. This is a real point. So it belongs to the real locus $l_{\infty}\us\simeq S^1$.
On the other hand, the center $\bm a$ of the projection $\Pi:\PP^{n+2}\lras\PP^{n+1}$ also belongs to the same circle as $\pi$ preserves the real structure.
The tri-holomorphic $S^1$-action preserves $l_{\infty}$ and its weight on $l_{\infty}$ is two.
Hence, there exits an element $t\in S^1$ such that $t(\bm a') = (-t)(\bm a') = \bm a$.

We show that either $t(L)$ or $(-t)(L)$ belongs to $\Sigma''\minus\ptl\Sigma''\,(\subset W')$, which is sufficient to prove the proposition.
Since $\bm a\in t(H)$, there exists a hyperplane $h\subset\PP^{n+1}$ such that $t(H) = \pi\inv(h)$, and it is real.
Further, since $l_{\infty}\not\subset H$ as above, $h$ does not pass through the vertex of the cone. Hence, the cut $h\cap {\rm C}(\Lmd)$ is a smooth rational curve.
Therefore, $\tilde\Psi(t(L))$ is a double cover of this curve whose branch divisor is the restriction $h|_{\Sigma}$. 
Again by \cite[Theorem 4.1]{Hi25}, if the twistor line $L$ is sufficiently general, then the image $\tilde\Psi(t(L))$ has exactly $g$ nodes as its only singularities. 
Using irreducibility of $\tilde\Psi(t(L))$,
this means that $h|_{\Sigma}$ is of the form $q+\ol q + 2D'$ for some non-real point $q$ and some real divisor $D'$ of degree $g$ which does not contain $q, \ol q$, nor a multiple point.
By exchanging $q$ and $\ol q$ if necessary, we may suppose that $q$ belongs to the half $\Sigma'$
(which includes $\Sigma''$ and $\tau(\Sigma'')$ by definition).
As $\mf a(h|_{\Sigma})=\mf o$ from our choice of the base point on $\Sigma$, using the notations from the proof of Proposition \ref{p:ext}, this means that $q,\ol q$ and $D'$ satisfy the equation
\begin{align}\label{eq}
\beta(q+\ol q) = -\mf t_{\mf o}(\mf a(D')).
\end{align}
Further, the point $q$ does not belong to $\ptl\Sigma''$ because otherwise $\ol q=\tau(q)$ as $q$ is not real as above, which means that $h$ passes through the vertex of the cone.
Hence, $q$ belongs to either the open quarter $\Sigma''\minus\ptl\Sigma''$ or another open quarter $\tau(\Sigma'')\minus\ptl\Sigma''$.
Since $t=-1$ exchanges $\Sigma''$ and $\tau(\Sigma'')$, by replacing $t$ with $-t$ if necessary, we may suppose that $q\in \Sigma''\minus\ptl\Sigma''$.

The equation \eqref{eq} means that the point $\mf a(D')\in ({\rm J}_{\Sigma})\us$ is over the Seifert surface $\mf V=\{\mf a(q+\ol q)\set q\in\Sigma''\}$ under the covering map $(-\mf t_{\mf o}$).
Let $\tilde{\mf V}$ be the surface in $({\rm J}_{\Sigma})\us_{\mf o}$ bounded by the lift of the boundary $\ptl\mf V$ we used in the proof of Proposition \ref{p:ext}.
The restriction of $(-\mf t_{\mf o})$ to $\tilde{\mf V}$ is the map that identifies the $n$ singular points of $\ptl\tilde{\mf V}$.
Then since the twistor line $t(L)$ can be continuously moved to a twistor line that intersects the chain $C_0\cup C_1\cup\dots\cup C_{2n}$, the point $\mf a(D')$ belongs to the connected component of $(-\mf t_{\mf o})\inv(\mf V)$ which is bounded by $\ptl\tilde{\mf V}$.
This implies that the hyperplane $h$ belongs to $\Sigma''\minus\ptl\Sigma''$.
Once this is shown for a generic twistor line $L$, 
it holds for an arbitrary twistor line $L$ not through the chain 
$C_0\cup C_1\cup\dots\cup C_{2n}$ by continuity.
\proofend

\medskip
For any point $p\in M$, we denote by $L_p\subset Z$ the twistor line through $p$, and using Proposition \ref{p:im}, consider the mapping from $M$ to $\hat W'$ that sends a point $p\in M$ to $\tilde\Psi(L_p)\in \hat W'$. This is a continuous mapping that sends
\begin{itemize}[itemsep=0pt]
\item %[(i)]
the component $C_i$ with $0\le i\le 2n$ and $i\neq n$ to the semicircle $\Sigma'_i$ or $(\Sigma\us_i)'$ of $\ptl\Sigma''$, 
\item %[(ii)]
the central sphere $C_n$ to the boundary $\ptl \hat W'$,
\item %[(iii)]
the open set $M\minus (C_0\cup\dots\cup C_{2n})$ to $W' \minus \bm A'$.
\end{itemize}
Further, recalling that the point $\infty\in \hat M$ is mapped by the quotient map $\varpi:\hat W\lras \hat W$ to a point that belongs to the middle arc $\gamma_0$ in the rotational axis $\bm A$ in $\hat W$ (see \eqref{axis}), we assign this point to $\infty$.
Thus we have obtained a mapping from $\hat M$ to $\hat W'$.
From the property \eqref{infty}, this is continuous also at $\infty\in\hat M$.
Since $\tilde\Psi$ is scalar $S^1$-equivariant, this descends to a continuous mapping 
$$
\bm w:\hat W\lras \hat W'.
$$
Moreover, since $\tilde\Psi$ is also equivariant with respect to the $S^1$-actions induced by the tri-holomorphic $S^1$-action, 
$\bm w$ is equivariant under this $S^1$-action.
An $S^1$-action on $\hat W$ or $\hat W'$ always refers to this action.

\begin{proposition}\label{p:cplt}
The map $\bm w$ is homeomorphic.
\end{proposition}

\proof From the $S^1$-equivariancy, we have the following commutative diagram of continuous mappings:
\[
\begin{tikzcd}
\hat W \arrow[r,"\bm w"] \arrow[d] 
  & \hat W' \arrow[d] \\
\hat W/S^1 \arrow[r,"{[\bm w]}"] 
  & \hat W'/S^1
\end{tikzcd}
\]
where $[\bm w]$ is the induced map between the orbit spaces.
The $S^1$-actions on $\hat W$ and $\hat W'$ are both semi-free in the sense that the action is free away from the fixed locus.
From the equivariance and the above properties of the map $\bm w$, this means that $\bm w$ maps $S^1$-orbit to $S^1$-orbit bijectively. 
Hence, to prove the proposition, it is enough to show that $[\bm w]$ is homeomorphic.

First, we show that both $\hat W/S^1$ and $\hat W'/S^1$ are manifolds with corners and homeomorphic to a closed disk.
For the former quotient, we recall that $\hat M=M\cup\{\infty\}$ and $M$ is a (smooth) toric surface.
Let $M_{\RR}\subset M$ be the fixed locus of the standard real structure induced by complex conjugation.
By the moment map, $M_{\RR}$ can be naturally identified with the associated polytope of the toric surface $M$. In particular, $M_{\RR}$ is identified with the quotient space $M/T^2$ and is a manifold with corners. The point $\infty$ can be naturally attached to $M_{\RR}$ to give a compactification, and let $\hat M_{\RR}=M_{\RR}\cup\{\infty\}$ be the resulting surface. This is still a manifold with corners, and as a topological space, it is homeomorphic to a closed disk. 
The quotient $\hat M/S^1$ is identified with $\hat M_{\RR}$ under the quotient map $\hat M\lras \hat M/T^2\simeq\hat W/S^1$.
Obviously, as a subset of $\hat M$, the boundary $\ptl \hat M_{\RR}$ is contained in $C_0\cup C_1\cup\dots\cup C_{2n}\cup\{\infty\}$ and it is mapped isomorphically to $\hat\gamma$, which is by definition the union of the axis $\bm A$ and the image of the semicircle $\gamma'_n\subset C_n$ by the quotient map $\varpi:\hat M\lras\hat W$.

For the latter quotient $\hat W'/S^1$,  the quarter $\Sigma''$ is naturally a manifold with corners and embedded to $\hat W'$ from our construction of $\hat W'$.
Further, $\Sigma''$ is mapped homeomorphically to the quotient $\hat W'/S^1$ by Proposition \ref{p:slice}. 
Hence, the latter quotient is also identified with a manifold with corners, and obviously, it is homeomorphic to a closed disk.
Both of the two quotients $\hat W/S^1$ and $\hat W'/S^1$ have exactly $2n$ corners and $2n$ edeges, which correspond to $T^2$-fixed points and 1-dimensional orbits, respectively.

Again by the property of the map $\bm w$, it maps the circle $\hat\gamma\subset\hat W$ to $\ptl\Sigma''\subset\hat W'$ homeomorphically, where we are thinking $\Sigma''$ as a subset of $\hat W'$ as before.
Therefore, $[\bm w]$ maps $\ptl(\hat W/S^1)$ to $\ptl(\hat W'/S^1)$ homeomorphically.
Hence, $[\bm w]$ maps the open set $\hat W\minus\ptl\hat W$ to the open set $\hat W'\minus\ptl\hat W'$.
We show that this map (from $\hat W\minus\ptl\hat W$) is injective.
It is easy to see that this is equivalent to the following: for any twistor line $L\subset Z$
that is not through $C_0\cup\dots\cup C_{2n}$, the divisor $D:=\tilde\Psi\inv(\tilde\Psi(L) )\subset\tilde Z$ does not contain a (real) twistor line $L'\subset Z$ such that $L'\not\in \{s(L)\set s\in S^1\}$. 
Let $\tilde D\lras D$ be the composition of the normalization of $D$ and a resolution of singularities of the normalization, and $\tilde L$ and $\tilde L'$ be the strict transforms into $\tilde D$ of $L$ and $L'$ respectively.
Then as $L\cap L'\neq\emptyset$, $\tilde L\cap \tilde L'\neq\emptyset$.
Further, by varying $L$ and $L'$ by the scalar $S^1$-action in $D$, we readily see that both $\tilde L$ and $\tilde L'$ have zero as self-intersection number in $\tilde D$.
So $\tilde L$ and $\tilde L'$ induce surjective holomorphic maps onto $\PP^1$.
If $L'\not\in \{s(L)\set s\in S^1\}$, then $L'$ and $L$ are not linearly equivalent because except for two $\CC^*$-invariant members, all members of the pencil generated by $L$ can be written $s(L)$ for some $s\in\CC^*$ but such a curve can be a twistor line only when $s\in S^1$.
Hence, if $L'\not\in \{s(L)\set s\in S^1\}$, then the two maps to $\PP^1$ are mutually distinct.
This contradicts $\tilde L\cap \tilde L'\neq\emptyset$. Hence, such a twistor line $L'$ does not exist and $[\bm w]$ is injective on $\hat W\minus\ptl\hat W$.

Hence, $[\bm w]$ is injective on the whole of $\hat W$.
Namely, it is a continuous injective map between closed disks that maps the boundary to the boundary homeomorphically. It is well known that such a map between closed disks is always a homeomorphism. Hence, the proposition is proved.
\proofend

\medskip
Thus, we have obtained the following.

%\begin{theorem}
%The hyperplane sections of\/ $\ms T$ parameterized by $W'$ and 
%$\ptl\hat W' \simeq S^2$ consist precisely of the minitwistor lines that 
%induce, respectively, the EW structure on $W$ obtained from a toric ALE 
%gravitational instanton of type $A_{2n-1}$ by reduction with respect to the 
%scalar $S^1$-action, and the conformal infinity of the EW space $W$.
%\end{theorem}

\begin{theorem}\label{t:main}
The hyperplane sections of $\ms T$ parameterized by $W'$ consist precisely of 
the minitwistor lines that induce the EW structure on the space $W$, obtained from a 
toric ALE gravitational instanton of type $A_{2n-1}$ by reduction with respect 
to the scalar $S^1$-action.
The hyperplane sections of $\ms T$ parameterized by $\ptl\hat W' \simeq S^2$ 
consist precisely of the minitwistor lines corresponding to points of the conformal 
infinity of the EW space $W$.
\end{theorem}

\subsection{Non-real singularities of minitwistor lines}\label{ss:sing}
Finally, we discuss the reality of the nodes of minitwistor lines.
For each minitwitor line, the set of nodes is clearly real as a whole, but the reality of each node is a non-trivial issue.
Regarding this problem, based on the investigations we have conducted so far, we see that if $g>1$ then $\ms T$ has minitwistor lines having non-nodal singularities, as touched immediately after Proposition \ref{p:inn}:

\begin{proposition}\label{p:re}
If $g>1$, then there exists an $S^1$-orbit on the smooth EW space $W\minus\bm A$ such that the minitwistor lines in $\ms T$ corresponding to points on the orbit have at least one $\sigma$-conjugate pair of non-real nodes.
If $g=1$, then for any point of the EW space $W\minus\bm A$, the corresponding minitwistor line in $\ms T$ has exactly one singularity and it is a real ordinary node.
\end{proposition}

\proof
We recall that in \eqref{pen1} using the ramification points, we defined the following two divisors of degree $n$ on $\Sigma$:
\begin{align}\label{pen2}
D_{\rm L}= \sum_{i=1}^n r_i
\qandq 
D_{\rm R}= \sum_{i=n+1}^{2n} r_i.
\end{align}
If $\ul h_{\rm L}$ (resp.\,$\ul h_{\rm R}$) is the hyperplane in $\PP^n$ spanned by the $n$ points $a_1,\dots, a_n\in\Lmd$ (resp.\,$a_{n+1},\dots,a_{2n}\in\Lmd$), and further put $h_{\rm L}=\pi\inv(\ul h_{\rm L})$ (resp.\,$h_{\rm R}=\pi\inv(\ul h_{\rm R})$), then we have $h_{\rm L}|_{\Sigma} = 2D_{\rm L}$ and $h_{\rm R}|_{\Sigma} = 2D_{\rm R}$.
By Proposition \ref{p:itl}, the hyperplanes $H_{\rm L}=\Pi\inv(h_{\rm L})$ and $H_{\rm R}=\Pi\inv(h_{\rm R})$ cut out the meromorphic images $\tilde\Psi(L_n)$ and $\tilde\Psi(L_{n+1})$ respectively from $\ms T\subset\PP^{n+2}$, where $L_n$ and $L_{n+1}$ are the twistor lines through $C_n\cap C_{n-1}$ and $C_n\cap C_{n+1}$ respectively.
As points of $\hat W'$, they are exactly the two corners of $\Sigma''$ which belong to $\ptl\hat W'$, or equivalently, the two ends of the axis $\bm A'\simeq\bm A$ of rotation.
The points $r_1,\dots, r_n$ 
(resp.\,$r_{n+1},\dots,r_{2n}$)
are ordinary nodes of the cut $H_{\rm L}\cap \ms T$
(resp.\,$H_{\rm R}\cap \ms T$), and all of them are real.

Take a continuous path in the interior of the quarter $\Sigma''$ which connects the two corners $H_{\rm L}$ and $H_{\rm R}$ of $\Sigma''$. From Proposition \ref{p:ext}, this determines a family of minitwistor lines continuously varying from 
$H_{\rm L}\cap \ms T$ to $H_{\rm R}\cap \ms T$ and since we are avoiding the boundary $\ptl\Sigma''$, none of these minitwistor lines pass through the singularities of $\ms T$ except for the initial and the end ones.
By the equation \eqref{hqS}, if a point on the path is sufficiently close to $H_{\rm L}$
(resp.\,$H_{\rm R}$), then the node $r_n$ (resp.\,$r_{n+1}$) of $H_{\rm L}\cap \ms T$
(resp.\,$H_{\rm R}\cap \ms T$)
will be smoothed out as it belongs to the central circle, and the singularities of the minitwistor line have to be the moves of the residual nodes $r_1,\dots,r_{n-1}$ (resp.\,$r_{n+2},\dots,r_{2n}$). By continuity, all of them are distinct, real and ordinary nodes.
If $g=n-1$ is even, then $r_1$ and $r_{2n}$ belong to distinct real circles in $\Sigma$ and there exist no pair $r_i$ and $r_j$ of points such that $i\le n$, $j\ge n+2$, and $r_i$ and $r_j$ belong to the same real circle. 
Again by continuity, this transition can happen only when the nodes vary in the following manner:
\begin{itemize}\setlength{\itemsep}{-2pt}
\item
[(1)]
A pair of nodes lying on the same real circle collide to be a single real point of multiplicity two.
\item
[(2)]
Next, they again separate to become a pair of $\sigma$-conjugate points.
\item
[(3)] 
Next, the pair of $\sigma$-conjugate points will again collide to be a single real point of multiplicity two.
\item
[(4)] 
Finally, the double point again separates to become a pair of real points on the same real circle.
\end{itemize}
(See Figure \ref{f:Rsurf2} for this process in the case $g=2$.)
The pair of $\sigma$-conjugate nodes arises in step (2).
Obviously, such a pair of nodes occurs along an $S^1$-orbit in $W\simeq W'$.

If $g=n-1$ is odd and greater than one, then although $r_1$ and $r_{2n}$ belong to the same real circle $\Sigma\us_0$, the same transition has to occur for the residual points $r_2,r_3,\dots,r_{n-1}$
and $r_{n+2},\dots, r_{2n-1}$.

On the other hand, if $g=1$, then except for the initial curve $H_{\rm L}\cap \ms T$ and the final curve $H_{\rm R}\cap \ms T$, all minitwistor lines along the path have only one node and by continuity they always have to belong to the central real sphere $\ms T\us_0$. 
Hence, the singularity of all minitwistor lines has to be a real node.
\proofend
 
\medskip
Thus, if $g>1$, the situation is quite different from the Lorentzian case studied in \cite{H25}.
In that case, as long as the corresponding point of the EW space lies off the axis of rotation (i.e., as long as it is a regular minitwistor line in the terminology of \cite{H25}), a minitwistor line has exactly $g$ ordinary nodes, all of which are real.
The main reason for this difference is that, in the Lorentzian case, the nodes lie on mutually distinct real spheres, so that the above type of collision can never occur.

In the above proof, we only use the two corners $2D_{\rm L}$ and $2D_{\rm R}$ of $\Sigma''$.
But a little more information on the distribution of minitwistor lines having non-real or non-nodal singularities can be derived if we consider the central semicircle of $\ptl\Sigma''$ as follows.

The degree $2n$ divisors on $\Sigma$ that correspond to points on the central semicircle $(\Sigma\us_n)'\subset\ptl\Sigma''$ are twice of divisors in (a half of) the real pencil $|D_{\rm L}|\us = |D_{\rm R}|\us$, and from Proposition \ref{p:rp}, any such a divisor has a unique point $q$ that belongs to the central semicircle.
If a point $q\in\Sigma''\minus\ptl\Sigma''$ is sufficiently close to the central semicircle, then the singularities of a minitwistor line determined by the point $q$ are over points of the divisor $D'_q$, where $h_q|_{\Sigma} = q + \ol q + 2D'_q$ as before.
Hence, information about members of the real pencil $|D_{\rm L}|\us = |D_{\rm R}|\us$ gives that of the singularities of a minitwistor line whose corresponding point of the EW space $W$ is sufficiently close to the conformal infinity.

Such information can be derived from Proposition \ref{p:n1} below.
Recall that the pencil $|D_{\rm L}|$ is base point free.
So it induces a surjective holomorphic map $\psi:\Sigma\lras\PP^1$
of degree $n$.
Let $s_{\rm L}$ and $s_{\rm R}$ denote the points
of $\PP^1$ that satisfy
$$
\psi\inv(s_{\rm L}) = D_{\rm L}
\qandq
\psi\inv(s_{\rm R}) = D_{\rm R}.
$$
These are real points. 
The target $\PP^1$ of $\psi$ has a natural real structure which is the complex conjugation.

\begin{proposition}\label{p:n1}
The $n:1$ map $\psi:\Sigma\lras\PP^1$ has the following properties.
\begin{itemize}\setlength{\itemsep}{-2pt}
\item
If $\Sigma\us_i$ is a real circle (i.e.,\,if $i\equiv n\,(2)$),
then $\psi(\Sigma\us_i)\subset\RP^1$ and
if $\Sigma_i$ is a pure imaginary circle (i.e.,\,if $i\not\equiv n\,(2)$), then $\psi(\Sigma_i)\subset \sqrt{-1}\RP^1:=\{(1:u)\in\PP^1\set u\in i\RR\}\cup\{(0:1)\}$.
\item When $g$ is odd, the two real circles $\Sigma\us_0$ and $\Sigma\us_n$ (= the central one) are mapped isomorphically onto the real circle $\RP^1\subset\PP^1$, while the remaining real circles $\Sigma\us_i$ are mapped doubly to closed intervals in $\RP^1$ containing the point $s_{\rm L}$ if $i<g$
and closed intervals in $\RP^1$ containing the point $s_{\rm R}$ if $i> g+2$.
\item
When $g$ is even, only the central circle $\Sigma\us_n$ is mapped isomorphically onto $\RP^1\subset\PP^1$, while the remaining real circles $\Sigma\us_i$ are mapped doubly to closed intervals in $\RP^1$ containing the point $s_{\rm L}$ if $i< g$
and closed intervals containing the point $s_{\rm R}$ if $i>g+2$.
\end{itemize}

\end{proposition}

\proof
We take a holomorphic coordinate on $\PP^1$ which places the two points $s_{\rm L}$ and $s_{\rm R}$ at $0$ and $\infty$ respectively.
Then since $\psi$ is of degree $n$, both $s_{\rm L}$ and $s_{\rm R}$ are regular values of $\psi$ and $\psi$ can be regarded as a rational function on $\Sigma$ which has $D_{\rm R}$ and $D_{\rm L}$ as the zero divisor and the pole divisor respectively.

In the previous coordinates $(v,z)$ on $\ms O(2)$ (see \eqref{Sgm}),
noting the relation 
$$
\frac{\prod_{i=1}^n (z-a_i)}v = \frac v{\prod_{i=n+1}^{2n} (z-a_i)},
$$ 
which holds on $\Sigma$, we readily see that the function 
$$
g := \frac{\prod_{i=1}^n (z-a_i)}v\quad{\text{on $\Sigma$}}
$$
is a rational function which has simple zeroes at $z=a_1,\dots,a_n$ and simple poles at $z=a_{n+1},\dots,a_{2n}$
and has no other zeroes nor poles.
Therefore, $\psi= g$.
All the present assertions can be obtained using this explicit presentation of $g$ in an elementary way, and we omit the details.
(See the proof of Proposition \ref{p:tl} where the same function appears.)
\proofend

\begin{remark}
{\em
Topologically, $\psi$ can be regarded as the stereographic projection from the ``center'' of $\Sigma$ onto a sphere enclosing $\Sigma$.}
\end{remark}

By Proposition \ref{p:n1}, if $g=2$ for example, then that half of the real pencil has the following structure. 
As in Section \ref{s:mtl}, write  $|D_{\rm L}|\us=\{D_t\set t\in S^1\}$ and take the parameter $t$ in a way that $D_0 = D_{\rm L}$ and $D_{\infty} = D_{\rm R}$ hold.
Then the half of the real pencil can be written $\{D_t\set 0\le t\le \infty\}$.
Proposition \ref{p:n1} means that there exist positive real numbers $t_1<t_2$ and two points $P_1\in \Sigma\us_1$ and $P_2\in\Sigma\us_5$  such that $D_{t_i} = 2P_i$ for $i=1,2$ and the following hold. (See Figure \ref{f:Rsurf2}.)

\begin{figure}
\includegraphics{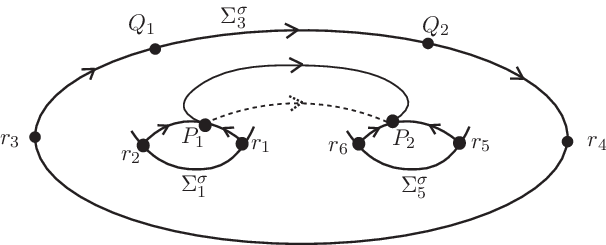}
\caption{How non-real nodes of a minitwistor line arise.}
\label{f:Rsurf2}
\end{figure}

\begin{itemize}\setlength{\itemsep}{-2pt}

\item 
As $t$ increases from $0$ to $t_1$, the divisor $D_0 = r_1 + r_2 + r_3$ evolves so that the two points $r_1$ and $r_2$ approach the point $P_1$ from opposite directions, while $r_3$ moves along $\Sigma\us_3$ in one direction. 
At $t = t_1$, it takes the form $Q_1 + 2P_1$ for some point $Q_1 \in \Sigma\us_3$.

\item
As $t$ increases from $t_1$ to $t_2$, the divisor $D_{t_1} = Q_1 + 2P_1$ changes in such a way that the component $2P_1$ splits into a divisor of the form $p_t + \overline{p}_t$ with $p_t \neq \overline{p}_t$, while $Q_1$ continues to move along $\Sigma\us_3$ in the same direction. 
At $t = t_2$, the divisor $D_t$ becomes $Q_2 + 2P_2$ for some point $Q_2 \in \Sigma\us_3$.

\item
As $t$ increases from $t_2$ to $\infty$, the divisor $D_{t_2} = Q_2 + 2P_2$ changes so that $2P_2$ splits into two distinct points on $\Sigma\us_5$, which move toward $r_5$ and $r_6$, respectively, in opposite directions, while $Q_2$ continues to move along $\Sigma\us_3$ toward $r_4$.

\end{itemize}
Hence, in a neighborhood of the conformal infinity of the EW space 
$W$, the singularities of a minitwistor line are of one of the following types
and all of them really occur:
\begin{itemize}\setlength{\itemsep}{-2pt}
\item  two real ordinary nodes lying on the real sphere $\ms T\us_1$,
\item  one tacnode lying on the real sphere $\ms T\us_1$,
\item  two ordinary nodes that are mutually $\sigma$-conjugate,
\item  one tacnode lying on the real sphere $\ms T\us_5$,
\item  two real ordinary nodes lying on the real sphere $\ms T\us_5$.
\end{itemize}
Even if $g>2$, similar information on the singularities of minitwistor lines can be derived from Proposition~\ref{p:n1}, again at least in a neighborhood of the conformal infinity. However, determining the singularities for arbitrary points of the EW space appears not to be easy.

\end{document}